\documentclass[10pt,reqno,a4paper]{article} 

\usepackage{anysize}
\marginsize{3.0cm}{3.0cm}{2.2cm}{2.2cm}

\usepackage[english]{babel}
\usepackage[centertags]{amsmath}
\usepackage{amsfonts,amssymb,amsthm}
\usepackage{color}
\usepackage{hyperref} 
\usepackage{mathrsfs}
\usepackage[sans]{dsfont}

\let\OLDthebibliography\thebibliography
\renewcommand\thebibliography[1]{
  \OLDthebibliography{#1}
  \setlength{\parskip}{0pt}
  \setlength{\itemsep}{0pt plus 0.3ex} }

\numberwithin{equation}{section}

\theoremstyle{plain}
\newtheorem{theorem}{Theorem}[section]
\newtheorem{proposition}[theorem]{Proposition}
\newtheorem{lemma}[theorem]{Lemma}

\newtheorem{prop}[theorem]{Proposition}

\newtheorem{definition}[theorem]{Definition}

\theoremstyle{definition}
\newtheorem{remarkbase}[theorem]{Remark}
\newenvironment{remark}{\pushQED{\qed} \remarkbase}{\popQED\endremarkbase}
\newtheorem{exbase}[theorem]{Example}
\newenvironment{ex}{\pushQED{\qed} \exbase}{\popQED\endexbase}

\newcommand{\N}{{\mathbb N}}
\newcommand{\R}{{\mathbb R}}
\newcommand{\C}{{\mathbb C}}
\newcommand{\Z}{\mathbb Z}

\newcommand{\T}{{\mathbb T}}

\newcommand{\mE}{\mathcal{E}}
\newcommand{\mF}{\mathcal{F}}

\newcommand{\mH}{\mathcal{H}}
\newcommand{\mI}{\mathcal{I}}

\newcommand{\mN}{\mathcal{N}}

\newcommand{\mS}{\mathcal{S}}

\renewcommand{\a}{\alpha}
\renewcommand{\b}{\beta}
\newcommand{\g}{\gamma}
\renewcommand{\d}{\delta}

\newcommand{\e}{\varepsilon}
\newcommand{\ph}{\varphi}
\newcommand{\lm}{\lambda}

\newcommand{\Om}{\Omega}
\newcommand{\om}{\omega}

\newcommand{\s}{\sigma}
\renewcommand{\t}{\tau}
\renewcommand{\th}{\vartheta}

\newcommand{\la}{\langle}
\newcommand{\ra}{\rangle}
\newcommand{\pa}{\partial}

\DeclareMathOperator{\Lip}{Lip}

\renewcommand{\oe}{\"o}
\newcommand{\rhost}{\rho}
\DeclareMathOperator{\diam}{diam}
\DeclareMathOperator{\dist}{dist}
\newcommand{\eqreff}[1]{(\ref{#1})}
\newcommand{\EE}{{\mathbb{E}}}
\newcommand{\sn}{A} 

\title{A Whitney extension theorem for functions \\ 
taking values in scales of Banach spaces}
 
\author{\normalsize{Pietro Baldi}} 
\date{} 
\begin{document}

\maketitle

\begin{small}
\noindent
\textbf{Abstract.}
We introduce a modified version of the Whitney extension operators
for collections of functions from a closed subset of $\R^n$ 
into scales of Banach spaces with smoothing operators. 
We prove an extension theorem for collections 
whose elements take values in different spaces of the scale.
A motivation for considering this kind of collections 
comes from very basic observations on the composition of functions 
of more than one real variable. 
The idea at the base of the proof is 
rather natural in the context of scales of Banach spaces, 
and consists in introducing smoothing operators 
in the construction of the extension, 
with smoothing parameters related to the 
diameter of each Whitney dyadic cube. 
Classical examples of scales of Banach spaces with smoothing operators
are also given, and some new related observations are proved.

\end{small}

\begin{small}
\tableofcontents 
\end{small}

\section{Introduction}
\label{sec:intro}

In this paper we show how 
the construction of the standard Whitney extension operators 
can be modified in order to deal with scales $(E_a)_a$ of Banach spaces, 
and, in particular, with collections $\{ f^{(j)} : |j| \leq k \}$ 
whose elements are functions $f^{(j)} : F \to E_{a_j}$, 
defined on a closed subset $F$ of $\R^n$, 
taking values into different spaces $E_{a_j}$ of the scale. 

The concrete example we have in mind 
regards functions $u(x,\th)$ where $x$ varies in a closed set $F \subset \R^n$, 
and, for each $x \in F$, the function $u(x) = u(x, \cdot) : \th \mapsto u(x,\th)$ 
belongs to some function space, which is, e.g., 
the Sobolev space $H^s(\R^d,\C)$, 
or the space $C^m(\T^d,\C)$ of periodic, continuously differentiable functions
of the variable $\th \in \R^d$, 
or the space of functions with decay $O(|\th|^{-m})$ as $|\th| \to \infty$, 
etc.; 
moreover, $u$ is differentiable \emph{\`a la} Whitney with respect to $x$ on $F$, 
but every time we differentiate $u$ with respect to $x$, 
the regularity (or the decay, etc.) of $u$ with respect to $\th$ 
deteriorates of a fixed amount $\d > 0$,
so that, if, say, $u(x) = u(x, \cdot) \in H^s(\R^d)$, 
then the first (Whitney) partial derivatives $\pa_{x_i} u(x, \cdot)$ 
belong to $H^{s-\d}(\R^d)$, 
the second derivatives $\pa_{x_i x_j} u(x, \cdot)$ 
belong to $H^{s-2\d}(\R^d)$,
and so on. 
We will observe below how similar situations emerge naturally 
from the analysis of the composition of functions.

Although motivated by these concrete examples, 
our construction is entirely abstract, in the sense that 
it applies to any scale of Banach spaces equipped with smoothing operators. 

Before introducing our main result (which is Theorem \ref{thm:WET sigma}), 
explaining in more details its motivation 
and what happens when trying to adapt directly the classical proof to our case, 
we make a step back and start with recalling 
the classical Whitney extension theorem, following Stein's book \cite{S}.

\subsection{Classical Whitney extension theorem} 
\label{subsec:classical WET}

Let $\N := \{0,1,\ldots \}$ denote the set of non-negative integers. 
Recall the standard multi-index notation: 
for $j = (j_1, \ldots, j_n) \in \N^n$ 
and $x = (x_1, \ldots, x_n) \in \R^n$, $n \geq 1$, 
we denote 
\begin{equation} \label{multi-index notation}
\begin{aligned}
j! & = j_1! \, j_2! \, \cdots j_n!, \quad
& 
x^j & = x_1^{j_1} \, x_2^{j_2} \, \cdots x_n^{j_n}, 
\\
|j| & = j_1 + \ldots + j_n, 
& 
\pa_x^j & 
= \pa_{x_1}^{j_1} \, \pa_{x_2}^{j_2} \, \cdots \pa_{x_n}^{j_n}.
\end{aligned}
\end{equation}

\begin{definition} \emph{(The space $\Lip(\rhost,F)$, from \cite{S})}.
\label{def Lip Stein}
Let $F \subseteq \R^n$ be a closed set, 
$k \geq 0$ an integer, 
and $k < \rhost \leq k+1$. 
We say that a collection 
$\{ f^{(j)} : j \in \N^n, \ 0 \leq |j| \leq k \}$ 
of real functions $f^{(j)} : F \to \R$ 
belongs to $\Lip(\rhost,F)$ 
if there exists a constant $M \geq 0$ 
such that the functions 
$f^{(j)}$ and the remainders $R_j$ defined by Taylor's formula
\begin{equation} \label{2502.1}
f^{(j)}(x) = 
\sum_{\begin{subarray}{c} 
\ell \in \N^n \\ 
|j+\ell| \leq k
\end{subarray}} 
\frac{1}{\ell!}\, f^{(j+\ell)}(y) (x-y)^\ell
+ R_j(x,y)
\end{equation}
satisfy
\begin{equation}
\label{0408.1}
|f^{(j)}(x)| \leq M, 
\quad \  
|R_j(x,y)| \leq M |x-y|^{\rhost-|j|}
\quad \ \forall x,y \in F, \ |j| \leq k.
\end{equation}
The norm of an element 
$f = \{ f^{(j)} : 0 \leq |j| \leq k \}$ 
of the space $\Lip(\rhost,F)$ 
is defined as 
\begin{equation} \label{0203.1}
\| f \|_{\Lip(\rhost,F)} 
:= \inf \{ M \geq 0 : \eqreff{0408.1} \text{ holds}\, \}.
\end{equation}
\end{definition}

On a closed set $F$, in general, the functions $f^{(j)}$, $1 \leq |j| \leq k$, 
are not uniquely determined by $f^{(0)}$;
on the other hand, if $F = \R^n$, then $f^{(0)}$ determines $f^{(j)}$  
as its partial derivatives.

\begin{theorem} \emph{(Whitney extension theorem, from \cite{S})}.
\label{thm:WET Stein}
Let $F, k, \rhost$ be like in Definition \ref{def Lip Stein}. 
There exists a linear continuous extension map 
\[
\mE_k : \Lip(\rhost,F) \to \Lip(\rhost,\R^n)
\]
which to every $f \in \Lip(\rhost,F)$ 
gives an extension $\mE_k f$ to all of $\R^n$,  
with bound
\[
\| \mE_k f \|_{\Lip(\rhost,\R^n)} 
\leq C \| f \|_{\Lip(\rhost,F)}
\]
where $C$ is independent of $F$. 
\end{theorem}

As pointed out in \cite{S}, 
Theorem \ref{thm:WET Stein} also holds for
more general modulus of continuity $\om(|x-y|)$ in place of $|x-y|^{\rhost-k}$.

\subsection{Whitney extension theorem for a fixed Banach space} 
\label{subsec:fixed.Banach.space}

Whitney’s extension problems go back 
to the 1934 works \cite{W-anal}, \cite{W-diff} of Whitney.
Since then, they have been studied by several authors and 
generalized to any set $F \subset \R^n$ (not only a closed one), 
to extensions that are $C^m(\R^n,\R)$, or $C^{m,\om}(\R^n,\R)$ (i.e.\ functions 
whose $m$th derivative has modulus of continuity $\om$), 
or Sobolev $W^{m,p}(\R^n,\R)$, 
also with accurate quantitative estimates 
for polynomial approximations: 
the reference literature certainly includes the recent series of works of 
Fefferman, also with Klartag  
(see e.g.\ \cite{F-Annals-2005}, \cite{F-Annals-2006}, \cite{FK}, \cite{FK-2}), 
Bierstone, Milman, Paw\l{}ucki 
(e.g.\ \cite{BMP-Inventiones}, \cite{BMP-Annals}), 
Shvartsman (e.g.\ \cite{Shv});
see the expository paper \cite{F-expo} 
for a richer list of references 
and for an overview of some recent results on the topic. 

In this paper, as said above, we are interested 
to generalizing Whitney's results 
in another direction, which is the one 
of extending functions defined on $F \subset \R^n$
taking values on scales of Banach spaces. 

If, instead of a scale, we consider \emph{one} Banach space $Y$, 
then the validity of Whitney extension theorem is well-known:
see e.g.\ Federer \cite{Fed}, page 225. 
In particular, consider the following, 
obvious adaptation of Definition \ref{def Lip Stein}. 

\begin{definition} \emph{(The space $\Lip(\rhost,F,Y)$).}
\label{def Lip Y}
Let $F \subseteq \R^n$ be a closed set, 
$k \geq 0$ an integer, 
$k < \rhost \leq k+1$, 
and let $Y$ be a (real or complex, finite or infinite-dimensional) Banach space.  
We say that a collection 
$\{ f^{(j)} : j \in \N^n, \ 0 \leq |j| \leq k \}$ 
of functions $f^{(j)} : F \to Y$ 
belongs to $\Lip(\rhost,F,Y)$ 
if there exists a constant $M \geq 0$ 
such that the functions 
$f^{(j)}$ and the remainders $R_j$ 
defined by Taylor's formula \eqref{2502.1} satisfy
\begin{equation}
\label{2502.Y}
\| f^{(j)}(x) \|_Y \leq M, 
\quad \  
\| R_j(x,y) \|_Y \leq M |x-y|^{\rhost-|j|},
\quad \ 
\forall x,y \in F, \ |j| \leq k.
\end{equation}
The norm of an element 
$f = \{ f^{(j)} : 0 \leq |j| \leq k \}$ 
of the space $\Lip(\rhost,F,Y)$ 
is defined as 
\begin{equation} \label{0203.1.Y}
\| f \|_{\Lip(\rhost,F,Y)} 
= \inf \{ M \geq 0 : \eqreff{2502.Y} \text{ holds}\, \}.
\end{equation}
\end{definition}

As observed in Appendix B of \cite{BBHM}, 
the proof of Theorem \ref{thm:WET Stein} in Stein's book \cite{S} 
holds {with no change} for the space $\Lip(\rhost,F,Y)$,
and the extension operator $\mE_k$ satisfies 
\begin{equation} \label{ext.Y}
\| \mE_k f \|_{\Lip(\rhost,\R^n,Y)} 
\leq C \| f \|_{\Lip(\rhost,F,Y)}
\end{equation}
with $C$ independent on the closed set $F$ 
and on the Banach space $Y$.

\subsection{What is the point with the composition of functions}
\label{subsec:what.comp}

The extension theorem for functions $f^{(j)} : F \to Y$ 
taking values in a fixed Banach space $Y$ 
(see Definition \ref{def Lip Y} and \eqref{ext.Y})
holds, in particular, when $Y$ is a function space, 
which could be, for example, a Sobolev space, a $C^m$ space, a H\oe{}lder space, 
etc. 

Suppose that a function space $Y$ contains not only the sum $u+v$, 
but also the pointwise product $uv$ and the composition $u \circ v$ 
of any two functions $u,v \in Y$.
The composition operation then behaves differently from the operations 
of sum and product when introducing the dependence on a ``parameter'' 
$x \in F \subset \R^n$, 
namely when passing from considering elements $u,v$ of $Y$ 
to considering functions $F \to Y$, i.e.\ elements $u(x), v(x)$ of $Y$ 
depending on $x \in F \subset \R^n$. 
The central point is that the norm of $\Lip(\rho,F,Y)$ 
is not the most adapt one 
when dealing with the composition of functions. 
To show it, we are going to consider some elementary examples. 
These observations come before the question about the possibility 
of extending an element of $\Lip(\rho,F,Y)$ to one in $\Lip(\rho,\R^n,Y)$; 
they rather involve the passage from elements of $Y$ 
to functions of $x \in F$ taking values in $Y$. 
To not mix these two aspects, 
we consider $F=\R$ in the following examples.

\medskip

\begin{ex} \label{ex:comp.1} 
Let $Y = C_b(\R,\R)$  
be the set of bounded, continuous functions $u : \R \to \R$,
equipped with the sup-norm 
$\| u \|_Y = \| u \|_\infty
= \sup \{ |u(\th)| : \th \in \R \}$. 
The sum $u+v$, 
the product $uv$
and the composition $u \circ v$
all belong to $Y$ for all $u,v \in Y$.

Now we introduce the dependence on a real variable $x$ 
in the role of a parameter, and  
we use Definition \ref{def Lip Y} to describe 
the regularity with respect to the parameter $x$. 
Let $k=0$, $\rhost = 1$, $F = \R$. 
Hence $\Lip(\rhost,F,Y) = \Lip(1,\R,Y)$
is the space of functions $f : \R \to Y$ such that 
\begin{equation} \label{1108.1}
\| f(x) \|_Y \leq M,
\quad \ 
\| R_0(x,y) \|_Y \leq M |x-y|,
\quad \ \forall x,y \in F,
\end{equation}
namely, denoting $f(x,\th) := f(x)(\th)$,  
\begin{align*}
| f(x,\th) | \leq M, 
\quad \ 
| f(x,\th) - f(y,\th) | \leq M |x-y|,
\quad \ 
\forall x,y,\th \in \R,
\end{align*}
for some constant $M > 0$. 
The sum $f+g$ and the product $fg$ are in $\Lip(1,\R,Y)$ 
for all $f,g \in \Lip(1,\R,Y)$, 
but, in general, the composition 
\begin{equation} \label{def.h.fg}
h(x) := f(x) \circ g(x), \quad \ 
h(x)(\th) = f(x) \big( g(x)(\th) \big), \quad  
\text{i.e.} \ 
h(x,\th) = f(x, g(x,\th)), 
\end{equation}
does not belong to $\Lip(1,\R,Y)$:
consider, for example, 
\[
f(x,\th) := \begin{cases} 
\sqrt{\th} & \text{for} \ 0 \leq \th \leq 1, \\
0 & \text{for} \ \th < 0, \\
1 & \text{for} \ \th > 1,
\end{cases}
\qquad \quad 
g(x,\th) := 
\begin{cases} 
x & \text{for} \ 0 \leq x \leq 1, \\
0 & \text{for} \ x < 0, \\
1 & \text{for} \ x > 1.
\end{cases}
\]
Both $f$ and $g$ belong to $\Lip(1,\R,Y)$, 
but $h(x,\th) = f(x, g(x,\th))$ does not, because
$h(x,\th) = \sqrt{x}$ for all $x \in [0,1]$. 
Thus $Y$ is closed for composition, 
namely $u \circ v \in Y$ for all $u,v \in Y$,
while $\Lip(1,\R,Y)$ is \emph{not} closed 
for the operation of composition of elements of $Y$. 
\end{ex}

The same issue with the composition operation 
is also present if, 
unlike in Example \ref{ex:comp.1}, 
the amount of derivatives (or incremental ratios) with respect to $x$ 
is not larger than the one with respect to $\th$. 
The next example shows this.

\begin{ex} \label{ex:comp.2} 
Let $Y = C^1_b(\R,\R)$ be the set of $C^1$ functions $u : \R \to \R$
with finite norm 
\[
\| u \|_Y := \| u \|_{C^1} := \max_{m=0,1} \, \sup_{\th \in \R} |\pa_\th^m u(\th)| 
\, < \infty.
\]
For all $u,v \in Y$, the composition $u \circ v$ belongs to $Y$. 
As in Example \ref{ex:comp.1}, 
let $k=0$, $\rho = 1$, $F = \R$. 
Hence $\Lip(\rho,F,Y) = \Lip(1,\R,Y)$
is the space of functions $f : \R \to Y$ 
satisfying \eqref{1108.1} for some $M > 0$, 
namely, denoting $f(x,\th) := f(x)(\th)$,  
\begin{alignat*}{2}
|f(x,\th)| & \leq M, 
\quad \ & 
|f(x,\th) - f(y,\th)| & \leq M |x-y|, 
\\
|\pa_\th f(x,\th)| & \leq M , 
\qquad  & 
|\pa_\th f(x,\th) - \pa_\th f(y,\th)| & \leq M |x-y|,
\end{alignat*}
for all $x,y,\th \in \R$.
Now fix a $C^\infty$ function $\psi : \R \to \R$ 
such that $\psi(t) = 1$ for all $|t| \leq 1$ 
and $\psi(t) = 0$ for all $|t| \geq 2$. 
For all $x,\th \in \R$, let 
\[
f(x,\th) := |\th|^{\frac32} \psi(\th),
\qquad 
g(x,\th) := (x+\th) \psi(x+\th).
\]
Thus $f,g \in \Lip(1,\R,Y)$. 
Let $h(x) := f(x) \circ g(x)$, like in \eqref{def.h.fg}. 
For $x,\th$ in the strip $|x+\th| \leq 1$ 
one has $h(x,\th) = |x+\th|^{3/2}$ 
and $\pa_\th h(x,\th) = \frac32 |x+\th|^{1/2} \mathrm{sign}(x+\th)$.
In particular, $\pa_\th h(x,0) = \frac32 \sqrt{x}$ for all $x \in [0,1]$.
Hence the ratio $|\pa_\th h(x,\th) - \pa_\th h(y,\th)| / |x-y|$ 
with $\th=0$, $y=0$, $x \in (0,1]$ has no finite upper bound; 
as a consequence, $h$ does not belong to $\Lip(1,\R,Y)$.  
This shows that, like in Example \ref{ex:comp.1}, 
$Y$ is closed for composition, 
but $\Lip(1,\R,Y)$ is not. 
\end{ex}

The issue with the composition operation is even more evident if 
we replace the Whitney regularity with respect to $x$ of Definition \ref{def Lip Y}
with the standard $C^k$ regularity of continuously differentiable functions;
this means that the Lipschitz continuity in $x$ 
of Examples \ref{ex:comp.1}-\ref{ex:comp.2} is replaced 
by the $C^1$ regularity with respect to $x$.  
Moreover the question can also be formulated 
in terms of bounds for norms and seminorms of $C^\infty$ functions.
Slightly informally, the next example shows this.

\begin{ex} \label{ex:comp.3} 
For functions $f \in C^\infty(\R^2, \R)$, 
for $k,m \in \N$, define
\[
\| f \|_{C^k_x C^m_\th}
:= \sup \big\{ \big| \pa_x^j \pa_\th^\ell f(x,\th) \big| : 
j \leq k, \ \, \ell \leq m, \ (x,\th) \in \R^2 \big\}.
\]
For $f, g \in C^\infty(\R^2,\R)$,  
the function $h(x,\th) := f(x, g(x,\th))$ 
has partial derivative 
\[
\pa_x h(x,\th) 
= (\pa_x f)(x, g(x,\th)) + (\pa_\th f)(x, g(x,\th)) \, \pa_x g(x,\th),
\]
and one has the natural bound
\begin{equation} \label{0803.1}
\| h \|_{C^1_x C^0_\th} 
\leq \| f \|_{C^1_x C^0_\th} + \| f \|_{C^0_x C^1_\th} \| g \|_{C^1_x C^0_\th}.
\end{equation}
If we are constrained to use the same regularity in $\th$ 
both for $f$ and for $\pa_x f$, 
then we have to bound both
$\| f \|_{C^1_x C^0_\th}$ and $\| f \|_{C^0_x C^1_\th}$ 
with $\| f \|_{C^1_x C^1_\th}$,
so that \eqref{0803.1} is deteriorated to 
\begin{equation} \label{0803.2}
\| h \|_{C^1_x Y_0} 
\leq \| f \|_{C^1_x Y_1} (1 + \| g \|_{C^1_x Y_0})
\end{equation}
where, to emphasize the similarity with the norms in Definition \ref{def Lip Y},
we have denoted 
$\| \  \|_{C^1_x Y_0} := \| \  \|_{C^1_x C^0_\th}$
and 
$\| \  \|_{C^1_x Y_1} := \| \  \|_{C^1_x C^1_\th}$.
Bound \eqref{0803.2} contains an artificial ``loss of regularity'' 
in estimating $h$ in terms of $f,g$, 
which is merely due to considering
the same fixed regularity in $\th$ 
both for $f$ and for $\pa_x f$.
\end{ex}

There is an obvious way of eliminating 
the artificial ``loss of regularity'' of Example \ref{ex:comp.3},
which simply consists in not limiting ourselves to use 
only norms corresponding to Definition \ref{def Lip Y}, 
but allowing the use of more natural norms.  
The following example shows this.

\begin{ex} \label{ex:comp.4} 
Consider Example \ref{ex:comp.3}, and define 
$\| f \|_{1} := \max \{ \| f \|_{C^0_x C^1_\th}, \ \| f \|_{C^1_x C^0_\th} \}$.
Then, from \eqref{0803.1}, we obtain
\[
\| h \|_1 \leq \| f \|_1 (1 + \| g \|_1),
\]
which is a composition estimate similar to \eqref{0803.2} but without artificial losses.
\end{ex}

The elementary examples above show that, 
to avoid artificial losses in the composition estimates, 
it is natural to introduce a ``decreasing regularity'' version 
of the norms $\| \ \|_{C^k_x C^m_\th}$, 
where each $x$-derivative more 
is compensated by a $\th$-derivative less: 
\begin{equation} \label{def dr CC}
\| f \|_{C^k_x C^m_\th}' 
:= \max \big\{ \| f \|_{C^0_x C^m_\th} \,, \| f \|_{C^1_x C^{m-1}_\th} \,, 
\ldots, \| f \|_{C^k_x C^{m-k}_\th} \big\},
\quad \ m \geq k.
\end{equation}
In fact, this is similar to consider the joint regularity $C^m$ 
in the pair $(x,\th)$.  

This kind of ``decreasing regularity'' norms 
can be generalized by considering higher dimension $x \in \R^n$, $\th \in \R^d$, 
other kind of regularities (H\oe{}lder, Sobolev, etc.{}) 
in $\th$ and/or in $x$ instead of $C^m$, 
a more general balance between regularity in $x$ and in $\th$ 
(where one derivative in $x$ 
counts like $\d$ derivatives in $\th$), 
and so on.

``Decreasing regularity'' norms are very natural, and commonly used, 
for solutions of evolution PDEs: 
for example, the solution $u(t,x)$ of the Schr\oe{}\-din\-ger equation 
$\pa_t u + i \Delta u = 0$ on $\R^d$ 
with initial datum $u(0,x) = u_0(x)$ 
in the Sobolev space $H^s_x = H^s(\R^d)$ 
is $C^0_t H^s_x \cap C^1_t H^{s-2}_x$, 
where $(t,x)$ correspond to $(x,\th)$ above; 
here one derivative in $t$ counts like $2$ derivatives in $x$. 

Arising from different problems, 
``decreasing regularity'' norms have been recently introduced 
in the context of KAM theory in \cite{BM-Euler}
for Sobolev periodic functions $H^s(\T^d,\R^3)$, 
where the parameter $x$ is a frequency vector that varies in $\R^n$ 
and has its significance only when it is restricted 
to a ``Cantor-like'' closed subset of $\R^n$ 
of ``nonresonant frequencies'' (as is typical in KAM results).

\medskip

Motivated by the basic observations above about the composition of functions,
the goal of this paper is to construct Whitney extension operators,
similar to those in Theorem \ref{thm:WET Stein},
which are adapted to ``decreasing regularity'' norms like \eqref{def dr CC}, 
and which work not only for functions of $C^m$ regularity in $\th$, 
but also for those of H\oe{}lder or Sobolev regularity. 

Instead of proceeding by cases, we consider 
more general scales of Banach spaces equipped with smoothing operators 
(see subsection \ref{subsec:scales}), 
which constitute an abstract setting applicable 
to a variety of concrete cases.

\subsection{Why the classical proof does not work here} 
\label{subsec:nontrivial.proof}

The extension operators $\mE_k$ constructed in Stein's book \cite{S} 
do not preserve ``decreasing regularity'' norms like those in \eqref{def dr CC}: 
trying to follow the proof in \cite{S}, 
almost immediately one finds terms 
having less regularity in $\th$ than what is needed. 
In fact, given a collection $f = \{ f^{(j)} : |j| \leq k \}$ 
defined on a closed set $F \subset \R^n$, 
its extension at any point $x \in \Om := \R^n \setminus F$  
is defined in \cite{S} as a sum of the form 
\begin{equation} \label{farlocco}
\sum_{i} \sum_{|\ell| \leq k} 
\frac{1}{\ell!} f^{(\ell)}(p_i) (x-p_i)^\ell \ph_i^*(x)
\end{equation}
where $(\ph_i^*)$ is a partition of unity related to Whitney's decomposition  
(see Proposition \ref{prop:partition of unity} below),  
the sum $\sum_{i}$ is over cubes close to $F$, 
and $p_i$ are points in $F$. 
Now consider ``decreasing regularity'' norms: 
let $m \geq k$, and suppose, for example, that, for every $x \in F$, 
$f^{(j)}(x)$ is a $C^{m-|j|}$ function of $\th$.  
Then at $x \in \Om$ the sum \eqref{farlocco} 
is a $C^{m-k}$ function of $\th$
(like the less regular of its terms),
while the required regularity in $\th$ for the extension \eqref{farlocco}
would be the same as $f^{(0)}$, namely $C^m$. 

Also, trying to change naively the proof of \cite{S}, 
one encounters terms having less regularity in $\th$ 
and more smallness in $x$ (namely higher powers of $|x-y|$ 
where $x,y$ are two values of the parameter $x$), 
and other terms that, vice versa, 
are more regular than necessary but are not sufficiently small.

A classical, effective tool to ``convert regularity into size factors'' 
and vice versa is given by smoothing operators,
which we insert in the construction of the extension. 
In concrete function spaces, smoothing operators 
are usually obtained by convolution with mollifiers 
or other Fourier cut-offs (see section \ref{sec:list});
they have been used to deal with ``loss of regularity'' problems
at least since the work of Nash \cite{Nash}, 
and are a key ingredient in implicit function theorems 
for scales of Banach spaces, namely theorems of Nash-Moser type
(see e.g.\ \cite{Nash}, \cite{Zehnder}, \cite{Geodesy}, \cite{AG}, 
the recent versions in \cite{Berti-Bolle-Procesi}, \cite{ES}, 
and the sharp one in \cite{BH-NMH}).

In our construction we relate the smoothing parameter $\theta$ 
to the distance of points 
from the closed set $F$ using the diameter of each Whitney dyadic cube, 
see subsection \ref{subsec:def g} and, in particular, definition \eqref{def theta}. 
This is the key ingredient for adapting almost completely 
the proof of \cite{S} to the present case. 

There is, in addition, a point of the proof where another smoothing parameter 
has to be introduced; since the explanation requires more technical details, 
we postpone it to Remark \ref{rem:two decomp}.

\bigskip

\textbf{Acknowledgments.} 
We thanks 
L. Asselle, 
G. Benedetti, 
M. Berti, 
L. Franzoi, 
F. Giuliani, 
C. Guillarmou, 
E. Haus, 
I. Marcut, 
R. Montalto, 
F. Murgante, 
P. Ventura 
for interesting discussions 
related to the topics of this paper.
Supported by INdAM-GNAMPA Project 2019.

\section{An extension theorem for scales of Banach spaces}
\label{sec:2}

In subsection \ref{subsec:scales} we define 
the scales of Banach spaces with smoothing operators.
In subsection \ref{subsec:main result} 
we state our extension result (Theorem \ref{thm:WET sigma}).
In subsection \ref{subsec:corollary} we give an application if it.

\subsection{Scales of Banach spaces with smoothing operators}
\label{subsec:scales}

Let $a_0 \in \R$, 
and let $\mI \subseteq [a_0,\infty)$ be a set of real numbers 
with $a_0 = \min \mI$. 
Let $(E_a, \| \, \|_a)$, $a \in \mI$, 
be a decreasing family of (real or complex) Banach spaces 
with continuous inclusions
\begin{equation} \label{0303.1}
E_b \hookrightarrow E_a, \quad \ 
\| u \|_a \leq \| u \|_b \quad \text{for} \  a,b \in \mI, \ a \leq b.	
\end{equation}
Set 
\[
E_\infty := \bigcap_{a \in \mI} E_a.
\]
Since $E_\infty$ is the usual notation for the intersection of the spaces of the scale,
we write $E_\infty$ even in the case in which $\sup \mI < \infty$.

We assume that the scale $(E_a)$ is equipped with smoothing opertors, 
namely we assume that there exists a family $(S_\theta)$
with continuous parameter $\theta \in [1,\infty)$ of linear operators
\begin{equation} \label{S.generic}
S_\theta : E_{a_0} \to E_\infty 
\end{equation}
satisfying for all $\theta \geq 1$ 
the following two basic smoothing properties:
\begin{alignat}{2} \label{S1} 
\| S_\theta u \|_b 
& \leq A_{ab} \, \theta^{b-a} \| u \|_a 
& \ \quad & \text{if} \ a \leq b, \ \ a,b \in \mI,
\\
\label{S2} 
\| u - S_\theta u \|_a 
& \leq B_{ab} \, \theta^{-(b-a)} \| u \|_b 
&& \text{if} \ a \leq b, \ \ a,b \in \mI, 
\end{alignat}
for some constants $A_{ab}, B_{ab}$ that are bounded 
for $a,b$ in bounded subsets of $\mI$.
In particular, 
\[
\| S_\theta u \|_a \leq A_{aa} \| u \|_a, 
\quad \ 
\| (I - S_\theta) u \|_a \leq B_{aa} \| u \|_a.
\]
From \eqref{S1}-\eqref{S2} one obtains the logarithmic convexity of the norms,
namely the interpolation inequality
\begin{equation} \label{interpolazione}
\| u \|_b \leq C \| u \|_a^{\frac{c-b}{c-a}} \| u \|_c^{\frac{b-a}{c-a}} 
\quad \ \forall a,b,c \in \mI, \ \ a \leq b \leq c,
\quad \ C = 2 A_{ab}^{\frac{c-b}{c-a}} B_{bc}^{\frac{b-a}{c-a}}.
\end{equation}
To prove \eqref{interpolazione}, split $u = S_\theta u + (u - S_\theta u)$,
use \eqref{S1} to estimate $\| S_\theta u \|_b$ 
and \eqref{S2} to estimate $\| u - S_\theta u \|_b$, 
then optimize the choice of $\theta$. 
From \eqref{S1}-\eqref{S2} we also obtain
\begin{equation} \label{0303.4}
\| (S_{\theta_1} - S_{\theta_2}) u \|_b 
\leq C \max \{ \theta_1^{b-a}, \, \theta_2^{b-a} \} \, \| u \|_a
\quad \ \forall a,b \in \mI, \ \ \theta_1, \theta_2 \in [1,\infty),
\end{equation}
with 
\[
C = \max \{ B_{pq} \max \{ A_{pp}, A_{qq} \},  
A_{pq} \max \{ B_{pp}, B_{qq} \} \},
\quad 
p := \min \{ a,b \}, \ \ 
q := \max \{ a,b \}.  
\]
To prove \eqref{0303.4}, 
split $(S_{\theta_1} - S_{\theta_2}) u 
= S_{\theta_1} (I - S_{\theta_2}) u  
+ (S_{\theta_1} - I) S_{\theta_2} u$,
consider separately the two cases $a \leq b$ and $a > b$, 
and apply \eqref{S1}-\eqref{S2} to each term.  

Some examples of scales $(E_a)$ with smoothings $(S_\theta)$ 
are given in Section \ref{sec:list}.

\subsection{Main result}
\label{subsec:main result}

We begin with defining a more general version 
of the ``decreasing regularity'' norms \eqref{def dr CC}, 
based on the norms of the scale $(E_a)$. 

\begin{definition} \emph{(The space $\Lip(\rho,F,\sigma,\g, \delta)$).}
\label{def Lip sigma}
Let $a_0 \in \R$, $\mI \subseteq [a_0, \infty)$, $a_0 \in \mI$, 
and let $(E_a, \| \ \|_a)_{a \in \mI}$ 
be a scale of Banach spaces equipped with smoothing operators
as described in subsection \ref{subsec:scales}. 
Let 
\begin{equation} \label{param}
0 \leq k < \rho \leq k+1, \quad \ 
\gamma > 0, \quad \ 
\delta > 0,
\end{equation}
with $k$ integer, $\rho, \g, \delta$ real.
Let $F$ be a closed subset of $\R^n$, $n \geq 1$. 
Let $\sigma = \sigma_0, \ldots, \sigma_k, \sigma_\rho$ be real numbers, 
all belonging to $\mI$, 
with 
\begin{equation} \label{def sj}
\sigma_m := \sigma - m \delta 
\quad \forall m = 0, \ldots, k, 
\quad \ 
\sigma_\rho := \sigma - \rho \delta. 
\end{equation}
For every multi-index $j \in \N^n$ of length $|j| \leq k$, 
denote $\sigma_j := \sigma_{|j|} = \sigma - |j| \delta$. 
Consider a collection $f = \{ f^{(j)} : j \in \N^n, \ |j| \leq k \}$ 
of functions 
\[
f^{(j)} : F \to E_{\sigma_j}.
\]
For all $j \in \N^n$, $|j| \leq k$, 
all $x,y \in F$, let 
\begin{align} \label{def Pj Rj}
P_j(x,y) & := \sum_{\begin{subarray}{c} \ell \in \N^n \\ |j+\ell| \leq k \end{subarray}} 
\frac{1}{\ell!} f^{(j+\ell)}(y) (x-y)^\ell,
\quad \ 
R_j(x,y) := f^{(j)}(x) - P_j(x,y).
\end{align}
We say that the collection $f$ belongs to $\Lip(\rho, F, \sigma, \g, \d)$ 
if there exists a constant $M \geq 0$ such that 
\begin{equation}
\label{0903.1-2}
\g^{|j|} \| f^{(j)}(x) \|_{\sigma_j} \leq M,  
\quad \ 
\g^\rho \| R_j(x,y) \|_{\sigma_\rho} \leq M |x-y|^{\rho-|j|},
\quad \ 
\forall x,y \in F, \ \  |j| \leq k.
\end{equation}
The norm of $f$ is defined as 
\begin{equation} \label{0903.8}
\| f \|_{\Lip(\rho, F, \sigma, \g, \delta)} 
= \inf \{ M \geq 0 : (\ref{0903.1-2}) \text{ holds} \}.
\end{equation}
\end{definition}

Before stating the main result of the paper 
(which is Theorem \ref{thm:WET sigma} below), 
in the next proposition we observe that 
the family of spaces $\Lip(\rho,F,\sigma,\gamma,\delta)$ 
(where $\sigma$ is the only varying parameter) 
inherits the structure of scale of Banach spaces with smoothing operators 
from the scale $(E_a)$. 
The proof is elementary.

\begin{prop}
\label{prop:about.Lip}
\emph{(Inherited structure of the spaces $\Lip(\rho,F,\sigma,\gamma,\delta)$).}

$(i)$ Each space $\Lip(\rho,F,\sigma,\gamma,\delta)$ 
defined in Definition \ref{def Lip sigma} is a Banach space. 

$(ii)$ Let $(E_a)_{a \in \mI}$ be a scale of Banach spaces 
with smoothing operators $(S_\theta)_{\theta \geq 1}$ as described in subsection 
\ref{subsec:scales}. Let $k,\rho,\delta$ be like in Definition \ref{def Lip sigma}, 
and let $\mI' \subset \mI$ be the set of the ``admissible indices'',
namely the set of all $\s \in \mI$ 
such that $\s_0 = \s, \ldots, \s_k, \s_\rho$ (defined in Definition \ref{def Lip sigma})
all belong to $\mI$, so that $\Lip(\rho,F,\sigma,\gamma,\delta)$ is well defined 
for all $\sigma \in \mI'$. 
For $\sigma \in \mI'$, denote $X_\sigma := \Lip(\rho,F,\sigma,\gamma,\delta)$, 
and $X_\infty := \cap_{\sigma \in \mI'} X_\sigma$.
For $\sigma \in \mI'$, $\theta \geq 1$, 
 $f = \{ f^{(j)} : |j| \leq k \} \in X_\sigma$, 
define the collection 
$S_\theta f = \{ (S_\theta f)^{(j)} : |j| \leq k \}$ 
by setting $(S_\theta f)^{(j)}(x) := S_\theta[ f^{(j)}(x)]$ 
for all $x \in F$.  
Fix a number $b_0 \in \mI'$, 
and define $\mI'_0 := \mI' \cap [b_0, \infty)$. 
Then $(X_\sigma)_{\sigma \in \mI'_0}$ with $(S_\theta)_{\theta \geq 1}$ 
is a scale of Banach spaces with smoothing operators 
satisfying the properties described in subsection \ref{subsec:scales}.
\end{prop}

\begin{proof}
The proof of $(i)$ is the standard pointwise/uniform limit argument: 
let $f_1, f_2, \ldots$
be a Cauchy sequence of elements of $X_\sigma := \Lip(\rho,F,\sigma,\gamma,\delta)$,  
with $f_i = \{ f_i^{(j)} : |j| \leq k \}$. 
For all $x,y \in F$, $|j| \leq k$, all $i, i'$, 
the collection $f_i - f_{i'}$ satisfies
\[
(f_i - f_{i'})^{(j)}(x) = f_i^{(j)}(x) - f_{i'}^{(j)}(x), 
\quad \ 
R_j(x,y;f_i-f_{i'}) = R_j(x,y;f_i) - R_j(x,y;f_{i'}),
\]
where $R_j(x,y;f_i)$ is the Taylor remainder of the function $f_i^{(j)}$
as defined in \eqref{def Pj Rj}, and similarly for $R_j(x,y;f_{i'})$.  
Hence for all $x,y \in F$, $|j| \leq k$, all $i, i'$,
one has 
\begin{align} \label{2109.fj}
\g^{|j|} \| f_i^{(j)}(x) - f_{i'}^{(j)}(x) \|_{\s_j} 
& \leq \| f_i - f_{i'} \|_{X_\s}, 
\\ 
\g^\rho \| R_j(x,y;f_i) - R_j(x,y;f_{i'}) \|_{\s_\rho} 
& \leq \| f_i - f_{i'} \|_{X_\s} |x-y|^{\rho-|j|}. 
\label{2109.Rj}
\end{align}
From \eqref{2109.fj}, for every $x \in F$, $|j| \leq k$, 
the sequence $(f_i^{(j)}(x))_{i = 1,\ldots}$ 
is Cauchy in $E_{\s_j}$, therefore it converges to some limit in $E_{\s_j}$, 
which we denote by $f^{(j)}(x)$. Let $\e > 0$, and take $i_0 \in \N$ such that 
$\| f_i - f_{i'} \|_{X_\s} \leq \e$ for all $i,i' \geq i_0$. 
For any $i \geq i_0$, passing to the limit as $i' \to \infty$
in \eqref{2109.fj} and in \eqref{2109.Rj}, we find that 
\[
\g^{|j|} \| f_i^{(j)}(x) - f^{(j)}(x) \|_{\s_j} 
\leq \e, 
\quad \ 
\g^\rho \| R_j(x,y;f_i) - R_j(x,y;f) \|_{\s_\rho} 
\leq \e |x-y|^{\rho-|j|} 
\]
for all $x,y \in F$, $|j| \leq k$, all $i \geq i_0$. 
This implies that $f \in X_\s$ 
with $\| f_i - f \|_{X_\s} \leq \e$ for $i \geq i_0$, 
and this proves that $f_i \to f$ in $X_\s$ as $i \to \infty$. 

The proof of $(ii)$ is straightforward, as each property \eqref{0303.1}, 
\ldots, \eqref{S2} for $(X_\s)$ is deduced from the corresponding property 
for $(E_a)$, using the linearity of the smoothing operators $S_\theta$ 
and, in particular, the identity 
\[
R_j(x,y ; S_\theta f) = S_\theta[ R_j(x,y;f) ].
\]
Properties \eqref{interpolazione} (interpolation) 
and \eqref{0303.4} follow from \eqref{S1}, \eqref{S2},
as explained in subsection \ref{subsec:scales}.
\end{proof}

The following theorem is the main result of this paper.

\begin{theorem} \label{thm:WET sigma}
\emph{(Extension theorem for spaces $\Lip(\rho,F,\sigma,\gamma,\delta)$).}
Assume the hypotheses of Definition \eqref{def Lip sigma}. 
There exists a linear operator 
\[
\mE_{k} : \Lip(\rho,F,\sigma,\g,\d) \to \Lip(\rho,\R^n,\sigma,\g,\d)
\]
such that, 
given $f = \{ f^{(j)} : |j| \leq k \} \in \Lip(\rho,F,\sigma,\g,\d)$,  
the collection $g := \mE_k f = \{ g^{(j)} : |j| \leq k \}
\in \Lip(\rho,\R^n,\sigma,\g,\d)$ is an extension of $f$, namely 
\[
g^{(j)} (x) = f^{(j)}(x) 
\quad \ \forall x \in F, \ |j| \leq k,
\]
with norm
\begin{equation} \label{est.ext}
\| g \|_{\Lip(\rho,\R^n,\sigma,\g,\d)}
\leq C \| f \|_{\Lip(\rho,F,\sigma,\g,\d)},
\end{equation}
where $C = C' K_0 K$,
\begin{equation} \label{def K}
\begin{aligned} 
K_0 & := \max \{ 1, A_{aa}, B_{aa} : a \in \{ \sigma_0, \ldots, \sigma_k, \sigma_\rho\} \},
\\
K & := \max \{ 1, A_{ab}, B_{ab} : a,b \in \{ \sigma_0, \ldots, \sigma_k, \sigma_\rho\}, 
\ a \leq b \},
\end{aligned}
\end{equation}
$A_{ab}, B_{ab}$ are the constants in \eqref{S1}-\eqref{S2},
and $C'$ is a constant depending only on $k,n$; 
in particular, $C'$ is independent of $f, \rho, F,\sigma,\g,\d$.

The function $g^{(0)} : \R^n \to E_{\s_\rho}$ is differentiable $k$ times
in every point of $\R^n$, with partial derivatives
\[
\pa_x^j g^{(0)}(x) = g^{(j)}(x) \in E_{\s_j} \subseteq E_{\s_\rho}
\quad \ \forall x \in \R^n, \ |j| \leq k.
\]
In addition, on the open set $\Om := \R^n \setminus F$ one has 
\[
g^{(j)} (x) \in E_\infty 
\quad \ \forall x \in \Om, \ |j| \leq k,
\]
and for every $a \in \mI$ the function $g^{(0)}$ 
(more precisely, its restriction to $\Om$)
is of class $C^\infty(\Om, E_a)$.

The operator $\mE_k$ depends on $k, F, \g, \d$ 
and on the family $(S_\theta)_{\theta \geq 1}$ 
of smoothing operators, 
and it is independent of $\rho,\sigma$. 
\end{theorem}

\begin{remark} 
(\emph{The extension operator $\mE_k$ is independent of $\sigma$}).
\label{rem:indip.sigma}
Let $\s = \s_0, \ldots, \s_k, \s_\rho$ and 
$\s' = \s'_0, \ldots, \s'_k, \s'_\rho$ 
all belong to $\mI$, with $\s_m := \s - m \delta$, $\s'_m := \s'- m \delta$, 
and $\s > \s'$. 
Let $f \in \Lip(\rho, F, \sigma, \g, \delta)$. 
Then, by \eqref{0303.1}, $f$ also belongs to $\Lip(\rho,F,\sigma',\g,\delta)$. 
Thus, in principle, $f$ has an extension $g:= \mE_k^{(\s)} f$ 
because $f \in \Lip(\rho, F, \sigma, \g, \delta)$, 
and also an extension $g':= \mE_k^{(\s')} f$ 
because $f \in \Lip(\rho, F, \sigma', \g, \delta)$. 
The last sentence of Theorem \ref{thm:WET sigma} 
says that $g(x) = g'(x)$ for all $x \in \R^n$. 
\end{remark}

\begin{remark} 
(\emph{Constants $K_0, K$}).
\label{rem:KK0} 
When $(E_a)$ is given by $L^2$-based Sobolev spaces $H^s(\R^d)$ 
or $H^s(\T^d)$, the constants in \eqref{def K} 
are $K_0 = K = 1$: see Examples \ref{ex:1}, \ref{ex:3}.
\end{remark}

\begin{remark} 
(\emph{Fixed Banach space $Y$ as a special case}).
\label{rem:trivial} 
It is natural to expect that Theorem \ref{thm:WET sigma} 
includes the case of a fixed Banach space $Y$ described in subsection 
\ref{subsec:fixed.Banach.space} as a special case. 
On the other hand, $\delta = 0$ is not allowed by the assumption \eqref{param}, 
and this, at a first glance, could seem to be an obstacle. 
However, the obstruction is fictitious: 
given $Y$, we define 
a trivial scale of Banach spaces with smoothing operators 
by setting $\mI = [0,\infty)$, 
$(E_a, \| \ \|_a) = (Y, \| \ \|_Y)$ 
for all $a \geq 0$, 
and $S_\theta = I$ (the identity map) for all $\theta \geq 1$. 
The properties \eqref{0303.1}, \ldots, \eqref{S2} are then trivially satisfied,
and Theorem \ref{thm:WET sigma} applies with, say, $\delta = 1$. 
\end{remark}

Theorem \ref{thm:WET sigma} is proved in section \ref{sec:proof}.

\subsection{Some consequence} 
\label{subsec:corollary}

In this subsection we discuss some direct consequence of Theorem \ref{thm:WET sigma},  
along the line of the observations in Appendix B of \cite{BBHM}: 
using the extension operator, 
one proves that some properties of differentiable functions defined on $\R^n$ 
also hold for collections of functions defined on a closed subset $F \subset \R^n$.
With respect to \cite{BBHM}, 
where the spaces are those of Definition \ref{def Lip Y} 
(and Definition \ref{def Lip Y gamma} below)
with a fixed Banach space $Y$ as a codomain,
the novelty of the present discussion 
is the use of ``decreasing regularity'' norms.
With respect to \cite{BM-Euler}, 
where ``decreasing regularity'' norms are used 
for differentiable functions defined on $\R^n$,
the novelty of the present discussion is that 
we also deal with functions defined on a closed set $F \subset \R^n$,
differentiable in the sense of Whitney 
(i.e., in the sense of Definition \ref{def Lip sigma}).

We begin with the elementary observation (Lemma \ref{lemma:lippi.lippi})
that the spaces $\Lip(\rho,F,\sigma,\g,\delta)$ of Definition \ref{def Lip sigma} 
are contained, with continuous inclusion, in the spaces 
$\Lip(\rho,F,Y)$ of Definition \ref{def Lip Y}, 
or, more precisely, in their version $\Lip(\rho,F,Y,\gamma)$ 
with the weight $\g$, which we now define 
(the only difference between 
Definitions \ref{def Lip Y gamma} and \ref{def Lip Y} 
is the presence of $\gamma$).

\begin{definition}
\label{def Lip Y gamma}
\emph{(The space $\Lip(\rho,F,Y,\gamma)$)}.
Let $F \subseteq \R^n$ be a closed set, 
$k \geq 0$ an integer, 
$k < \rho \leq k+1$, 
and let $Y$ be a (real or complex) Banach space.  
We say that a collection 
$\{ f^{(j)} : j \in \N^n, \ 0 \leq |j| \leq k \}$ 
of functions $f^{(j)} : F \to Y$ 
belongs to $\Lip(\rho,F,Y,\gamma)$ 
if there exists a constant $M \geq 0$ 
such that the functions 
$f^{(j)}$ and the remainders $R_j$ 
defined by Taylor's formula \eqref{2502.1} satisfy
\begin{equation}
\label{2502.Y.gamma}
\g^{|j|} \| f^{(j)}(x) \|_Y \leq M, 
\quad \  
\g^{\rho} \| R_j(x,y) \|_Y \leq M |x-y|^{\rhost-|j|},
\quad \ 
\forall x,y \in F, \ |j| \leq k.
\end{equation}
The norm of an element 
$f = \{ f^{(j)} : 0 \leq |j| \leq k \}$ 
of the space $\Lip(\rhost,F,Y,\gamma)$ 
is defined as 
\begin{equation} \label{0203.1.Y.gamma}
\| f \|_{\Lip(\rho,F,Y,\gamma)} 
= \inf \{ M \geq 0 : \eqreff{2502.Y.gamma} \text{ holds}\, \}.
\end{equation}
\end{definition}

Clearly $\Lip(\rho,F,Y)$ and $\Lip(\rho,F,Y,\gamma)$ coincide as sets;  
the difference is only in the definition of the norms (with or without $\gamma$). 

\begin{lemma} \label{lemma:lippi.lippi}
\emph{($\Lip(\rho,F,\sigma,\gamma, \delta)$ is a subspace of $\Lip(\rho,F,Y,\gamma)$).}
Let $(E_a)_{a \in \mI}$, $F, \rho, k, \sigma, \g, \delta$ be as 
in Definition \ref{def Lip sigma}, 
and let $Y := E_{\s_\rho}$. Then 
\begin{equation} \label{Y leq sigma}
\Lip(\rho,F,\sigma,\gamma,\delta) \subset \Lip(\rho,F,Y,\gamma), 
\quad \ 
\| f \|_{\Lip(\rho,F,Y,\gamma)} 
\leq \| f \|_{\Lip(\rho,F,\sigma,\gamma,\delta)} 
\end{equation}
and, more precisely,
\begin{equation} \label{lippi.lippi}
\| f \|_{\Lip(\rho,F,\sigma,\gamma,\delta)} 
= \max \big\{ \| f \|_{\Lip(\rho,F,Y,\gamma)} \,, \ 
\sup_{x \in F} \g^{|j|} \| f^{(j)}(x) \|_{\s_j} , 
\ |j| \leq k \big\}
\end{equation}
for all $f = \{ f^{(j)} : |j| \leq k \} \in \Lip(\rho,F,\sigma,\gamma,\delta)$.
\end{lemma}

\begin{proof} It follows directly from Definitions \ref{def Lip sigma}, 
\ref{def Lip Y gamma} because, by \eqref{0303.1}, 
$\| f^{(j)}(x) \|_Y
= \| f^{(j)}(x) \|_{\s_\rho}
\leq \| f^{(j)}(x) \|_{\s_j}$ 
for all $x \in F$, all $|j| \leq k$. 
\end{proof}

Now consider the case $F = \R^n$ 
(here we follow the discussion on page 176 of \cite{S}).
Let $Y$ be a Banach space. 
If a collection $f = \{ f^{(j)} : |j| \leq k \}$ 
belongs to $\Lip(\rho,\R^n,Y,\gamma)$, 
then, for all $1 \leq |j| \leq k$ and all $x \in \R^n$, 
$f^{(j)}(x)$ is the partial derivative $\pa_x^j f^{(0)}(x)$, 
therefore the function $f^{(0)}$ alone 
determines the entire collection $f$. 
If, moreover, $\rho = k+1$, then the partial derivatives 
$f^{(j)}$ of $f^{(0)}$ of order $|j| = k$ are Lipschitz functions 
of $\R^n$ into $Y$. 
Suppose, in addition, that $Y$ is a Hilbert space
(or, more generally, that $Y$ has the Radon-Nikodym property).
Then, as is observed in \cite{BBHM} 
(referring to Chapter 5 of \cite{Ben-Lin}
to adapt the argument of page 176 of \cite{S}),
one applies a generalized version of Rademacher's Theorem
about the almost everywhere differentiability of Lipschitz functions, 
and obtains that, 
at almost every point of $\R^n$,
the function $f^{(0)} : \R^n \to Y$ admits partial derivatives of order $k+1$; 
moreover such partial derivatives $\pa_x^j f^{(0)}$, $|j| = k+1$, 
defined almost everywhere in $\R^n$, 
belong to $L^\infty(\R^n,Y)$, 
with bound 
\[
\g^{k+1} \| \pa_x^j f^{(0)} \|_{L^\infty(\R^n,Y)} 
\leq \| f \|_{\Lip(k+1,\R^n,Y,\g)}, 
\quad \ |j| = k+1.
\]
Thus to each element $f = \{ f^{(j)} : |j| \leq k \}$ of $\Lip(k+1,\R^n,Y,\gamma)$ 
corresponds a function, which is $f^{(0)}$, 
in the Sobolev space $W^{k+1,\infty}(\R^n,Y)$, 
with bound 
\begin{equation} \label{def norm WY}
\| f^{(0)} \|_{W^{k+1,\infty}(\R^n,Y,\gamma)} 
:= \max_{|j| \leq k+1} \, \g^{|j|} \| \pa_x^j f^{(0)} \|_{L^\infty(\R^n,Y)} 
\leq \| f \|_{\Lip(k+1,\R^n,Y,\gamma)}.
\end{equation}
It is also observed in \cite{BBHM}, 
following the discussion on page 176 
of \cite{S}, that the opposite correspondence also holds, 
namely that any function $f^{(0)} \in W^{k+1,\infty}(\R^n,Y)$,
together with its partial derivatives $f^{(j)} := \pa_x^j f^{(0)}$ 
of order $|j| \leq k$, 
gives a collection in $\Lip(k+1,\R^n,Y,\gamma)$, 
which satisfies 
\begin{equation} \label{frutta}
\| f \|_{\Lip(k+1,\R^n,Y,\gamma)} 
\leq C \| f^{(0)} \|_{W^{k+1,\infty}(\R^n,Y,\gamma)}, 
\end{equation}
for some $C \geq 1$ depending on $n,k$. 
Thus, identifying a collection $f = \{ f^{(j)} : |j| \leq k \}$ 
with its $0$-th function $f^{(0)}$, we can conclude that 
$\| \cdot \|_{W^{k+1,\infty}(\R^n,Y,\gamma)}$ 
(defined in \eqref{def norm WY})
and $\| \cdot \|_{\Lip(k+1,\R^n,Y,\gamma)}$ 
(defined in \eqref{0203.1.Y.gamma})
are equivalent norms on the space $\Lip(k+1,\R^n,Y,\gamma)$. 

\medskip

Now consider a scale $(E_a)_{a \in \mI}$ of Banach spaces with smoothing operators, 
as described in subsection \ref{subsec:scales}, 
and assume, in addition, that each $E_a$ is a Hilbert space 
(or, more generally, that $E_a$ has the Radon-Nikodym property). 
Let $f = \{ f^{(j)} : |j| \leq k \}$ be an element 
of the space $\Lip(k+1,\R^n,\sigma,\gamma, \delta)$ 
defined in Definition \ref{def Lip sigma}. 
By Lemma \ref{lemma:lippi.lippi}, $f$ also belongs to $\Lip(k+1,\R^n,Y,\gamma)$
where $Y := E_{\s_\rho}$, $\rho := k+1$, 
and therefore, as observed above, 
$f^{(0)} \in W^{k+1,\infty}(\R^n,Y)$. 
Moreover, by Definition \ref{def Lip sigma}, 
$f^{(j)}(x) \in E_{\s_j}$ for all $x \in \R^n$, all $|j| \leq k$, 
with 
\begin{equation} \label{ginepro.1}
\sup_{x \in \R^n} \g^{|j|} \| f^{(j)}(x) \|_{\s_j} 
\leq \| f \|_{\Lip(k+1,\R^n,\sigma,\gamma, \delta)}, 
\quad \ |j| \leq k.
\end{equation}
Define
\begin{equation} \label{ginepro.2}
\| f^{(0)} \|_{\sn} := \max \big\{ \| f^{(0)} \|_{W^{k+1,\infty}(\R^n,Y,\gamma)} \,, 
\g^{|j|} \| f^{(j)} \|_{L^\infty(\R^n,E_{\s_j})}, \  |j| \leq k \big\}, 
\end{equation}
where $Y := E_{\s_\rho}$,
$\rho := k+1$ (the letter $\sn$ in the index has no special meaning, it is just a short name).
From \eqref{def norm WY}, \eqref{Y leq sigma} (in which $\rho=k+1$ and $F=\R^n$)
and \eqref{ginepro.1} we deduce that
\begin{equation} \label{ginepro.3}
\| f^{(0)} \|_{\sn} 
\leq \| f \|_{\Lip(k+1,\R^n,\sigma,\gamma, \delta)}.
\end{equation}
By \eqref{lippi.lippi} (in which $\rho=k+1$ and $F=\R^n$), 
\eqref{frutta} and \eqref{ginepro.2} we obtain 
\begin{equation} \label{ginepro.4}
\| f \|_{\Lip(k+1,\R^n,\sigma,\gamma, \delta)}
\leq C \| f^{(0)} \|_{\sn}.
\end{equation}
Thus we have the following equivalence. 

\begin{lemma} \emph{(Equivalent norm on $\Lip(k+1,\R^n,\sigma,\gamma,\delta)$).} 
\label{lemma:equiv.Z}
Let $\rho = k+1$, $F = \R^n$, 
and consider the space $\Lip(k+1,\R^n,\sigma,\gamma,\delta)$ 
defined in Definition \ref{def Lip sigma}. 
For $f = \{ f^{(j)} : |j| \leq k \}$ in that space, define
\begin{equation} \label{ginepro.5}
\| f^{(0)} \|_{Z(k+1,\R^n,\sigma,\gamma,\delta)} 
:= \max_{|j| \leq k+1} \ \g^{|j|} \| \pa_x^j f^{(0)} \|_{L^\infty(\R^n,E_{\s_j})}
\end{equation}
(the existence almost everywhere of the partial derivatives $\pa_x^j f^{(0)}$ 
of order $|j| = k+1$ 
and the fact that they belong to $L^\infty$ 
are discussed above).
Then 

$(i)$ $\| f^{(0)} \|_{Z(k+1,\R^n,\sigma,\gamma,\delta)}$ 
defined in \eqref{ginepro.5} 
and $\| f^{(0)} \|_{\sn}$ 
defined in \eqref{ginepro.2} 
coincide; 

$(ii)$ $\| f^{(0)} \|_{Z(k+1,\R^n,\sigma,\gamma,\delta)}$ 
defined in \eqref{ginepro.5} 
and $\| f \|_{\Lip(k+1,\R^n,\sigma,\gamma, \delta)}$ 
defined in Definition \ref{def Lip sigma} 
are equivalent norms. 
\end{lemma} 

\begin{proof} $(i)$ follows from the definitions 
\eqref{ginepro.5} and \eqref{def norm WY}, 
because $\| \pa_x^j f^{(0)}(x) \|_{Y} 
= \| \pa_x^j f^{(0)}(x) \|_{\s_\rho} 
\leq \| \pa_x^j f^{(0)}(x) \|_{\s_j}$. 
$(ii)$ follows from $(i)$, \eqref{ginepro.3}, 
 \eqref{ginepro.4}.
\end{proof}

The advantage given by Lemma \ref{lemma:equiv.Z} 
is in the fact that working with the norms \eqref{ginepro.5} 
is simpler than with those in Definition \ref{def Lip sigma},
because in \eqref{ginepro.5} there are only derivatives, 
and no Taylor's remainders to estimate. 

\medskip

Finally we come to functions defined on a closed set $F \subset \R^n$.
A natural way of using our extension theorem is this: 
\begin{itemize}
\item
by Theorem \ref{thm:WET sigma}, 
any element of $\Lip(k+1,F,\sigma,\gamma,\delta)$ 
has an extension in $\Lip(k+1,\R^n,\sigma,\gamma,\delta)$; 

\item
using the norms \eqref{ginepro.5}, where only the derivatives are involved, 
we prove estimates and properties for the extended functions; 

\item
then we consider the restriction to $F$ of the extended functions, 
so as to obtain similar estimates 
and properties for elements of $\Lip(k+1,F,\sigma,\gamma,\delta)$.  
\end{itemize}

This path is rather general, in the sense that it applies 
to $\Lip(k+1,F,\sigma,\gamma,\delta)$ 
provided that the spaces of the scale $(E_a)_{a \in \mI}$ 
are Hilbert spaces, or have the Radon-Nikodym property. 
Clearly the last step of the path relies on the trivial inequality
\begin{equation} \label{trivial.restriction}
\sup_{x \in F} \| f^{(j)}(x) \|_{\s_j} 
\leq \sup_{x \in \R^n} \| f^{(j)}(x) \|_{\s_j}, 
\end{equation}
which holds just because $F \subset \R^n$, on one side, 
and on \eqref{est.ext} on the other side. 
This is essentially how norms \eqref{0203.1.Y.gamma} 
with fixed codomain $Y$ are treated in \cite{BBHM}.

As an example, in Proposition \ref{prop:appl} below 
we state precisely some of the basic estimates 
for Sobolev functions in the periodic setting.

For any integer $d \geq 1$, for any real $s$, 
let $H^s(\T^d,\EE)$, $\EE = \R$ or $\C$,
be the Sobolev space of functions $u : \R^d \to \EE$ of $d$ real variables, 
$\th \mapsto u(\th) = u(\th_1, \ldots, \th_d)$,
periodic with period $2\pi$ in each argument $\th_i$, 
with finite Sobolev norm 
\[
\| u \|_{H^s} := \Big( \sum_{\xi \in \Z^d} 
|\hat u_\xi|^2 \langle \xi \rangle^{2s} \Big)^{\frac12},
\quad \ 
\langle \xi \rangle := (1 + |\xi|^2)^{\frac12},
\]
where $\hat u_\xi$ is the Fourier coefficient of $u$ 
of frequency $\xi = (\xi_1, \ldots, \xi_d) \in \Z^d$, and 
$|\xi|^2 = \xi_1^2 + \ldots + \xi_d^2$.
Let 
\begin{equation} \label{def Ea Sob}
E_a := H^a(\T^d,\EE), \quad \ 
\| u \|_a := \| u \|_{H^a}, \quad \  
\mI := [0,\infty), \quad \ 
a_0 := 0. 
\end{equation}
The family $(E_a)_{a \in \mI} = (H^s(\T^d,\EE))_{s \geq 0}$
is a scale of Banach spaces equipped with smoothing operators 
as described in subsection \ref{subsec:scales}
(see subsection \ref{subsec:Sobolev}); 
moreover each $E_a = H^a(\T^d,\EE)$ is a Hilbert space.
Let $n,k$ be integers, with $n \geq 1$, $k \geq 0$. 
Let $F$ be a closed subset of $\R^n$, and let $\g > 0$.
For $s \in [0,\infty)$, we consider the space 
$\Lip(k+1, F, s, \g, 1)$ defined in Definition \ref{def Lip sigma}; 
here $\rho := k+1$, $\delta := 1$. 
Its elements are then collections 
$f = \{ f^{(j)} : j \in \N^n, \ |j| \leq k \}$ 
of functions $f^{(j)}: F \to E_{s-|j|}$,  
where $E_{s-|j|} = H^{s-|j|}(\T^d,\EE)$. 
For $s \in [0,\infty)$, to shorten the notation, we denote
\begin{equation} \label{def XsF}
X_{s,F} := \Lip(k+1, F, s, \g, 1), 
\quad \ 
\| u \|_{X_{s,F}} := \| u \|_{\Lip(k+1, F, s, \g, 1)}.
\end{equation}
For elements $u = \{ u^{(j)} : |j| \leq k \}$ of $X_{s,F}$ 
we use the notation (like in subsection \ref{subsec:what.comp})
\[
u^{(j)}(x,\th) := u^{(j)}(x) (\th) 
\quad \ \forall x \in F \subseteq \R^n, \ \ 
\th \in \R^d.
\]

\begin{prop} \label{prop:appl} 
Consider the scale of Sobolev spaces defined in \eqref{def Ea Sob}. 
Let $n \geq 1$, $k \geq 0$ be integers,
$F \subseteq \R^n$ a closed set, 
$\g > 0$.
For any real $s \geq 0$, let $X_{s,F}$, $\| \cdot \|_{X_{s,F}}$ 
be defined in \eqref{def XsF}. 
Let $\mE_k$ be the extension operator given by Theorem \ref{thm:WET sigma}.
For any $u \in X_{s,F}$, let 
\begin{equation} \label{def tilde}
\tilde u := (\mE_k u)^{(0)}
\end{equation}
denote the $0$-th element of the collection 
$\mE_k u = \{ (\mE_k u)^{(j)} : |j| \leq k \}$. 

\smallskip

$(i)$ \emph{(Product)}. Let $s \geq s_0 > k + 1 + \frac{d}{2}$. 
For $u,v \in X_{s,F}$, define the pointwise product $uv$ as the collection 
$\{ (uv)^{(j)} : |j| \leq k \}$ 
where $(uv)^{(0)}$ is the restriction to $x \in F$ of the pointwise product 
$\tilde u \tilde v$, namely 
$(uv)^{(0)}(x,\th) = \tilde u(x,\th) \tilde v(x,\th)$ for all $x \in F$, all $\th \in \R^d$,
and $(uv)^{(j)}$ is the restriction to $x \in F$ of the $j$-th partial derivative
$\pa_x^j (\tilde u \tilde v)$ of $\tilde u \tilde v$. 
Then $uv$ belongs to $X_{s,F}$ and satisfies 
\[
\| uv \|_{X_{s,F}}
\leq C_s \| u \|_{X_{s,F}} \| v \|_{X_{s_0,F}} 
+ C_{s_0} \| u \|_{X_{s_0,F}} \| v \|_{X_{s,F}},
\]
where $C_s, C_{s_0}$ are positive constants, 
with $C_s$ depending on $n,k,s$, 
 $C_{s_0}$ depending on $n,k,s_0$,
and both independent of $\g,F$.

\smallskip

$(ii)$ \emph{(Inverse diffeomorphism close to the identity)}. 
Let $s \geq s_0 > d + k + 3$.
There exists a constant $\d_0 \in (0,1)$, depending on $n,k,s_0$, 
with the following property. 
Let $\a = (\a_1, \ldots, \a_d)$,  
with $\a_i \in X_{s_0,F}$, $\a_i(x,\th)$ real, and 
\[
\| \a \|_{X_{s_0,F}} := 
\max_{i = 1, \ldots, d} \| \a_i \|_{X_{s_0,F}} \, 
\leq \d_0.
\]
For each $x \in F$, let
\[
f(x) : \R^d \to \R^d, 
\quad \ 
\th \mapsto f(x)(\th) = f(x,\th) := \th + \a(x,\th).
\]
For $x \in \R^n$, let $\tilde f(x,\th) := \th + \tilde \a(x,\th)$, 
where $\tilde \a := (\tilde \a_1, \ldots, \tilde \a_d)$ 
is defined as in \eqref{def tilde}.  
Then $\tilde f(x)$ is a diffeomorphism of $\R^d$ 
and also a diffeomorphism of $\T^d$, 
for all $x \in \R^n$. 
For all $x \in \R^n$, the inverse diffeomorphism $\tilde f(x)^{-1}$ has the form 
$\tilde f(x)^{-1}(\th) = \th + \bar \b(x,\th)$, 
where $\bar \b := (\bar \b_1, \ldots, \bar\b_d)$ 
has the same periodicity as $\a$ as a function of $\th$
(we write $\bar \b$, and not $\tilde \b$, 
because ``tilde'' has the meaning defined in \eqref{def tilde}). 
For each $i=1, \ldots, d$, 
consider the collection $\b_i = \{ \b_i^{(j)} : |j| \leq k \}$ 
where $\b_i^{(0)}$ is the restriction to $x \in F$ of the function 
$\bar \b_i$, namely $\b_i^{(0)}(x,\th) = \bar \b_i(x,\th)$ 
for all $x \in F$, all $\th \in \R^d$, 
and $\b_i^{(j)}$ is the restriction to $x \in F$ 
of the $j$-th partial derivative $\pa_x^j \bar \b_i$ of $\bar \b_i$.
Then each $\b_i$ is real, it belongs to $X_{s_0,F}$, 
and $\b := (\b_1, \ldots, \b_d)$ satisfies 
\[
\| \b \|_{X_{s_0,F}} \leq 2 \| \a \|_{X_{s_0,F}} \leq 2 \d_0, 
\quad \ 
\| \b \|_{X_{s,F}} \leq C_s \| \a \|_{X_{s,F}},
\]
where $C_s$ depends on $n,k,s$. 
Clearly for all $x \in F$
the map $\R^d \to \R^d$, $\th \mapsto \th + \b^{(0)}(x,\th)$ 
is the inverse diffeomorphism of the map $\R^d \to \R^d$, 
$\th \mapsto \th + \a^{(0)}(x,\th)$. 

\smallskip

$(iii)$ \emph{(Composition)}. Let $s \geq s_0 > d + k + 3$. 
Let $\a,f,\tilde \a, \tilde f$ be like in $(ii)$, and let $u \in X_{s,F}$. 
Define the composition map $u \circ f$ as the collection 
$\{ (u \circ f)^{(j)} : |j| \leq k \}$ where 
$(u \circ f)^{(0)}$ is the restriction to $x \in F$ 
of the map
\[
(\tilde u \circ \tilde f)(x,\th) := 
\tilde u(x) \big( \tilde f (x)(\th) \big) 
= \tilde u (x, \tilde f (x,\th)),
\]
namely $(u \circ f)^{(0)}(x,\th) = (\tilde u \circ \tilde f)(x,\th)$ 
for all $x \in F$, all $\th \in \R^d$,
and $(u \circ f)^{(j)}$ is the restriction to $x \in F$ 
of the $j$-th partial derivative $\pa_x^j (\tilde u \circ \tilde f)$
of $(\tilde u \circ \tilde f)$. 
Then $u \circ f \in X_{s,F}$, with bounds
\begin{align*}
\| u \circ f \|_{X_{s,F}} 
& \leq C_s ( \| u \|_{X_{s,F}} 
+ \| \a \|_{X_{s,F}} \| u \|_{X_{s_0,F}} ), 
\\ 
\| u \circ f - u \|_{X_{s,F}} 
& \leq C_s (\| u \|_{X_{s+1,F}} \| \a \|_{X_{s_0,F}}
+ \| u \|_{X_{s_0+1,F}} \| \a \|_{X_{s,F}}),
\end{align*}
where $C_s$ depends on $n,k,s$.
\end{prop}

\begin{proof} 
By the path described above, 
the proof is reduced to prove the corresponding estimates 
for functions defined on $\R^n$ working with norms \eqref{ginepro.5}. 
As is observed in subsection 2.1 of \cite{BM-Euler}, 
where norms \eqref{ginepro.5} are used, 
such estimates can be obtained, e.g., by adapting the proofs in \cite{BM-Mem}.
\end{proof}

\begin{remark} (\emph{Other estimates can be proved similarly}).
As is done in Proposition \ref{prop:appl}, one can easily pass 
from estimates regarding differentiable functions defined on $\R^n$ 
with finite norms \eqref{ginepro.5} 
to the same estimates 
(up to some constant, depending only on $n,k$, 
given by the application of Theorem \ref{thm:WET sigma})
for Whitney differentiable functions defined on a closed set $F \subset \R^n$ 
with finite norms \eqref{0903.8},
following the same path: 
first extend, 
then work with standard derivatives and norms \eqref{ginepro.5}, 
then restrict. 
\end{remark}

\section{Proof of Theorem \ref{thm:WET sigma}}
\label{sec:proof}

In this section we prove Theorem \ref{thm:WET sigma}. 
The proof is split into many lemmas, gathered in subsections.

In subsection \ref{subsec:notation} we fix the notations. 

In subsection \ref{subsec:cubes} we recall Whitney's decomposition in \cite{S} 
of open sets of $\R^n$ into dyadic cubes,   
stating only the properties that are used 
in the proof of Theorem \ref{thm:WET sigma}; 
precise references or direct proofs of such properties 
are postponed to section \ref{sec:app.cubes}.
 
In subsection \ref{subsec:def g} we define the collection $g$, 
extension of $f$ to $\R^n$, 
which is our candidate to prove the theorem in the case $\g=1$.
 
In subsection \ref{subsec:hyp} we state the hypotheses
that are tacitly assumed in all lemmas 
of subsections \ref{subsec:prime formule} to \ref{subsec:Rj}.

In subsection \ref{subsec:prime formule} we give equivalent expressions 
for $g^{(0)}$ and its derivatives of order $\leq k+1$, 
involving Taylor's polynomials.

In subsection \ref{subsec:est.far} we estimate $g$ 
at points $x \in \R^n$ at distance $\geq 1/2$ from $F$.
 
In subsection \ref{subsec:P-P j} we give a formula for the difference 
of Taylor's polynomial centered at two different points.

In subsection \ref{subsec:decomp} we introduce some further identities
involving the smoothing operators $S_\theta$, 
and we discuss why using only one of such decompositions is not enough.

In subsection \ref{subsec:est.close} we use such identities 
to prove estimates for $g$ in the vicinity of $F$, namely at points 
$x \in \R^n$ at $\dist(x,F) < 1/2$.

In subsection \ref{subsec:patch} we patch together 
the estimates proved in subsections \ref{subsec:est.far} and \ref{subsec:est.close}.

In subsection \ref{subsec:Rj} we extend the estimates 
of Taylor's remainders to include the case of 
two points both outside $F$.

In subsection \ref{subsec:conclusion} we complete the proof 
of Theorem \ref{thm:WET sigma} for $\g=1$,
and then we deduce the general case by rescaling.

\subsection{Notation} 
\label{subsec:notation}

To complete the multi-index notation \eqref{multi-index notation}, 
we say that two multi-indices $m=(m_1, \ldots, m_n)$, 
$j=(j_1,\ldots,j_n) \in \N^n$ satisfy 
``\,$m \leq j$\,'' (or ``\,$j \geq m$\,'')
if $m_i \leq j_i$ for all $i=1, \ldots, n$; 
we say that ``\,$m<j$\,'' (or ``\,$j>m$\,'')
if $m \leq j$ and $m \neq j$. 

The notation ``\,$a \lesssim b$'' (where $a,b$ are real numbers)
means ``there exists a constant $C>0$, 
\emph{depending only on $k,n$}, 
such that $a \leq C b$''. 
For example, a typical inequality is 
\[
\sum_{|\ell| \leq k} \frac{1}{\ell!} 4^{\rho - |\ell|} 
< 4^\rho \sum_{\ell_1, \ldots, \ell_n \geq 0} 
\frac{1}{\ell_1! \cdots \ell_n!} \Big( \frac14 \Big)^{\ell_1 + \ldots + \ell_n}
\leq 4^{k+1} e^{n/4},
\]
which we just write as 
$\sum_{|\ell| \leq k} \frac{1}{\ell!} 4^{\rho - |\ell|} \lesssim 1$.

\subsection{Dyadic cubes}
\label{subsec:cubes}

We recall the Whitney decomposition of open sets of $\R^n$ 
into dyadic cubes as given in Stein's book \cite{S} 
(and also in Grafakos' one \cite{Grafakos}, Appendix J). 
Here we only state the properties 
we need along the proof of Theorem \ref{thm:WET sigma}. 
Detailed references or direct proofs of all these properties 
are given in section \ref{sec:app.cubes}.

\begin{definition} \emph{(Cubes).}
\label{def:cubes}
By a \emph{cube} we mean a closed cube in $\R^n$,
i.e.\ a set of the form
$Q = \{ p + v : v \in [-r,r]^n \}$
with $p \in \R^n$ and $r>0$. 

We denote $\diam(Q)$ the diameter of $Q$, 
which is $\max \{ |x-y| : x,y \in Q \}$, 
and $\dist(Q,F)$ the distance of $Q$ from a closed set $F$, 
which is $\min \{ |x-y| : x \in Q, \ y \in F \}$
(this minimum is attained because $F$ is closed and $Q$ is compact).
\end{definition}

\begin{definition} \emph{(Expanded cubes).}
\label{def:Q*}
For any cube $Q = \{ p + v : v \in [-r,r]^n\}$,
we define $Q^*$ as the cube that has the same center $p$ as $Q$
and it is expanded by the factor 
$\lm = 1 + \e$, where $\e := \frac{1}{8 \sqrt n}$;
that is, $Q^* = \{ p+\lm v : v \in [-r,r]^n \}$.
\end{definition}

\begin{proposition} \emph{(Decomposition into cubes).} 
\label{prop:cubes}
Let $F \subseteq \R^n$ be a closed set, 
and let $\Om := \R^n \setminus F$ be its complement. 
Then there exists a collection of cubes
\[
\mF = \{ Q_1, Q_2, \ldots, Q_i, \ldots \}
\]
with the following properties.  
Denote 
\begin{equation} \label{def.qi}
q_i := \diam(Q_i).
\end{equation}
For each $Q_i$, 
let $Q_i^*$ be its expanded cube (Definition \ref{def:Q*}).
Then $\Om = \bigcup_{i=1}^\infty Q_i^*$, 
and for every point $x \in \Om$ there exists an open neighborhood $B_x$ of $x$, 
with $x \in B_x \subseteq \Om$, 
intersecting at most $12^n$ cubes $Q_i^*$.
For every $i = 1,2,\ldots$ one has 
$q_i > 0$, 
\begin{align} 
\label{diam.Q*.q}
q_i & \leq \diam(Q_i^*) \leq 2 q_i,
\\
\label{dist.QF.q}
q_i & \leq \dist(Q_i,F) \leq 4 q_i, 
\\
\label{dist.Q*F.q}
\frac12 q_i & \leq \dist(Q_i^*,F) \leq 4 q_i, 
\\
\label{dist.xF.q}
\frac12 q_i & \leq \dist(x,F) \leq 6 q_i
\quad \ \forall x \in Q_i^*. 
\end{align}
For each $i$, let $p_i$ be a point of $F$ that attains the distance between $Q_i$ and $F$, 
namely 
\begin{equation} \label{def.pi}
p_i \in F, \quad \ 
\dist(Q_i,F) = \dist(Q_i,p_i).
\end{equation}
Then
\begin{equation} \label{dist.xpi.q}
\frac12 q_i \leq |x-p_i| \leq 6 q_i
\quad \ \forall x \in Q_i^*.
\end{equation}
Moreover 
\begin{equation} \label{dist.xy.q}
\frac12 q_i \leq |x-y| \quad \ \forall x \in Q_i^*, \ \ y \in F.
\end{equation} 
\end{proposition}

\begin{remark} (\emph{Diameters $q_i$ as reference distances}).
In all the estimates of Proposition \ref{prop:cubes}, 
$q_i$ is chosen as the only ``reference distance'' 
with which all the other distances are compared. 
The various factors $\frac12, 4, 6$ in those bounds, 
as well as the value of $\e$ in Definition \ref{def:Q*}, 
could be improved (but this is not relevant in the present paper, 
as having sharp constants in Proposition \ref{prop:cubes} 
would lead to the same kind of result in Theorem \ref{thm:WET sigma}).
\end{remark}

\begin{proposition} \emph{(Partition of unity).}
\label{prop:partition of unity}
Let $F, \Om, \mF, Q_i, Q_i^*, q_i$ 
be as in Proposition \ref{prop:cubes}.
For each cube $Q_i \in \mF$ there exists a function 
$\ph_i^* \in C^\infty(\R^n, \R)$ such that 
\begin{equation} \label{POU}
0 \leq \ph_i^*(x) \leq 1
\ \  \forall x \in \R^n; 
\qquad 
\ph_i^*(x) = 0 
\ \  \forall x \notin Q_i^*;
\qquad 
\sum_{i=1}^\infty \ph_i^*(x) = 1 
\ \  \forall x \in \Om
\end{equation}
where the series is, in fact, a finite sum in a neighbourhoof of any $x \in \Om$.
Moreover, for all nonzero multi-indices $\ell$, 
\begin{equation} \label{POU.der}
\sum_{i=1}^\infty \pa_x^\ell \ph_i^*(x) = 0 \ \ \forall x \in \Om; 
\qquad 
|\pa_x^\ell \ph_i^*(x)| \leq \frac{C_\ell}{q_i^{|\ell|}} 
\ \ \forall x \in \Om, 
\end{equation}
where the constant $C_\ell > 0$ depends only on $n,\ell$, 
and not on $Q_i,x,F$.
\end{proposition}

We define 
\begin{equation} \label{def.mN}
\mN := \{ i : q_i \leq 1 \} \subseteq \N,
\end{equation}
where $q_i$ is defined in \eqref{def.qi}. 
The cubes $Q_i$ with index $i \in \mN$, namely those with diameter 
$q_i \leq 1$, have bounded distance from $F$, 
because, by \eqref{dist.QF.q}, $\dist(Q_i,F) \leq 4$ for all $i \in \mN$.

\begin{lemma} \emph{(Partition of unity close to $F$).}
\label{lemma:cubi vicini}
Let $x \in \Om$ with $\dist(x,F) < 1/2$. 
Then $\ph_i^*(x) = 0$ for all $i \notin \mN$, 
and therefore
\begin{equation} \label{1.0}
1 = \sum_{i=1}^\infty \ph_i^*(x)
= \sum_{i \in \mN} \ph_i^*(x), 
\qquad  
0 = \sum_{i \in \mN} \pa_x^\ell \ph_i^*(x)
\quad \ \forall \ell \in \N^n, \ \ \ell \neq 0.
\end{equation}
\end{lemma}

\subsection{Definition of the extension $g$} 
\label{subsec:def g}
Let $F \subseteq \R^n$ be a closed set, 
let $k \geq 0$ be an integer, 
and let $f = \{ f^{(j)} : j \in \N^n, \ |j| \leq k \}$ 
be a collection of functions, with $f^{(j)} : F \to E_{a_0}$. 
Let $\d > 0$. 
We define the collection 
\[
g = \{ g^{(j)} : j \in \N^n, \ |j| \leq k \}
\] 
in the following way. 
On $F$, we define 
\begin{equation} \label{1003.1}
g^{(j)}(x) := f^{(j)}(x) \quad \ \forall x \in F, \ |j| \leq k
\end{equation}
(\eqref{1003.1} is the only possible choice on $F$, 
since we want $g$ to be an extension of $f$). 
On $\Om := F^c = \R^n \setminus F$, we define 
\begin{equation} \label{def g0}
g^{(0)}(x) 
:= \sum_{i \in \mN} \, \sum_{\begin{subarray}{c} 
\ell \in \N^n  \\
|\ell| \leq k
\end{subarray}} 
\frac{1}{\ell!}\, S_{\theta_i}[ f^{(\ell)}(p_i) ] \, (x-p_i)^\ell \, \ph_i^*(x)
\quad \ \forall x \in \Om
\end{equation}
where $\mN$ is defined in \eqref{def.mN}, 
$p_i$ are the points of $F$ defined in \eqref{def.pi}, 
$\ph_i^*$ are defined in Proposition \ref{prop:partition of unity}, 
$S_ {\theta_i}$ are the smoothing operators $S_\theta$ in \eqref{S.generic} 
with parameters $\theta = \theta_i$, and 
\begin{equation} \label{def theta}
\theta_i := q_i^{-\tau},
\quad \ \tau := \frac{1}{\delta}, 
\end{equation}
where $q_i$ is defined in \eqref{def.qi}.

By Proposition \ref{prop:cubes}, 
for every $x \in \Om$ there exists an open neighborhood $B_x$ of $x$, 
with $B_x \subseteq \Om$, such that $B_x$ intersects 
at most $12^n$ of the expanded cubes $Q_i^*$; 
hence there exists a finite set $\mN_x \subseteq \mN$
of indices $i \in \N$ such that 
$B_x \cap Q_i^*$ is empty for all $i \notin \mN_x$. 
Therefore, on $B_x$, 
\begin{equation} \label{locally finite}
\ph_i^*(y) = 0 \quad \ \forall y \in B_x, \ i \notin \mN_x,
\end{equation}
and $g^{(0)}$ on $B_x$ is given by the finite sum 
\begin{equation} \label{locally finite sum}
g^{(0)}(y) = \sum_{i \in \mN_x} \, 
\sum_{|\ell| \leq k} 
\frac{1}{\ell!}\, S_{\theta_i}[ f^{(\ell)}(p_i) ] \, (y-p_i)^\ell \, \ph_i^*(y)
\quad \ \forall y \in B_x.
\end{equation}
For each $i,\ell$, 
the map $y \mapsto (y-p_i)^\ell \ph_i^*(y)$ is $C^\infty (\R^n, \R)$, 
while $S_{\theta_i}[ f^{(\ell)}(p_i) ]$ is a vector of $E_\infty$
which does not depend on $y$.  
Thus $g^{(0)}(y)$ is a $C^\infty$ function of $y$ from the open set $B_x$ 
to $E_{a_0}$ (in fact, to $E_s$ for all $s \in \mI$).  
Hence the derivative $\pa_x^j g^{(0)}(x)$ 
is well-defined in every point $x \in \Om$, 
for every multi-index $j \in \N^n$.
We define the remaining elements of the collection $g$ as
\begin{equation} \label{def gj}
g^{(j)}(x) := 
\pa_x^j g^{(0)}(x) 
\quad \ \forall x \in \Om, \ \	1 \leq |j| \leq k.
\end{equation}
Thus \eqref{1003.1}, \eqref{def g0}, \eqref{def gj} 
define $g^{(j)}(x)$ for all $0 \leq |j| \leq k$ and all $x \in \R^n$.

\subsection{Assumptions}
\label{subsec:hyp}

In all lemmas of subsections 
\ref{subsec:prime formule}, \ldots, \ref{subsec:Rj}
the following assumptions are always understood:
$\mI, a_0, E_a, \| \ \|_a, S_\theta$ are as 
in subsection \ref{subsec:scales};
$k, \a, \rho, \gamma, \delta, F, n, \sigma, \sigma_j, \sigma_\rho$ 
are as in Definition \ref{def Lip sigma}; 
the collection $f = \{ f^{(j)} : |j| \leq k \}$ 
belongs to $\Lip(\rho,F,\sigma,\g,\delta)$ 
as described in Definition \ref{def Lip sigma}; 
moreover we assume that $\g=1$, 
and that $f$ satisfies \eqref{0903.1-2} 
for some constant $M \geq 0$, namely 
\begin{alignat}{2}
\label{0903.1.normalized}
\| f^{(j)}(x) \|_{\sigma_j} & \leq M 
\quad && \forall x \in F, \ \ |j| \leq k,
\\
\| R_j(x,y) \|_{\sigma_\rho} & \leq M |x-y|^{\rho-|j|}
\quad && \forall x,y \in F, \ \  |j| \leq k.
\label{0903.2.normalized}
\end{alignat}
Also, we assume that 
$\Om, Q_i, q_i, Q_i^*, p_i, \ph_i^*, \mN$ 
are as in subsection \ref{subsec:cubes}, 
$g = \{ g^{(j)} : |j| \leq k \}, \theta_i, \tau$ 
as in subsection \ref{subsec:def g},
and $K_0,K$ are defined in \eqref{def K}.

\subsection{Equivalent formula for $g^{(0)}$ and its derivatives in $\Om$} 
\label{subsec:prime formule}

In \eqref{def Pj Rj} the function $P_j(x,y)$ has been defined for $x,y \in F$; 
however, the same formula is well defined for all $x \in \R^n$, 
because $P_j(x,y)$ is a polynomial in $x$. 
Thus, for $y \in F$, $x \in \R^n$, we define 
\begin{align}
P_j(x,y) 
& := \sum_{|j+\ell|\leq k} \frac{1}{\ell!}\, f^{(j+\ell)}(y) (x-y)^\ell
\quad \ \forall x \in \R^n, \ \ y \in F, \ \ |j| \leq k.
\label{def Pj}
\end{align}
For each $y \in F$, $f^{(j+\ell)}(y)$ is a vector of $E_{a_0}$ independent of $x$,
therefore the map $x \mapsto P_j(x,y)$ is $C^\infty(\R^n, E_{a_0})$, 
and, by induction, one proves that 
\[
\pa_x^j P_0(x,y) = P_j(x,y) 
\quad \ \forall x \in \R^n, \ \ y \in F, \ \ |j| \leq k
\]
(of course $f^{(j+\ell)}(y)$ is not obtained by differentiating with respect to $y$ --- 
it would not be allowed at this stage --- 
but by renaming the summation index after differentiating with respect to $x$).

Recall the formula 
\begin{equation} \label{ab.der}
\pa_x^j (ab) = \sum_{0 \leq m \leq j} \binom{j}{m} (\pa_x^m a) \, (\pa_x^{j-m} b)
\end{equation}
for the partial derivatives of the product of any two functions $a(x), b(x)$ 
(the case $j=0$ is trivially included).
By the definitions \eqref{def g0}, \eqref{def gj}, 
and recalling that the sum is locally finite 
(see \eqref{locally finite}-\eqref{locally finite sum}),
for all $x \in \Om$, $|j| \leq k$ we calculate 
\begin{align}
g^{(j)}(x) 
= \pa_x^j g^{(0)}(x) 
& = \sum_{0 \leq m \leq j} \binom{j}{m} 
\sum_{i \in \mN} \sum_{|\ell| \leq k} S_{\theta_i}[ f^{(\ell)}(p_i)] 
\pa_x^m \Big( \frac{ (x-p_i)^\ell }{\ell!} \Big) \, \pa_x^{j-m} \ph_i^*(x).
\label{tutti} 
\end{align}
By induction, one has
\[
\pa_x^m \Big( \frac{ (x-p_i)^\ell }{\ell!} \Big) = 
\frac{ (x-p_i)^{\ell-m} }{(\ell-m)!} \quad \ \forall m \in \N^n, \ m \leq \ell, 
\]
while $\pa_x^m ( (x-p_i)^\ell ) = 0$ for all multi-indices $m$ that are not $\leq \ell$. 
Hence, after changing the summation index $\ell - m = \ell'$ 
and then renaming $\ell' \to \ell$, 
and recalling definition \eqref{def Pj}, one has  
\begin{align}
g^{(j)}(x)
& = \sum_{0 \leq m \leq j} \binom{j}{m} 
\sum_{i \in \mN} 
\sum_{\begin{subarray}{c} \ell \in \N^n \\ |m+\ell| \leq k \end{subarray}} 
S_{\theta_i}[ f^{(m+\ell)}(p_i)]\frac{ (x-p_i)^{\ell} }{\ell!} \, \pa_x^{j-m} \ph_i^*(x)
\notag \\ & 
= \sum_{0 \leq m \leq j} \binom{j}{m} 
\sum_{i \in \mN} S_{\theta_i}[ P_m(x,p_i)] \pa_x^{j-m} \ph_i^*(x)
\quad \ \forall x \in \Om, \ \ |j| \leq k.
\label{tutti.1}
\end{align}

In the same way we also obtain a formula for the 
partial derivatives of $g^{(0)}(x)$ 
of order $k+1$: 
given any multi-index $j$ of length $|j| = k+1$, 
by \eqref{ab.der} the derivative $\pa_x^j g^{(0)}(x)$ 
is given by the r.h.s.\ of \eqref{tutti}; 
for $m=j$ one has $\pa_x^m((x-p_i)^\ell) = 0$ 
because $m$ is not $\leq \ell$ 
($|m| = |j| = k+1$, while $|\ell| \leq k$).
Hence the term with $m=j$ can be removed from the sum, and, 
after changing the summation index $\ell-m=\ell'$ 
and renaming $\ell' \to \ell$, 
we obtain
\begin{align}
\pa_x^j g^{(0)}(x)
& = \sum_{0 \leq m < j} \binom{j}{m} 
\sum_{i \in \mN} S_{\theta_i}[ P_m(x,p_i)] \pa_x^{j-m} \ph_i^*(x)
\quad \ \forall x \in \Om, \ \ |j| = k+1.
\label{der.k+1}
\end{align}
Note that all multi-indices $m$ in \eqref{der.k+1} 
have length $|m| \leq k$, 
therefore all the polynomials $P_m(x,p_i)$ in \eqref{der.k+1}
are well defined. 

\subsection{Estimates at points far from $F$}
\label{subsec:est.far}

We prove some estimates for $g$ and for polynomials 
at points $x$ not close to $F$.

\begin{lemma} \emph{($g=0$ at $\dist(x,F) > 6$).} 
\label{lemma:cubi lontani}
One has 
\[
g^{(j)}(x) = 0 
\quad \ \forall |j| \leq k, \ \ 
x \in \Om, \ \ 
\dist(x,F) > 6.
\]
\end{lemma}

\begin{proof}
Suppose that $\dist(x,F) > 6$ and $x \in Q_i^*$ for some $i \in \mN$. 
Then $q_i \leq 1$ and, by \eqref{dist.xF.q}, $6 < \dist(x,F) \leq 6 q_i \leq 6$, 
a contradiction. 
This implies that on the open set $\{ x \in \Om : \dist(x,F) > 6 \}$
one has $\ph_i^*(x) = 0$ for all $i \in \mN$. 
Hence, by definition \eqref{def g0}, 
$g^{(0)}(x) = 0$ for all $x$ in that open set. 
As a consequence, all the derivatives $g^{(j)}(x)$ 
are also zero on that set. 
\end{proof}

\begin{lemma} \emph{(Bound for $\| g^{(j)}(x) \|_{\sigma_\rho}$ at $1/2 \leq \dist(x,F) \leq 6$).}
\label{lemma:cubi intermedi}
One has 
\[
\| g^{(j)}(x) \|_{\sigma_\rho} \lesssim K M
\quad \ \forall |j| \leq k, 
\ \ x \in \Om, 
\ \ 1/2 \leq \dist(x,F) \leq 6.
\]
\end{lemma}

\begin{proof}
By \eqref{tutti.1},
\[
\| g^{(j)}(x) \|_{\sigma_\rho} 
\lesssim \sum_{0 \leq m \leq j} 
\sum_{i \in \mN} 
\sum_{|m+\ell| \leq k} 
\| S_{\theta_i}[ f^{(m+\ell)}(p_i)] \, \|_{\sigma_\rho} 
|x-p_i|^{|\ell|} |\pa_x^{j-m} \ph_i^*(x)|.
\]
By \eqref{S1}, \eqref{def K}, 
and then \eqref{0303.1}, one has 
\[
\| S_{\theta_i}[ f^{(m+\ell)}(p_i)] \, \|_{\sigma_\rho}
\leq K \| f^{(m+\ell)}(p_i) \|_{\sigma_\rho} 
\leq K \| f^{(m+\ell)}(p_i) \|_{\sigma_{m+\ell}}
\]
because $\sigma_\rho < \sigma_{m+\ell}$, see \eqref{def sj}. 
By \eqref{0903.1.normalized}, 
$ \| f^{(m+\ell)}(p_i) \|_{\sigma_{m+\ell}} \leq M$.
If $\pa_x^{j-m} \ph_i^*(x) \neq 0$, 
then $x \in Q_i^*$, and $|x-p_i| \lesssim q_i$
by \eqref{dist.xpi.q}. 
Moreover, by \eqref{POU.der}, 
$|\pa_x^{j-m} \ph_i^*(x)| \lesssim q_i^{|m|-|j|}$. 
Hence 
\[
\| g^{(j)}(x) \|_{\sigma_\rho} 
\lesssim K M \sum_{0 \leq m \leq j} 
\sum_{\begin{subarray}{c} i \in \mN \\ x \in Q_i^* \end{subarray}} 
\sum_{|m+\ell| \leq k} q_i^{|\ell| + |m| - |j|}.
\]
The exponent $|\ell| + |m| - |j|$ can be positive, negative or zero; 
however, $q_i \leq 1$ because $i \in \mN$ 
and $q_i \geq 1/12$ 
because, by assumption, $1/2 \leq \dist(x,F)$, 
and, by \eqref{dist.xF.q}, $\dist(x,F) \leq 6 q_i$. 
Hence 
\[
q_i^{|\ell| + |m| - |j|} \lesssim 1
\]
for all $\ell, m, j$ in the sum. 
Finally, the number of indices $i$ such that $x \in Q_i^*$ is bounded by $12^n$
by Proposition \ref{prop:cubes}. 
\end{proof}

\begin{lemma} 
\emph{(Bound for $\| g^{(j)}(x) \|_{\sigma_j}$ at $1/2 \leq \dist(x,F) \leq 6$).}
\label{lemma:cubi intermedi sj}
One has 
\[
\| g^{(j)}(x) \|_{\sigma_j} \lesssim KM 
\quad \ \forall |j| \leq k, 
\ \ x \in \Om, 
\ \ 1/2 \leq \dist(x,F) \leq 6.
\]
\end{lemma}

\begin{proof}
By \eqref{tutti.1}, 
\begin{equation} \label{p.01}
\| g^{(j)}(x) \|_{\sigma_j} 
\lesssim 
\sum_{0 \leq m \leq j} 
\sum_{i \in \mN} 
\sum_{|m+\ell| \leq k} 
\| S_{\theta_i}[ f^{(m+\ell)}(p_i)] \, \|_{\sigma_j} 
|x-p_i|^{|\ell|} |\pa_x^{j-m} \ph_i^*(x)|.
\end{equation}

$\bullet$ \emph{Case $|m|+|\ell| \leq |j|$}. 
For indices $m,\ell$ such that $|m|+|\ell| \leq |j|$ 
one has $\sigma_j \leq \sigma_{m+\ell}$ (see \eqref{def sj}), 
therefore we proceed similarly as in the proof of 
Lemma \ref{lemma:cubi intermedi}: 
by \eqref{S1}, \eqref{def K}, 
then \eqref{0303.1}, 
and then \eqref{0903.1.normalized}, one has 
\[
\| S_{\theta_i}[ f^{(m+\ell)}(p_i)] \, \|_{\sigma_j}
\leq K \| f^{(m+\ell)}(p_i) \|_{\sigma_j} 
\leq K \| f^{(m+\ell)}(p_i) \|_{\sigma_{m+\ell}}
\leq KM.
\]
If $\pa_x^{j-m} \ph_i^*(x) \neq 0$, 
then $x \in Q_i^*$, and therefore  
$|x-p_i| \lesssim q_i$
by \eqref{dist.xpi.q}; 
also, $|\pa_x^{j-m} \ph_i^*(x)|$ $\lesssim q_i^{|m|-|j|}$
by \eqref{POU.der}.
Hence 
\[
|x-p_i|^{|\ell|} |\pa_x^{j-m} \ph_i^*(x)|
\lesssim q_i^{|\ell|+|m|-|j|}. 
\]
The exponent $|\ell| + |m| - |j|$ is $\leq 0$; 
but $q_i \geq 1/12$ 
because, by assumption, $1/2 \leq \dist(x,F)$, 
and, by \eqref{dist.xF.q}, $\dist(x,F) \leq 6 q_i$. 
Hence $q_i^{|\ell|+|m|-|j|} \lesssim 1$, and therefore
\[
\| S_{\theta_i}[ f^{(m+\ell)}(p_i)] \, \|_{\sigma_j} 
|x-p_i|^{|\ell|} |\pa_x^{j-m} \ph_i^*(x)|
\lesssim KM.
\]

$\bullet$ \emph{Case $|j| < |m|+|\ell|$}. 
In this case one has $\sigma_j > \sigma_{m+\ell}$ (see \eqref{def sj}).
Therefore, by \eqref{S1}, \eqref{def K}, 
then \eqref{0903.1.normalized}, 
and then \eqref{def sj}, \eqref{def theta},  
\begin{align*}
\| S_{\theta_i}[f^{(m+\ell)}(p_i)] \, \|_{\sigma_j} 
& \leq K \theta_i^{\sigma_j - \sigma_{m+\ell}} \| f^{(m+\ell)}(p_i) \|_{\sigma_{m+\ell}}
\leq K M \theta_i^{\sigma_j - \sigma_{m+\ell}}
= K M q_i^{|j|-|m|-|\ell|}.
\end{align*}
As above, 
\[
|x-p_i|^{|\ell|} \lesssim q_i^{|\ell|}, 
\quad \  
|\pa_x^{j-m} \ph_i^*(x)| \lesssim q_i^{|m|-|j|},
\]
whence we directly obtain 
\[
\| S_{\theta_i}[ f^{(m+\ell)}(p_i)] \, \|_{\sigma_j} 
|x-p_i|^{|\ell|} |\pa_x^{j-m} \ph_i^*(x)|
\lesssim KM.
\]

In both cases, the general term in the sum \eqref{p.01} is $\lesssim KM$, 
and the number of indices $i$ such that $x \in Q_i^*$ is bounded by $12^n$
by Proposition \ref{prop:cubes}. 
\end{proof}

\begin{lemma}
\emph{(Bound for $\| P_j(x,y) \|_{\sigma_\rho}$ at $\dist(x,F) \geq 1/2$).}
\label{lemma:Pj lontani}
One has 
\[
\| P_j(x,y) \|_{\sigma_\rho} 
\lesssim M |x-y|^{\rho - |j|} 
\quad \ 
\forall |j| \leq k, 
\ \ y \in F, 
\ \ x \in \Om, \ \ \dist(x,F) \geq 1/2.
\]
\end{lemma}

\begin{proof} 
By \eqref{def Pj}, \eqref{0303.1}, \eqref{0903.1.normalized}, 
\[
\| P_j(x,y) \|_{\sigma_\rho} 
\lesssim M \sum_{|j+\ell| \leq k} |x-y|^{|\ell|}.
\]
Now $|x-y| \geq \dist(x,F)$ by definition of distance, 
and $\dist(x,F) \geq 1/2$ by assumption. 
Hence $(2|x-y|)^{\rho-|j|-|\ell| } \geq 1$ 
because $\rho-|j|-|\ell| \geq 0$, whence 
$|x-y|^{|\ell|} \lesssim |x-y|^{\rho - |j|}$.
\end{proof}

\begin{lemma}
\emph{(Bound for $\| g^{(j)}(x) - P_j(x,y) \|_{\sigma_\rho}$ at $\dist(x,F) \geq 1/2$).}
\label{lemma:Rj lontani}
One has 
\begin{align*}
& \| g^{(j)}(x) - P_j(x,y) \|_{\sigma_\rho} 
\lesssim K M |x-y|^{\rho-|j|} 
\\ 
& \forall |j| \leq k, \ \ 
y \in F, \ \ 
x \in \Om, \ \ 
\dist(x,F) \geq 1/2.
\end{align*}
\end{lemma}

\begin{proof}
By the estimates 
in Lemmas \ref{lemma:cubi lontani}, \ref{lemma:cubi intermedi}, 
\ref{lemma:Pj lontani},
\[
\| g^{(j)}(x) - P_j(x,y) \|_{\sigma_\rho} 
\leq \| g^{(j)}(x) \|_{\sigma_\rho} + \| P_j(x,y) \|_{\sigma_\rho} 
\lesssim KM (1 + |x-y|^{\rho-|j|}).
\]
Moreover $1 \leq (2|x-y|)^{\rho-|j|}$ 
because $\rho-|j| \geq 0$ and $2|x-y| \geq 2 \dist(x,F) \geq 1$. 
\end{proof}

\begin{lemma}
\emph{(Bound for $\| \pa_x^j g^{(0)}(x) \|_{\sigma_\rho}$ 
of order $|j|=k+1$ at $1/2 \leq \dist(x,F) \leq 6$).}
\label{lemma:der.k+1.far}
One has 
\[
\| \pa_x^j g^{(0)}(x) \|_{\sigma_\rho} \lesssim K M
\quad \ \forall |j| = k+1, 
\ \ x \in \Om, 
\ \ 1/2 \leq \dist(x,F) \leq 6.
\]
\end{lemma}

\begin{proof}
For $1/2 \leq \dist(x,F) \leq 6$, for $|j| = k+1$, 
by \eqref{der.k+1} one has 
\[
\| \pa_x^j g^{(0)}(x) \|_{\sigma_\rho} 
\lesssim \sum_{0 \leq m < j} 
\sum_{i \in \mN} 
\sum_{|m+\ell| \leq k} 
\| S_{\theta_i}[ f^{(m+\ell)}(p_i)] \, \|_{\sigma_\rho} 
|x-p_i|^{|\ell|} |\pa_x^{j-m} \ph_i^*(x)|.
\]
Then we follow word-by-word the proof of Lemma \ref{lemma:cubi intermedi},
with the only difference that now $m<j$ (instead of $m \leq j$) 
and $|j| = k+1$ (instead of $|j| \leq k$).
\end{proof}

\subsection{Formula for the difference of Taylor's polynomials}
\label{subsec:P-P j}

In the next subsections we will repeatedly make use of 
formula \eqref{P-P j} for the difference of two Taylor's polynomials. 
Such a formula is given by the Lemma on page 177 of \cite{S}. 
For the sake of completeness, 
we give here a (slightly different, and more explicit) proof.

\begin{lemma} \label{lemma:polinomi}
\emph{$($Formula for $P_j(x,y) - P_j(x,z))$}.
For all $x \in \R^n$, $y,z \in F$, $|j| \leq k$, one has 
\begin{equation} \label{P-P j}
P_j(x,y) - P_j(x,z) = \sum_{|j+\ell| \leq k} \frac{1}{\ell!} R_{j+\ell}(y,z) (x-y)^\ell.
\end{equation}
\end{lemma}

\begin{proof}
Splitting $(x-z) = (x-y) + (y-z)$, one has 
\begin{align*}
P_j(x,z) 
& = \sum_{\ell: |j+\ell| \leq k} \frac{1}{\ell!} f^{(j+\ell)}(z) (x-z)^\ell
\\ 
& = \sum_{\ell:|j+\ell| \leq k} \frac{1}{\ell!} f^{(j+\ell)}(z) 
\sum_{0 \leq m \leq \ell} \binom{\ell}{m} (x-y)^m (y-z)^{\ell-m}
\\ 
& = \sum_{\begin{subarray}{c} \ell, m \in \N^n \\ 
|j+\ell| \leq k \\ 
m \leq \ell \end{subarray}} 
\frac{1}{m! (\ell-m)!} f^{(j+\ell)}(z) (x-y)^m (y-z)^{\ell-m} 
\qquad [\text{change } \ell = m + h]
\\
& = \sum_{\begin{subarray}{c} h, m \in \N^n \\ 
|j+m+h| \leq k \end{subarray}} 
\frac{1}{m! h!} f^{(j+m+h)}(z) (x-y)^m (y-z)^h 
\\
& = \sum_{m : |j+m| \leq k} \frac{1}{m!} (x-y)^m
\Big( \sum_{h : |j+m+h| \leq k} 
\frac{1}{h!} f^{(j+m+h)}(z)  (y-z)^h \Big) 
\\ 
& = \sum_{m:|j+m| \leq k} \frac{1}{m!} P_{j+m}(y,z) (x-y)^m.
\end{align*}
Hence 
\[
P_j(x,y) - P_j(x,z) 
= \sum_{|j+m| \leq k} \frac{1}{m!} f^{(j+m)}(y) (x-y)^m 
- \sum_{|j+m| \leq k} \frac{1}{m!} P_{j+m}(y,z) (x-y)^m.
\]
Now $P_{j+m}(y,z)$ is the Taylor polynomial of $f^{(j+m)}$ 
centered at $z$, hence the difference 
$f^{(j+m)}(y) - P_{j+m}(y,z)$ is (by its definition) the remainder $R_{j+m}(y,z)$.
\end{proof}

\subsection{Decomposition by smoothing operators at points close to $F$}
\label{subsec:decomp}

Our aim now is to prove bounds like the ones in Lemmas \ref{lemma:Rj lontani},
\ref{lemma:cubi intermedi sj} also for $x$ close to $F$. 
We need first to decompose both the difference 
$(g^{(j)}(x) - P_j(x,y))$ and the function $g^{(j)}(x)$  
by means of the smoothing operators $S_\theta$.
We have to introduce two different decompositions
in Lemmas \ref{lemma:ZjmLH} and \ref{lemma:trick theta x};
the reason is discussed in Remark \ref{rem:two decomp}.

\begin{lemma}
\emph{(Decomposition of the difference $(g^{(j)}(x) - P_j(x,y))$ at $\dist(x,F) < 1/2$).}
\label{lemma:ZjmLH}
One has 
\begin{align*}
& g^{(j)}(x) - P_j(x,y) 
= \sum_{0 \leq m \leq j} \binom{j}{m} \Big( Z_{jm}^L(x,y) - Z_{jm}^H(x,y) \Big)
\\ 
& \forall |j| \leq k, \ \ 
y \in F, \ \ 
x \in \Om, \ \ 
\dist(x,F) < 1/2,
\end{align*}
where 
\begin{align} 
\label{def.ZjmL}
Z_{jm}^L(x,y) 
& := \sum_{i \in \mN} S_{\theta_i} [ P_m(x,p_i) - P_m(x,y)] \pa_x^{j-m} \ph_i^*(x),
\\
Z_{jm}^H(x,y) 
& := \sum_{i \in \mN} (I - S_{\theta_i}) [P_m(x,y)] \pa_x^{j-m} \ph_i^*(x).
\label{def.ZjmH}
\end{align}
\end{lemma}

\begin{proof} 
Let $x,y,j$ be like in the statement.  
Then, by Lemma \ref{lemma:cubi vicini}, 
\[
P_j(x,y) = P_j(x,y) \Big( \sum_{i \in \mN} \ph_i^*(x) \Big) 
= \sum_{i \in \mN} P_j(x,y) \ph_i^*(x)
\]
and, if $j$ is nonzero and $0 \leq m<j$,  
\[
0 = P_m(x,y) \Big( \sum_{i \in \mN} \pa_x^{j-m} \ph_i^*(x) \Big) 
= \sum_{i \in \mN} P_m(x,y) \pa_x^{j-m} \ph_i^*(x).
\]
Hence 
\begin{equation} \label{jm bino}
P_j(x,y) = \sum_{0 \leq m \leq j} \binom{j}{m} 
\sum_{i \in \mN} P_m(x,y) \pa_x^{j-m} \ph_i^*(x).
\end{equation}
Formula \eqref{jm bino} holds also for $j=0$. 
Splitting $I = S_{\theta_i} + (I - S_{\theta_i})$, 
we get
\begin{align} 
P_j(x,y) 
& = \sum_{0 \leq m \leq j} \binom{j}{m} 
\sum_{i \in \mN} S_{\theta_i} [ P_m(x,y) ] \pa_x^{j-m} \ph_i^*(x)
\notag \\ 
& \quad \ 
+ \sum_{0 \leq m \leq j} \binom{j}{m} 
\sum_{i \in \mN} (I - S_{\theta_i}) [ P_m(x,y) ] \pa_x^{j-m} \ph_i^*(x).
\label{jm banana}
\end{align}
We use identity \eqref{tutti.1} for $g^{(j)}(x)$ 
and \eqref{jm banana} for $P_j(x,y)$, and subtract. 
\end{proof}

\begin{lemma}
\emph{(Decomposition of $g^{(j)}(x)$ at $\dist(x,F) < 1/2$).}
\label{lemma:trick theta x}
One has 
\begin{align*}
& g^{(j)}(x)  
= G_j(x) + \sum_{0 \leq m < j} \binom{j}{m} \Big( Z_{jm}^L(x,y) + W_{jm}^H(x,y) \Big)
\\ 
& \forall |j| \leq k, \ \ 
y \in F, \ \ 
x \in \Om, \ \ 
\dist(x,F) < 1/2,
\end{align*}
where $Z_{jm}^L(x,y)$ is defined in Lemma \ref{lemma:ZjmLH},
\begin{align}
G_j(x) 
& := \sum_{i \in \mN} S_{\theta_i} [P_j(x,p_i)] \ph_i^*(x),
\label{def Gj} 
\\
W_{jm}^H(x,y) 
& := \sum_{i \in \mN} (S_{\theta_i} - S_{\theta_x}) [P_m(x,y)] \pa_x^{j-m} \ph_i^*(x),
\label{def WjmH} 
\\ 
\theta_x 
& := \Big( \frac{6}{\dist(x,F)} \Big)^\tau,
\label{def theta x}
\end{align}
and $\t$ is defined in \eqref{def theta}.
\end{lemma}

\begin{proof} 
Let $x,y,j$ be like in the statement.  
Separating $m=j$ from $m<j$ in formula \eqref{tutti.1}, one has  
\begin{equation} \label{fu.aux.Gj}
g^{(j)}(x) 
= G_j(x) + \sum_{0 \leq m < j} \binom{j}{m} 
\sum_{i \in \mN} S_{\theta_i} [P_m(x,p_i)] \pa_x^{j-m} \ph_i^*(x).
\end{equation}
For $0 \leq m<j$, by Lemma \ref{lemma:cubi vicini}, 
\[
0 = S_{\theta_x} [P_m(x,y)] \Big( \sum_{i \in \mN} \pa_x^{j-m} \ph_i^*(x) \Big) 
= \sum_{i \in \mN} S_{\theta_x} [P_m(x,y)] \pa_x^{j-m} \ph_i^*(x).
\]
Thus, splitting $S_{\theta_x} = S_{\theta_i} + (S_{\theta_x} - S_{\theta_i})$
and taking the sum over $0 \leq m < j$, one has 
\begin{align} 
0 & = \sum_{0 \leq m < j} \binom{j}{m} \sum_{i \in \mN} 
S_{\theta_i} [P_m(x,y)] \pa_x^{j-m} \ph_i^*(x)
\notag \\ & \quad \  
+ \sum_{0 \leq m < j} \binom{j}{m} \sum_{i \in \mN} 
(S_{\theta_x} - S_{\theta_i}) [P_m(x,y)] \pa_x^{j-m} \ph_i^*(x).
\label{aux.theta.x}
\end{align}
The difference of \eqref{fu.aux.Gj} and \eqref{aux.theta.x} 
gives the lemma. 
\end{proof}

\begin{remark} (\emph{Why two different decompositions 
in Lemmas \ref{lemma:ZjmLH}, \ref{lemma:trick theta x}}).
\label{rem:two decomp}
We discuss why we need both the decompositions introduced in the last two lemmas. 

The decomposition of Lemma \ref{lemma:ZjmLH} is the natural one, 
as it is the combination of the corresponding formula in \cite{S}
with the splitting $S_{\theta_i} + (I - S_{\theta_i})$ 
given by the smoothing operators adapted to the partition of unity, 
and it is obtained by the same ingredients 
used to define $g$ in subsection \ref{subsec:def g}. 
We use this decomposition in the proof of Lemma \ref{lemma:piano.1}.

In \cite{S} the estimate for the difference $g^{(j)}(x) - P_j(x,y)$ 
is used, with triangular inequality, 
to obtain directly (and almost trivially) a bound 
for $g^{(j)}(x)$, which corresponds to our Lemma \ref{lemma:piano.4}. 
However, trying to prove Lemma \ref{lemma:piano.4} in the same way 
(namely $\| g^{(j)}(x) \|_{\sigma_j} 
\leq \| g^{(j)}(x) - P_j(x,y) \|_{\sigma_j} + \| P_j(x,y) \|_{\sigma_j}$, 
and using the decomposition of Lemma \ref{lemma:ZjmLH} 
to estimate the difference $\| g^{(j)}(x) - P_j(x,y) \|_{\sigma_j}$), 
we encounter the following ``regularity trouble'':
for multi-indices $|m+\ell| > |j|$ one has $s_{m+\ell} < \sigma_j$, 
and the corresponding terms $(I - S_{\theta_i}) f^{(m+\ell)}(y)$ 
in $Z_{jm}^H(x,y)$ do not belong, in general, to $E_{\sigma_j}$
(the smoothing inequality \eqref{S2} does not help here, 
as it goes in the ``wrong'' direction).  

As a second natural attempt, then, we try to estimate $g^{(j)}(x)$ as it is 
(i.e., without adding and subtracting $P_j(x,y)$), 
but we encounter a ``power trouble'': 
for $|m+\ell| < |j|$ we have terms with small factors $|x-p_i|^{|\ell|}$ 
that are not sufficiently small 
to compensate the small denominators $q_i^{-|j|+|m|}$ 
coming from $\pa_x^{j-m} \ph_i^*(x)$.  

In other words, we face a problem of ``lack of regularity'' on the one side, 
and a problem of ``lack of smallness'' on the other side. 

We overcome this difficulty by introducing the decomposition 
in Lemma \ref{lemma:trick theta x}, which gives 
all the necessary smallness to compensate the small denominators 
(thanks to the fact that essentially we are still considering 
the difference between $g^{(j)}$ and its Taylor polynomial $P_j$)
and also has all the sufficient regularity to complete the estimate 
in the required norm (thanks to the smoothing operator $S_{\theta_x}$). 

Note that this troubles are not present 
in \cite{S} nor in the case of Definition \ref{def Lip Y},
where the norm $\| \ \|_Y$ is only one.

At this point one could wonder why not using 
such a useful smoothing $S_{\theta_x}$ 
from the beginning, already in the definition of $g$.
The reason is that $\theta_x$ in \eqref{def theta x} 
is defined in terms of $\dist(x,F)$,  
and, in general, the map $x \mapsto \dist(x,F)$ is merely Lipschitz,
while our goal is to construct an extension $g$ of $f$ 
that preserves the regularity in $x$; except for the case $k=0$,
such regularity is higher than just Lipschitz, 
therefore $\theta_x$ must be avoided to define $g$.
In fact, the decomposition of $\Om$ in dyadic cubes 
and the corresponding partition of unity is  
introduced in \cite{S} (and, before, by Whitney himself \cite{W-anal}) 
precisely to replace $\dist(x,F)$ with a smooth way 
to approach $F$ coming from $\Om$.  
\end{remark}

In the same way as in Lemma \ref{lemma:ZjmLH}, 
we give a decomposition of any partial derivative of $g^{(0)}(x)$ 
of order $k+1$.

\begin{lemma}
\emph{(Decomposition of $\pa_x^j g^{(0)}(x)$ or order $|j|=k+1$ 
at $\dist(x,F) < 1/2$).}
\label{lemma:pappa}
One has 
\begin{align*}
& \pa_x^j g^{(0)}(x) 
= \sum_{0 \leq m < j} \binom{j}{m} \Big( Z_{jm}^L(x,y) - Z_{jm}^H(x,y) \Big)
\\ 
& \forall |j| = k+1, \ \ 
y \in F, \ \ 
x \in \Om, \ \ 
\dist(x,F) < 1/2,
\end{align*}
where $Z_{jm}^L(x,y)$, $Z_{jm}^H(x,y)$ are defined 
by formulas \eqref{def.ZjmL}-\eqref{def.ZjmH}. 
\end{lemma}

\begin{proof}  
{[Similar to the proof of Lemma \ref{lemma:ZjmLH}]}
Let $x,y,j$ be like in the statement, and let $0 \leq m < j$  
(therefore $|m| \leq k$). 
Since $j-m$ is nonzero, by Lemma \ref{lemma:cubi vicini} one has
\[
0 = P_m(x,y) \Big( \sum_{i \in \mN} \pa_x^{j-m} \ph_i^*(x) \Big) 
= \sum_{i \in \mN} P_m(x,y) \pa_x^{j-m} \ph_i^*(x).
\]
Hence 
\begin{equation} \label{jm senza testa}
0 = \sum_{0 \leq m < j} \binom{j}{m} 
\sum_{i \in \mN} P_m(x,y) \pa_x^{j-m} \ph_i^*(x)
\end{equation}
and, splitting $I = S_{\theta_i} + (I - S_{\theta_i})$, 
\begin{align} 
0 & = \sum_{0 \leq m < j} \binom{j}{m} 
\sum_{i \in \mN} S_{\theta_i} [ P_m(x,y) ] \pa_x^{j-m} \ph_i^*(x)
\notag \\ 
& \quad \ 
+ \sum_{0 \leq m < j} \binom{j}{m} 
\sum_{i \in \mN} (I - S_{\theta_i}) [ P_m(x,y) ] \pa_x^{j-m} \ph_i^*(x).
\label{jm banana senza testa}
\end{align}
We use identity \eqref{der.k+1} for $\pa_x^{j} g^{(0)}(x)$ 
and \eqref{jm banana senza testa}, and subtract. 
\end{proof}

\subsection{Estimates at points close to $F$}
\label{subsec:est.close}

We prove some estimates at points $\dist(x,F) \leq 1/2$, 
and we begin with the additional condition $|x-y| \leq 2 \dist(x,F)$. 

\begin{lemma}
\emph{(Bound for $\| g^{(j)}(x) - P_j(x,y) \|_{\sigma_\rho}$ 
at $\dist(x,F) < 1/2$, $|x-y| \leq 2 \dist(x,F)$).}
\label{lemma:piano.1}
One has
\begin{align} 
& \| g^{(j)}(x) - P_j(x,y) \|_{\sigma_\rho} 
\lesssim KM |x-y|^{\rho-|j|} 
\notag \\ 
& \forall |j| \leq k, \ \ 
y \in F, \ \ 
x \in \Om, \ \ 
\dist(x,F) < 1/2, \ \ 
|x-y| \leq 2 \dist(x,F).
\label{est.Rj} 
\end{align}
\end{lemma}

\begin{proof} 
With the notation of Lemma \ref{lemma:ZjmLH}, 
we begin with estimating $Z_{jm}^L(x,y)$. 
By Lemma \ref{lemma:polinomi}, one has 
\begin{equation} \label{P-P j copia}
P_m(x, p_i) - P_m(x,y) = \sum_{|m+\ell| \leq k} 
\frac{1}{\ell!} R_{m+\ell}(p_i,y) (x-y)^\ell.
\end{equation}
Therefore 
\[
\| Z_{jm}^L(x,y) \|_{\sigma_\rho} 
\lesssim \sum_{i \in \mN} \sum_{|m+\ell| \leq k}  
\| S_{\theta_i} [ R_{m+\ell}(p_i, y)] \, \|_{\sigma_\rho} 
|x-y|^{|\ell|} |\pa_x^{j-m} \ph_i^*(x)|.
\]
By \eqref{S1}, \eqref{def K}, 
$\| S_{\theta_i} [ R_{m+\ell}(p_i, y)] \, \|_{\sigma_\rho} 
\leq K \| R_{m+\ell}(p_i, y) \|_{\sigma_\rho}$. 
Since $p_i, y$ are both in $F$, one has, 
by \eqref{0903.2.normalized}, 
\[
\| R_{m+\ell}(p_i, y) \|_{\sigma_\rho} 
\leq M |p_i - y|^{\rho - |m| - |\ell|},
\]
while $\pa_x^{j-m} \ph_i^*(x)$ is estimated by \eqref{POU.der},
and it is nonzero only if $x \in Q_i^*$. 
Hence 
\begin{equation} \label{aux.1707}
\| Z_{jm}^L(x,y) \|_{\sigma_\rho} 
\lesssim K M \sum_{\begin{subarray}{c} i \in \mN \\ x \in Q_i^* \end{subarray}} 
\sum_{|m+\ell| \leq k}  
|p_i - y|^{\rho - |m| - |\ell|}  
|x-y|^{|\ell|} q_i^{|m| - |j|}.
\end{equation}
By assumption, $|x-y| \lesssim \dist(x,F)$, 
and, by \eqref{dist.xF.q}, $\dist(x,F) \lesssim q_i$. 
Therefore $|x-y| \lesssim q_i$ and
\begin{equation} \label{aux.2407.1}
q_i^{|m| - |j|} \lesssim |x-y|^{|m| - |j|}
\end{equation}
because $|m| - |j| \leq 0$
(recall that $m \leq j$ by Lemma \ref{lemma:ZjmLH}).
By \eqref{dist.xpi.q}, $|p_i - x| \lesssim q_i$, 
and,  by \eqref{dist.xy.q}, $q_i \lesssim |x-y|$.
Therefore $|p_i - x| \lesssim |x-y|$ 
and, by triangular inequality, 
$|p_i - y| \leq |p_i - x| + |x-y| \lesssim |x-y|$. 
Hence 
\begin{equation} \label{aux.2407.2}
|p_i - y|^{\rho - |m| - |\ell|}  
\lesssim  |x - y|^{\rho - |m| - |\ell|}  
\end{equation}
because $\rho - |m| - |\ell| \geq 0$. 
By \eqref{aux.2407.1}, \eqref{aux.2407.2}, 
the general term in the sum \eqref{aux.1707} is 
$\lesssim |x-y|^{\rho - |j|}$.  
Moreover the sum \eqref{aux.1707} has at most $12^n$ terms because, by Proposition \ref{prop:cubes}, 
there are at most $12^n$ indices $i$ such that $x \in Q_i^*$.   
Therefore 
\begin{equation} \label{est.ZjmL.aux}
\| Z_{jm}^L(x,y) \|_{\sigma_\rho} 
\lesssim K M |x-y|^{\rho - |j|}.
\end{equation}

Now we estimate $Z_{jm}^H(x,y)$. 
By \eqref{def Pj}, 
then \eqref{S2}, \eqref{def K}, 
and then \eqref{def sj}, \eqref{def theta}, 
\eqref{0903.1.normalized},
one has 
\begin{align*}
| (I-S_{\theta_i}) [P_m(x,y)]\, \|_{\sigma_\rho} 
& \lesssim \sum_{|m+\ell| \leq k} 
\| (I-S_{\theta_i}) [ f^{(m+\ell)}(y)]\, \|_{\sigma_\rho} |x-y|^{|\ell|}
\\
& \lesssim K \sum_{|m+\ell| \leq k} 
\theta_i^{-(\sigma_{m+\ell}-\sigma_\rho)} \| f^{(m+\ell)}(y) \|_{\sigma_{m+\ell}} |x-y|^{|\ell|}
\\
& \lesssim KM \sum_{|m+\ell| \leq k} 
q_i^{\rho - |m| - |\ell|} |x-y|^{|\ell|}.
\end{align*}
If $\pa_x^{j-m} \ph_i^*(x) \neq 0$, 
then $x \in Q_i^*$, and, by \eqref{POU.der}, 
$|\pa_x^{j-m} \ph_i^*(x)| \lesssim q_i^{-|j|+|m|}$. 
Thus, recalling the definition of $Z_{jm}^H(x,y)$, we get  
\[
\| Z_{jm}^H(x,y) \|_{\sigma_\rho} 
\lesssim K M \sum_{\begin{subarray}{c} i \in \mN \\ x \in Q_i^* \end{subarray}} 
\sum_{|m+\ell| \leq k} q_i^{\rho - |\ell| - |j|} |x-y|^{|\ell|}.
\]
As observed above in this proof, we have 
$q_i \lesssim |x-y|$ and also $|x-y| \lesssim q_i$. 
The number of $i$ such that $x \in Q_i^*$ is at most $12^n$.
Hence we obtain 
\begin{equation} \label{est.ZjmH.aux}
\| Z_{jm}^H(x,y) \|_{\sigma_\rho} 
\lesssim KM |x-y|^{\rho - |j|}.
\end{equation}
By Lemma \ref{lemma:ZjmLH} and bounds \eqref{est.ZjmL.aux}, \eqref{est.ZjmH.aux}
we get the thesis.
\end{proof}

\begin{lemma}
\emph{(Bound for $\| g^{(j)}(x) - P_j(x,y) \|_{\sigma_\rho}$ 
with $y \in F$, $\dist(x,F) < 1/2$, $|x-y| > 2 \dist(x,F)$).}
\label{lemma:piano.3}
One has 
\begin{align} 
& \| g^{(j)}(x) - P_j(x,y) \|_{\sigma_\rho} 
\lesssim K M |x-y|^{\rho-|j|} 
\label{est.Rj.gen} 
\\ 
& \forall |j| \leq k, \ \ 
y \in F, \ \ 
x \in \Om, \ \ 
\dist(x,F) < 1/2, \ \ 
|x-y| > 2 \dist(x,F).
\notag 
\end{align}
\end{lemma}

\begin{proof}
Let $x,y,j$ be like in the statement. 
Consider a point $z \in F$ such that $|x-z| = \dist(x,F)$.  
Then Lemma \ref{lemma:piano.1} applies at the points $x,z$, 
and it gives the inequality 
\begin{equation} \label{z.pass.1}
\| g^{(j)}(x) - P_j(x,z) \|_{\sigma_\rho} 
\lesssim KM |x-z|^{\rho-|j|}.
\end{equation}
By Lemma \ref{lemma:polinomi}, 
\[
P_j(x, z) - P_j(x,y) = \sum_{|j+\ell| \leq k} \frac{1}{\ell!} R_{j+\ell}(z,y) (x-z)^\ell.
\]
Both $y,z \in F$, therefore, by \eqref{0903.2.normalized}, 
\[
\| P_j(x, z) - P_j(x,y) \|_{\sigma_\rho} 
\lesssim M \sum_{|j+\ell| \leq k} |z-y|^{\rho - |j| - |\ell|} |x-z|^{|\ell|}.
\]
Since $|x-z| = \dist(x,F) \leq |x-y|$, one has 
$|z-y| \leq |z-x| + |x-y| \lesssim |x-y|$,
whence
\begin{equation} \label{z.pass.2}
\| P_j(x, z) - P_j(x,y) \|_{\sigma_\rho} 
\lesssim M |x-y|^{\rho - |j|}.
\end{equation}
By \eqref{z.pass.1}, \eqref{z.pass.2} and triangular inequality 
we obtain \eqref{est.Rj.gen}.   
\end{proof}

\begin{lemma} \label{lemma:piano.4} 
\emph{(Bound for $\| g^{(j)}(x) \|_{\sigma_j}$ at $\dist(x,F) < 1/2$).}
One has
\begin{equation} \label{est.gj}
\| g^{(j)}(x) \|_{\sigma_j} \lesssim K_0 K M
\quad \ \forall |j| \leq k, \ \ x \in \Om, \ \ \dist(x,F) < 1/2.
\end{equation}
\end{lemma}

\begin{proof}
Let $x,j$ be like in the statement. 
Fix a point $y \in F$ such that 
\begin{equation} \label{y.special}
|x-y| = \dist(x,F),
\end{equation}
and consider the decomposition of $g^{(j)}(x)$ given in Lemma \ref{lemma:trick theta x}.
We estimate each term of such a decomposition separately.

\emph{Estimate of $G_j(x)$}. 
[Similar to ``case $|j| < |m|+|\ell|$'' 
in the proof of Lemma \ref{lemma:cubi intermedi sj}.]
By \eqref{def Gj}, \eqref{def Pj}, 
\[
\| G_j(x) \|_{\sigma_j} 
\lesssim \sum_{i \in \mN} \sum_{|j+\ell| \leq k} 
\| S_{\theta_i} [f^{(j+\ell)}(p_i)] \, \|_{\sigma_j} |x-p_i|^{|\ell|} \ph_i^*(x).
\]
Since $\sigma_j \geq \sigma_{j+\ell}$, 
by \eqref{S1}, \eqref{def K}, 
then \eqref{def sj}, \eqref{def theta},
\eqref{0903.1.normalized}, 
\begin{align*}
\| S_{\theta_i}[f^{(j+\ell)}(p_i)] \, \|_{\sigma_j} 
& \leq K \theta_i^{\sigma_j - \sigma_{j+\ell}} \| f^{(j+\ell)}(p_i) \|_{\sigma_{j+\ell}}
\leq K M q_i^{-|\ell|}.
\end{align*}
If $\ph_i^*(x) \neq 0$, 
then $x \in Q_i^*$, and therefore  
$|x-p_i|^{|\ell|} \lesssim q_i^{|\ell|}$
by \eqref{dist.xpi.q}.
Since $\sum_{i \in \mN} \ph_i^*(x) = 1$, we obtain 
\begin{equation} \label{aux.Gj}
\| G_j(x) \|_{\sigma_j} 
\lesssim KM.
\end{equation}

\emph{Estimate of $Z_{jm}^L(x,y)$.} 
[Similar to estimate of $\| Z_{jm}^L(x,y) \|_{\sigma_\rho}$ 
in the proof of Lemma \ref{lemma:piano.1}.]
By Lemma \ref{lemma:polinomi}, 
\[
\| Z_{jm}^L(x,y) \|_{\sigma_j} 
\lesssim \sum_{i \in \mN} \sum_{|m+\ell| \leq k}  
\| S_{\theta_i} [ R_{m+\ell}(p_i, y)] \, \|_{\sigma_j} 
|x-y|^{|\ell|} |\pa_x^{j-m} \ph_i^*(x)|.
\]
Since $\sigma_j > \sigma_\rho$, by \eqref{S1}, \eqref{def K}, 
\[
\| S_{\theta_i} [ R_{m+\ell}(p_i, y)] \, \|_{\sigma_j} 
\leq K \theta_i^{\sigma_j - \sigma_\rho} \| R_{m+\ell}(p_i, y) \|_{\sigma_\rho}
= K q_i^{|j|-\rho} \| R_{m+\ell}(p_i, y) \|_{\sigma_\rho}.
\] 
Since $p_i, y$ are both in $F$, one has, by \eqref{0903.2.normalized}, 
\[
\| R_{m+\ell}(p_i, y) \|_{\sigma_\rho} 
\leq M |p_i - y|^{\rho - |m| - |\ell|},
\]
while $\pa_x^{j-m} \ph_i^*(x)$ is estimated by \eqref{POU.der},
and it is nonzero only if $x \in Q_i^*$. 
Thus 
\begin{equation*} 
\| Z_{jm}^L(x,y) \|_{\sigma_j} 
\lesssim K M \sum_{\begin{subarray}{c} i \in \mN \\ x \in Q_i^* \end{subarray}} 
\sum_{|m+\ell| \leq k}  
q_i^{|m| - \rho}
|p_i - y|^{\rho - |m| - |\ell|}  
|x-y|^{|\ell|}.
\end{equation*}
By \eqref{y.special}, \eqref{dist.xF.q}, 
$|x-y| \lesssim q_i$. 
By \eqref{dist.xpi.q}, $|p_i - x| \lesssim q_i$. 
Hence $|p_i - y| \leq |p_i-x| + |x-y| \lesssim q_i$,
and 
\[
|p_i - y|^{\rho - |m| - |\ell|}  
\lesssim q_i^{\rho - |m| - |\ell|},
\quad \   
|x-y|^{|\ell|} \lesssim q_i^{|\ell|}.
\]
By Proposition \ref{prop:cubes}, 
there are at most $12^n$ indices $i$ such that $x \in Q_i^*$. 
Therefore 
\begin{equation} \label{aux.ZjmL}
\| Z_{jm}^L(x,y) \|_{\sigma_j} 
\lesssim KM.
\end{equation}

\emph{Estimate of $W_{jm}^H(x,y)$.}
By \eqref{def WjmH}, \eqref{def Pj},  
\[
\| W_{jm}^H(x,y) \|_{\sigma_j}  
\lesssim \sum_{i \in \mN} \sum_{|m+\ell| \leq k} 
\| (S_{\theta_i} - S_{\theta_x}) [f^{(m+\ell)}(y)] \, \|_{\sigma_j} 
|x-y|^{|\ell|} |\pa_x^{j-m} \ph_i^*(x)|.
\]
By \eqref{y.special}, \eqref{dist.xF.q}, 
$|x-y| \lesssim q_i$; 
by \eqref{POU.der}, we get
\[
\| W_{jm}^H(x,y) \|_{\sigma_j}  
\lesssim \sum_{\begin{subarray}{c} i \in \mN \\ x \in Q_i^* \end{subarray}} 
\sum_{|m+\ell| \leq k} 
\| (S_{\theta_i} - S_{\theta_x}) [f^{(m+\ell)}(y)] \, \|_{\sigma_j} 
q_i^{|\ell|+|m| - |j|}.
\]
By \eqref{0303.4}, \eqref{def K}, one has 
\[
\| (S_{\theta_i} - S_{\theta_x}) [f^{(m+\ell)}(y)] \, \|_{\sigma_j} 
\leq K_0 K \max \{ \theta_i^{\sigma_j - \sigma_{m+\ell}} , 
\theta_x^{\sigma_j - \sigma_{m+\ell}} \} 
\| f^{(m+\ell)}(y) \|_{\sigma_{m+\ell}}
\]
(no matter whether $\sigma_j$ is larger, smaller or equal to $\sigma_{m+\ell}$). 
By \eqref{def theta}, \eqref{def theta x}, \eqref{dist.xF.q}, 
one has
\[
\theta_i^\delta 
\leq \theta_x^\delta 
\leq 12 \theta_i^\delta. 
\]
Hence, recalling \eqref{def sj},  
\begin{align*}
\max \{ \theta_i^{\sigma_j - \sigma_{m+\ell}} , \theta_x^{\sigma_j - \sigma_{m+\ell}} \} 
= \max \{ \theta_i^{\d(|m|+|\ell|-|j|)} , \theta_x^{\d(|m|+|\ell|-|j|)} \} 
& \lesssim \theta_i^{\d(|m|+|\ell|-|j|)} 
\\ & 
\lesssim q_i^{|j|-|m|-|\ell|}.
\end{align*}
Moreover, since $y \in F$, by \eqref{0903.1.normalized},  
$\| f^{(m+\ell)}(y) \|_{\sigma_{m+\ell}} \leq M$. 
By Proposition \ref{prop:cubes}, the number of indices $i \in \mN$ 
such that $x \in Q_i^*$ is at most $12^n$. Thus we obtain 
\begin{equation} \label{aux.WjmH}
\| W_{jm}^H(x,y) \|_{\sigma_j}  
\lesssim K_0 K M.
\end{equation}

The sum of \eqref{aux.Gj}, \eqref{aux.ZjmL}, \eqref{aux.WjmH} 
gives \eqref{est.gj}.
\end{proof}

\begin{lemma}
\emph{(Bound for $\| \pa_x^j g^{(0)}(x) \|_{\sigma_\rho}$ of order $|j|=k+1$ 
at $\dist(x,F) < 1/2$).}
\label{lemma:der k+1 close}
One has 
\[
\| \pa_x^j g^{(0)}(x) \|_{\sigma_\rho} 
\lesssim K M (\dist(x,F))^{\rho-k-1}
\quad \ \forall |j| = k+1, \ \ x \in \Om, \ \ \dist(x,F) < 1/2.
\]
\end{lemma}

\begin{proof} 
{[Similar to the proof of Lemma \ref{lemma:piano.1}]}
Let $x,j$ be like in the statement, and fix a point $y \in F$ such that 
$|x-y| = \dist(x,F)$. 
We consider the decomposition in Lemma \ref{lemma:pappa}
and we estimate each term $Z_{jm}^L(x,y)$, $Z_{jm}^H(x,y)$ separately. 
Following word-by-word the proof of Lemma \ref{lemma:piano.1}, 
we obtain the bounds \eqref{est.ZjmL.aux}, \eqref{est.ZjmH.aux}
also in the present case. 
Here $|j| = k+1$ and $|x-y| = \dist(x,F)$, therefore 
$|x-y|^{\rho - |j|} = (\dist(x,F))^{\rho-k-1}$, 
and the lemma is proved.
\end{proof}

\subsection{Estimates on $\R^n$}
\label{subsec:patch}

In the next three lemmas we patch the estimates proved in the previous subsections.

\begin{lemma}
\emph{(Bound for $\| g^{(j)}(x) - P_j(x,y) \|_{\sigma_\rho}$ 
with $y \in F$, $x \in \R^n$).}
\label{lemma:g-P.global}
One has 
\begin{align} 
& \| g^{(j)}(x) - P_j(x,y) \|_{\sigma_\rho} 
\lesssim K M |x-y|^{\rho-|j|} 
\label{est.Rj.global} 
\quad \ 
\forall |j| \leq k, \ \ 
y \in F, \ \ 
x \in \R^n.
\notag 
\end{align}
\end{lemma}

\begin{proof} 
Let $y \in F$. 
If $x \in F$, the inequality 
holds by \eqref{1003.1}, \eqref{0903.2.normalized}.
If $x \in \Om$, with $\dist(x,F) \geq 1/2$, 
the inequality holds by Lemma \ref{lemma:Rj lontani}. 
If $x \in \Om$, with $\dist(x,F) < 1/2$ and $|x-y| \leq 2 \dist(x,F)$, 
it holds by Lemma \ref{lemma:piano.1}. 
If $x \in \Om$, with $\dist(x,F) < 1/2$ and $|x-y| > 2 \dist(x,F)$, 
it holds by Lemma \ref{lemma:piano.3}. 
\end{proof}

\begin{lemma}
\emph{(Bound for $\| g^{(j)}(x) \|_{\sigma_j}$ with $x \in \R^n$).}
\label{lemma:g.global} 
One has
\begin{equation} \label{est.gj.global}
\| g^{(j)}(x) \|_{\sigma_j} \lesssim K_0 K M
\quad \ \forall |j| \leq k, \ \ x \in \R^n, 
\end{equation}
where $K_0, K$ are defined in \eqref{def K}.
\end{lemma}

\begin{proof}
For $x \in F$ one has $\| g^{(j)}(x) \|_{\sigma_j} \leq M$ 
by \eqref{0903.1.normalized}, \eqref{1003.1}.
For $x \in \Om$ with $\dist(x,F) \geq 1/2$ 
one has $\| g^{(j)}(x) \|_{\sigma_j} \lesssim K M$  
by Lemmas \ref{lemma:cubi lontani}, \ref{lemma:cubi intermedi sj}. 
For $x \in \Om$ with $\dist(x,F) < 1/2$, 
the inequality holds by Lemma \ref{lemma:piano.4}.
\end{proof}

\begin{lemma}
\emph{(Bound for $\| \pa_x^j g^{(0)}(x) \|_{\sigma_\rho}$ 
of order $|j|=k+1$ with $x \in \Om$).}
\label{lemma:der.k+1.global}
One has 
\[
\| \pa_x^j g^{(0)}(x) \|_{\sigma_\rho} 
\lesssim K M (\dist(x,F))^{\rho-k-1}
\quad \ \forall |j| = k+1, \ \ x \in \Om.
\]
\end{lemma}

\begin{proof} 
For $x \in \Om$ with $\dist(x,F) < 1/2$ the estimate is given by Lemma 
\ref{lemma:der k+1 close}.
For $1/2 \leq \dist(x,F) \leq 6$ one has 
\[
\| \pa_x^j g^{(0)}(x) \|_{\sigma_\rho} 
\lesssim KM
\lesssim KM (\dist(x,F))^{\rho-k-1},
\]
where the first inequality holds by Lemma \ref{lemma:der.k+1.far},
and the second one holds because 
$-1 < \rho-k-1 \leq 0$ and therefore
\[
6^{-1} \leq 6^{\rho-k-1}
\leq (\dist(x,F))^{\rho-k-1}. 
\]
Finally, in the open set $\{ x \in \Om : \dist(x,F) > 6 \}$ 
the function $g^{(0)}$ is identically zero 
by Lemma \ref{lemma:cubi lontani}, 
therefore $\pa_x^j g^{(0)}$ vanishes on that set.
\end{proof}

\subsection{Estimates of Taylor remainders in $\R^n$}
\label{subsec:Rj}

Our goal is to prove that, for some constant $C$,
\begin{alignat}{2} 
\| g^{(j)}(x) \|_{\sigma_j} 
& \leq CM
\quad \ && \forall |j| \leq k, \ \ x \in \R^n,
\label{goal.1}
\\
\| R_j(x,y ; g) \|_{\sigma_\rho} 
& \leq CM |x-y|^{\rho - |j|}
\quad \ && \forall |j| \leq k, \ \ x,y \in \R^n,
\label{goal.2}
\end{alignat}
where $R_j(x,y;g)$ is the Taylor remainder of $g^{(j)}$, 
namely
\begin{equation} \label{def Pj Rj g}
R_j(x,y;g) := g^{(j)}(x) - P_j(x,y;g), 
\quad \ 
P_j(x,y;g) := \sum_{|j+\ell| \leq k} \frac{1}{\ell!} g^{(j+\ell)}(y) (x-y)^\ell.
\end{equation}
Bounds \eqref{goal.1}-\eqref{goal.2} correspond to 
conditions \eqref{0903.1.normalized}-\eqref{0903.2.normalized} 
with $f,F$ replaced by $g,\R^n$. 

Inequality \eqref{goal.1} is given by Lemma \ref{lemma:g.global}.
Inequality \eqref{goal.2} restricted to the case $y \in F$ 
is given by Lemma \ref{lemma:g-P.global}, because, for $y \in F$, 
one has $g^{(\ell)}(y) = f^{(\ell)}(y)$ and therefore $P_j(x,y;g) = P_j(x,y)$.
Thus it remains to prove the inequality in \eqref{goal.2} for $y \in \Om$.

\begin{lemma}
\emph{(Bound for $\| R_j(x,y;g) \|_{\sigma_\rho}$ with $y \in \Om$, 
$\dist(L,F) \leq |x-y|$).}
\label{lemma:segment.close}
One has 
\[
\| R_j(x,y;g) \|_{\sigma_\rho} \lesssim KM |x-y|^{\rho-|j|} 
\quad \ \forall |j| \leq k, \ \ y \in \Om, \ \ x \in \R^n, \ \ 
\dist(L,F) \leq |x-y|,
\]
where $L := \{ y+\lm(x-y) : \lm \in [0,1] \}$ 
is the segment in $\R^n$ of endpoints $x,y$.
\end{lemma}

\begin{proof}
Let $x,y,j$ be like in the statement. 
Fix two points $z,p$ such that 
\[
z \in L, \quad 
p \in F, \quad
\dist(L,F) = |z-p|. 
\]
Since $z \in L$ one has 
$|x-z| \leq |x-y|$ 
and $|y-z| \leq |x-y|$; 
also, $|z-p| = \dist(L,F) \leq |x-y|$ by assumption. 
Therefore 
\begin{equation} \label{xy.zp}
\begin{aligned}
|x-p| \leq |x-z| + |z-p| \leq 2|x-y|, 
\\
|y-p| \leq |y-z| + |z-p| \leq 2|x-y|.
\end{aligned}
\end{equation}
Recalling the definition \eqref{def Pj Rj g} of $R_j(x,y;g)$,
adding and subtracting $P_j(x,p)$, 
one has 
\begin{equation} \label{aux.gP.PP}
\\| R_j(x,y;g) \|_{\sigma_\rho} 
\leq \| g^{(j)}(x) - P_j(x,p) \|_{\sigma_\rho} 
+ \| P_j(x,p) - P_j(x,y;g) \|_{\sigma_\rho}.
\end{equation} 
We estimate these terms separately.

\emph{Estimate of $\| g^{(j)}(x) - P_j(x,p) \|_{\sigma_\rho}$.} 
Since $p \in F$, by Lemma \ref{lemma:g-P.global} one has 
\[
\| g^{(j)}(x) - P_j(x,p) \|_{\sigma_\rho} 
\lesssim K M |x-p|^{\rho - |j|}.
\]
Also, $|x-p|^{\rho - |j|} \lesssim |x-y|^{\rho - |j|}$
because $|x-p| \lesssim |x-y|$ (see \eqref{xy.zp}) 
and $\rho - |j| \geq 0$. Hence 
\begin{equation} \label{aux.gP.25.1}
\| g^{(j)}(x) - P_j(x,p) \|_{\sigma_\rho} 
\lesssim K M |x-y|^{\rho - |j|}.
\end{equation}

\emph{Estimate of $\| P_j(x,p) - P_j(x,y;g) \|_{\sigma_\rho}$.}
Since $p \in F$, one has $P_j(x,p) = P_j(x,p;g)$. 
Lemma \ref{lemma:polinomi} (applied to $(g^{(\ell)}, \R^n)$ 
instead of $(f^{(\ell)}, F)$) gives 
\begin{equation} \label{P-P j segment}
P_j(x, y;g) - P_j(x,p;g) = \sum_{|j+\ell| \leq k} 
\frac{1}{\ell!} R_{j+\ell}(y,p;g) (x-y)^\ell.
\end{equation}
Since $p \in F$, one has 
\[
R_{j+\ell}(y,p;g) 
= g^{(j+\ell)}(y) - P_{j+\ell}(y,p;g)
= g^{(j+\ell)}(y) - P_{j+\ell}(y,p),
\]
and therefore, by Lemma \ref{lemma:g-P.global},
\begin{equation} \label{aux.25.2}
\| R_{j+\ell}(y,p;g) \|_{\sigma_\rho} 
= \| g^{(j+\ell)}(y) - P_{j+\ell}(y,p) \|_{\sigma_\rho} 
\lesssim KM |y-p|^{\rho-|j+\ell|}. 
\end{equation}
From \eqref{P-P j segment}, \eqref{aux.25.2} we deduce that
\[
\| P_j(x, y;g) - P_j(x,p;g) \|_{\sigma_\rho} 
\lesssim K M \sum_{|j+\ell| \leq k} |y-p|^{\rho-|j|-|\ell|} |x-y|^{|\ell|}.
\]
Now 
\[
|y-p|^{\rho-|j|-|\ell|} \lesssim |x-y|^{\rho-|j|-|\ell|}
\] 
because $|y-p| \lesssim |x-y|$ (see \eqref{xy.zp}) 
and $\rho-|j|-|\ell| \geq 0$. 
Hence, recalling that $P_j(x,p;g)$ $= P_j(x,p)$, we obtain 
\begin{equation} \label{aux.gP.25.3}
\| P_j(x,p) - P_j(x,y;g) \|_{\sigma_\rho} 
\lesssim KM |x-y|^{\rho-|j|}.
\end{equation}

By \eqref{aux.gP.PP}, \eqref{aux.gP.25.1}, \eqref{aux.gP.25.3}, 
the proof is complete.
\end{proof}

To extend the bound of Lemma \ref{lemma:segment.close} 
to the case $\dist(L,F) > |x-y|$, 
we need the following (classical) identity for the first order partial derivatives 
of Taylor's polynomials with respect to their center. 
From that formula we deduce a uniform bound for Taylor's remainder.

\begin{lemma}
\emph{(Formula for the gradient of $P_j(x,z;g)$ 
with respect to $z$ in $\Om$).}
\label{lemma:gradient}
One has 
\begin{align*}
& \pa_z^m ( P_j(x,z;g)) = \sum_{|j+m+\ell| = k+1} 
\frac{1}{\ell!} (\pa_z^{j+m+\ell} g^{(0)}(z)) \, (x-z)^\ell
\\ 
& \forall |j| \leq k, \ \ 
|m|=1, \ \ 
z \in \Om, \ \ 
x \in \R^n.
\end{align*} 
\end{lemma}

\begin{proof}
Let $x,z,j,m$ be like in the statement. 
By definition \eqref{def Pj Rj g},
\[
P_j(x,z;g) = \sum_{|j+\ell| \leq k} \frac{1}{\ell!} g^{(j+\ell)}(z) (x-z)^\ell,
\]
and its first derivative is
\begin{align}
\pa_z^m( P_j(x,z;g) ) 
& = \sum_{|j+\ell| \leq k} \frac{1}{\ell!} \{ \pa_z^m g^{(j+\ell)}(z) \} (x-z)^\ell
+ \sum_{|j+\ell| \leq k} g^{(j+\ell)}(z) \Big( \pa_z^m \frac{(x-z)^\ell}{\ell!} \Big).
\label{aux.25.5}
\end{align}
By definition \eqref{def gj}, 
one has $\pa_z^m g^{(j+\ell)}(z) = \pa_z^{j+m+\ell} g^{(0)}(z)$ 
for all $\ell$ such that $|j+\ell| \leq k$, 
and also 
$\pa_z^{j+m+\ell} g^{(0)}(z) = g^{(j+m+\ell)}(z)$ 
for all $\ell$ such that $|j+m+\ell| \leq k$. 
Thus 
\begin{align}
\sum_{|j+\ell| \leq k} \frac{1}{\ell!} \{ \pa_z^m g^{(j+\ell)}(z) \} (x-z)^\ell
& = \sum_{|j+m+\ell| \leq k} \frac{1}{\ell!} g^{(j+m+\ell)}(z) (x-z)^\ell
\notag \\ & \quad \ 
+ \sum_{|j+m+\ell| = k+1} \frac{1}{\ell!} (\pa_z^{j+m+\ell} g^{(0)}(z)) \, (x-z)^\ell.
\label{aux.25.6}
\end{align}
Recalling that $|m|=1$, one has 
\[
\pa_z^m \Big( \frac{(x-z)^\ell}{\ell!} \Big) 
= - \frac{(x-z)^{\ell-m}}{(\ell-m)!} \quad \ \text{if } \ell \geq m,
\]
while $\pa_z^m ((x-z)^\ell) = 0$ if $\ell$ is not $\geq m$. 
Hence 
\begin{align}
\sum_{|j+\ell| \leq k} g^{(j+\ell)}(z) \Big( \pa_z^m \frac{(x-z)^\ell}{\ell!} \Big)
& = - \sum_{\begin{subarray}{c} |j+\ell| \leq k \\ \ell - m \geq 0 \end{subarray}} 
g^{(j+\ell)}(z) \frac{(x-z)^{\ell-m}}{(\ell-m)!}
\notag \\ & 
= - \sum_{|j+m+\ell| \leq k} 
g^{(j+m+\ell)}(z) \frac{(x-z)^{\ell}}{\ell !},
\label{aux.25.7}
\end{align}
where in the last identity we have made 
the change of summation variable $\ell = m + \ell'$ 
and renamed $\ell' \to \ell$.
Inserting \eqref{aux.25.6}, \eqref{aux.25.7} into \eqref{aux.25.5} 
we get the thesis.
\end{proof}

\begin{lemma}
\emph{(Bound for $\| R_j(x,y;g) \|_{\sigma_\rho}$ with $y \in \Om$, $\dist(L,F) > |x-y|$).}
\label{lemma:segment.small}
One has 
\[
\| R_j(x,y;g) \|_{\sigma_\rho} \lesssim K M |x-y|^{\rho-|j|} 
\quad \ \forall |j| \leq k, \ \ 
\dist(L,F) > |x-y|,
\]
where $L$ is the segment in $\R^n$ of endpoints $x,y$.
\end{lemma}

\begin{proof} 
Let $x,y,j$ be like in the statement. 
The assumption $\dist(L,F) > |x-y|$ implies that $L \subset \Om$,
because every point $z \in L$ satisfies 
$\dist(z,F) \geq \dist(L,F) > |x-y| \geq 0$, 
namely $\dist(z,F) > 0$, therefore $z \in \Om$. 
Moreover for $x=y$ the lemma trivially holds, 
hence we assume that $|x-y| > 0$. 

For all $z \in \Om$, let 
\begin{equation} \label{def hz}
h(z) := P_j(x,z;g)
= \sum_{|j+\ell| \leq k} \frac{1}{\ell !} g^{(j+\ell)}(z) (x-z)^\ell.
\end{equation} 
As observed in subsection \ref{subsec:def g}, 
$g^{(0)}$ is $C^\infty(\Om, E_a)$ for all $a \in \mI$. 
Hence $g^{(j+\ell)} = \pa_z^{j+\ell} g^{(0)}$ is $C^\infty(\Om, E_a)$,
and the function $h$ is also $C^\infty(\Om, E_a)$. 
Moreover
\[
h(x) = P_j(x,x;g) = g^{(j)}(x), 
\quad \ 
h(y) = P_j(x,y; g),
\quad \ 
R_j(x,y;g) = h(x) - h(y).
\]
Since the segment $L$ is contained in $\Om$ 
and since $h \in C^{\infty}(\Om, E_{\sigma_\rho})$, 
by the mean value theorem one has 
\begin{equation} \label{MVT}
\| h(x) - h(y) \|_{\sigma_\rho} 
\leq \sup_{\begin{subarray}{c} z \in L \\ |m|=1 \end{subarray}} 
\| \pa_z^m h (z) \|_{\sigma_\rho} |x-y|.
\end{equation}
By definition \eqref{def hz} and Lemma \ref{lemma:gradient}, 
\[
\| \pa_z^m h(z) \|_{\sigma_\rho} 
= \| \pa_z^m ( P_j(x,z;g)) \|_{\sigma_\rho} 
\lesssim \sum_{|j+m+\ell| = k+1} 
\| \pa_z^{j+m+\ell} g^{(0)}(z) \|_{\sigma_\rho} |x-z|^{|\ell|}.
\]
By Lemma \ref{lemma:der.k+1.global}, 
\[
\| \pa_z^{j+m+\ell} g^{(0)}(z) \|_{\sigma_\rho} 
\lesssim KM (\dist(z,F))^{\rho-k-1} 
\]
for all $z \in L$, $|m|=1$, $|j+m+\ell| = k+1$. 
One has 
\[
(\dist(z,F))^{\rho-k-1} \leq |x-y|^{\rho-k-1}
\]
because $\dist(z,F) \geq \dist(L,F) > |x-y|$ and $\rho-k-1 \leq 0$. 
Also, 
\[
|x-z|^{|\ell|} 
\leq |x-y|^{|\ell|} 
= |x-y|^{k - |j|}
\]
because $|x-z| \leq |x-y|$ for all $z \in L$, 
and $|\ell| = k - |j|$ for $|j+m+\ell| = k+1$ (recall that $|m|=1$). 
As a consequence we obtain
\[
\sup_{\begin{subarray}{c} z \in L \\ |m|=1 \end{subarray}} 
\| \pa_z^m h(z) \|_{\sigma_\rho} 
\lesssim K M |x-y|^{\rho - |j|- 1},
\]
and by \eqref{MVT} the thesis follows.
\end{proof}

Patching Lemma \ref{lemma:segment.close} and Lemma \ref{lemma:segment.small}, 
we obtain property \eqref{goal.2}.

\begin{lemma}
\emph{(Bound for $\| R_j(x,y;g) \|_{\sigma_\rho}$ with $x,y \in \R^n$).}
\label{lemma:Rj.global}
One has 
\[
\| R_j(x,y;g) \|_{\sigma_\rho} \lesssim K M |x-y|^{\rho-|j|} 
\quad \ \forall |j| \leq k, \ \  x,y \in \R^n.
\]
\end{lemma}

\begin{proof}
For $y \in F$, $x \in \R^n$, the inequality holds by Lemma \ref{lemma:g-P.global}. 
For $y \in \Om$, $x \in \R^n$, with $\dist(L,F) \leq |x-y|$ 
($L$ being the segment of endpoints $x,y$), 
the inequality holds by Lemma \ref{lemma:segment.close}. 
For $y \in \Om$, $x \in \R^n$, with $\dist(L,F) > |x-y|$, 
it holds by Lemma \ref{lemma:segment.small}. 
\end{proof}

\subsection{Conclusion of the proof} 
\label{subsec:conclusion}

In the next lemma we summarize what we have proved so far
under the assumptions listed in subsection \ref{subsec:hyp};
in fact, this is Theorem \ref{thm:WET sigma} in the case $\g=1$.

\begin{lemma} \emph{(Extension in the case $\g=1$).}
\label{lemma:summarize.normalized}
Given $f$ $\in \Lip(\rho,F,\sigma,1,\d)$,
there exists $g \in \Lip(\rho,\R^n,\sigma,1,\d)$ 
(defined in subsection \ref{subsec:def g}) 
that coincides with $f$ on $F$ and satisfies 
\begin{equation} \label{op.norm.fg}
\| g \|_{\Lip(\rho, \R^n, \sigma, 1, \d)} 
\leq C \| f \|_{\Lip(\rho, F, \sigma, 1, \d)}
\end{equation}
with $C = C' K_0 K$, where $C'$ depends only on $k,n$. 
The function $g^{(0)} : \R^n \to E_{\sigma_\rho}$ is differentiable $k$ times
at every point $x \in \R^n$, with partial derivatives 
$\pa_x^j g^{(0)}(x) = g^{(j)}(x)$. 
Moreover $g^{(j)}(x) \in E_\infty$ for all $x \in \Om$,  
and $g^{(j)} \in C^\infty(\Om,E_a)$ for all $a \in \mI$. 
The mapping $\Lip(\rho,F,\sigma,1,\d) \to \Lip(\rho,\R^n,\sigma,1,\d)$,
$f \mapsto g$ is a bounded linear operator of norm $\leq C$,
depending on $k, F, \g, \d$ and the family $(S_\theta)_{\theta \geq 1}$,  
and independent of $\rho,\sigma$.
\end{lemma}

\begin{proof}
Let $f \in \Lip(\rho,F,\sigma,1,\d)$, 
and let $M \geq 0$ be a constant such that 
\eqref{0903.1.normalized}-\eqref{0903.2.normalized} hold.
By Lemma \ref{lemma:g.global} and Lemma \ref{lemma:Rj.global}, 
there exists a constant $C'>0$ depending only on $k,n$
such that the collection $g$ defined in subsection \ref{subsec:def g} 
satisfies \eqref{goal.1}-\eqref{goal.2} with $C = C' K_0 K$. 
Hence, recalling Definition \ref{def Lip sigma}, 
$g$ belongs to $\Lip(\rho, \R^n, \sigma, 1, \d)$, 
and, taking the inf over all constants $M$ 
such that \eqref{0903.1.normalized}-\eqref{0903.2.normalized} hold,
we get \eqref{op.norm.fg}. 
The inequality in Lemma \ref{lemma:Rj.global} also implies that 
the function $g^{(0)} : \R^n \to E_{\sigma_\rho}$ is $k$ times differentiable 
at every point $x \in \R^n$, with partial derivatives 
$\pa_x^j g^{(0)} (x) = g^{(j)}(x)$ for all $|j| \leq k$, all $x \in \R^n$.

By construction (see subsection \ref{subsec:def g}), 
$g$ coincides with $f$ on $F$,
$g^{(j)}(x) \in E_\infty$ for all $x \in \Om$,  
and $g^{(j)} \in C^\infty(\Om,E_a)$ for all $a \in \mI$.  
Moreover, by construction, the mapping $f \mapsto g$ is linear.
The definition of $g$ in subsection \ref{subsec:def g} 
involves $f, k, F, \g, \d$ and the family $(S_\theta)_{\theta \geq 1}$,  
and does not directly involve $\rho, \sigma$. 
\end{proof}

Now we complete the proof of the main theorem. 

\begin{proof}[\emph{\textbf{Proof of Theorem \ref{thm:WET sigma}}}] 
[The proof follows from Lemma \ref{lemma:summarize.normalized} 
by an elementary rescaling argument.] 
Assume the hypotheses of Theorem \ref{thm:WET sigma}.
In particular, let $f \in \Lip(\rho,F,\sigma,\g,\d)$ 
(here $\g$ is any positive real number), and let $M \geq 0$ be a constant such that 
\eqref{0903.1-2} holds.
We define another closed subset of $\R^n$,
\begin{equation}
\label{def F tilde}
\tilde F := \{ x \in \R^n : \g x \in F \} 
= \{ \g^{-1} y : y \in F \}, 
\end{equation}
and another collection $\tilde f := \{ \tilde f^{(j)} : |j| \leq k \}$
of functions, 
\begin{equation}
\label{def f tilde}
\tilde f^{(j)}(x) := \g^{|j|} f^{(j)}( \g x) 
\quad \ \forall |j| \leq k, \ \ x \in \tilde F.
\end{equation} 
By \eqref{def f tilde}, \eqref{def F tilde}, \eqref{0903.1-2},
one has 
\begin{equation}
\label{ok f tilde.1}
\| \tilde f^{(j)}(x) \|_{\sigma_j} 
= \g^{|j|} \| f^{(j)}( \g x) \|_{\sigma_j} 
\leq M
\quad \ \forall |j| \leq k, \ \ x \in \tilde F.
\end{equation}
For $x,y \in \tilde F$, let $R_j(x,y;\tilde f)$ 
denote the Taylor remainders of $\tilde f$, namely 
\begin{equation} \label{def Rj f tilde}
R_j(x,y;\tilde f) := \tilde f^{(j)}(x) - \sum_{|j+\ell| \leq k} 
\frac{1}{\ell!} \tilde f^{(j+\ell)}(y) (x-y)^\ell.
\end{equation}
For all $\ell \in \N^n$, all $x,y \in \R^n$ 
one has $(x-y)^\ell = \g^{-|\ell|} (\g x - \g y)^\ell$; 
therefore, by \eqref{def f tilde}, \eqref{def Rj f tilde},
\begin{align} 
R_j(x,y;\tilde f) 
& = \g^{|j|} f^{(j)}(\g x) 
- \sum_{|j+\ell| \leq k} \frac{1}{\ell!} \g^{|j|+|\ell|} f^{(j+\ell)}(\g y) 
\g^{-|\ell|} (\g x-\g y)^\ell
\notag \\ & 
= \g^{|j|} R_j(\g x, \g y;f)
\qquad  \forall |j| \leq k, \ \ x,y \in \tilde F
\label{Rj Rj tilde}
\end{align}
where $R_j(x,y;f) := R_j(x,y)$ is defined in \eqref{def Pj Rj}.
By \eqref{Rj Rj tilde}, \eqref{0903.1-2}, 
one has
\begin{align}
\| R_j(x,y;\tilde f) \|_{\sigma_\rho} 
& = \g^{|j|} \| R_j(\g x, \g y;f) \|_{\sigma_\rho}
\notag \\ & 
\leq \g^{|j|} \g^{-\rho} M |\g x - \g y|^{\rho - |j|}
= M |x-y|^{\rho - |j|}
\quad \  \forall |j| \leq k, \ \ x,y \in \tilde F.
\label{ok f tilde.2}
\end{align}
By \eqref{ok f tilde.1}, \eqref{ok f tilde.2}, 
$\tilde f$ belongs to $\Lip(\rho, \tilde F, \sigma, 1, \d)$ 
and Lemma \ref{lemma:summarize.normalized} can be applied to $\tilde f$.

Let $\tilde g \in \Lip(\rho, \R^n, \sigma, 1, \d)$ 
be the extension of $\tilde f$ given by Lemma \ref{lemma:summarize.normalized}.
Thus $\tilde g$ satisfies
\begin{equation} 
\label{ok g tilde.1}
\| \tilde g^{(j)}(x) \|_{\sigma_j} \leq C M,
\quad \ 
\| R_j(x,y;\tilde g) \|_{\sigma_\rho} \leq C M |x-y|^{\rho-|j|} 
\quad \  
\forall |j| \leq k, \ \ x,y \in \R^n,
\end{equation}
with $C$ given by Lemma \ref{lemma:summarize.normalized}.
We define 
\begin{equation}
\label{def gj gj tilde}
g^{(j)}(x) := \g^{-|j|} \tilde g^{(j)}(\g^{-1} x) 
\quad \ \forall |j| \leq k, \ \ x \in \R^n.
\end{equation}
For all $x \in F$ one has $\g^{-1} x \in \tilde F$, 
and therefore 
\begin{equation} \label{bic.0} 
g^{(j)}(x) 
= \g^{-|j|} \tilde g^{(j)}(\g^{-1} x) 
= \g^{-|j|} \tilde f^{(j)}(\g^{-1} x) 
= f^{(j)}(x) 
\quad \ \forall |j| \leq k, \ \ x \in F
\end{equation}
by \eqref{def gj gj tilde}, \eqref{def f tilde} 
and because $\tilde g = \tilde f$ on $\tilde F$
by Lemma \ref{lemma:summarize.normalized}.
Thus $g$ is an extension of $f$ to $\R^n$. 
Moreover
\begin{equation} \label{bic.1} 
\g^{|j|} \| g^{(j)}(x) \|_{\sigma_j} 
= \| \tilde g^{(j)}(\g^{-1} x) \|_{\sigma_j} 
\leq CM
\quad \ \forall |j| \leq k, \ \ x \in \R^n
\end{equation}
by \eqref{def gj gj tilde}, \eqref{ok g tilde.1}. 
By \eqref{def gj gj tilde} one has 
\begin{align}
R_j(x,y;g) 
& := g^{(j)}(x) - \sum_{|j+\ell| \leq k} \frac{1}{\ell!} g^{(j+\ell)}(y) (x-y)^\ell
\notag \\ 
& = \g^{-|j|} \tilde g^{(j)}(\g^{-1} x) 
- \sum_{|j+\ell| \leq k} \frac{1}{\ell!} \g^{-|j|-|\ell|} \tilde g^{(j+\ell)}(\g^{-1} y) 
\g^{|\ell|} (\g^{-1} x - \g^{-1} y)^\ell
\notag \\ 
& = \g^{-|j|} R_j(\g^{-1} x, \g^{-1} y ; \tilde g)
\qquad \forall |j| \leq k, \ \ x,y \in \R^n.
\label{rescal Rj}
\end{align}
Hence, by \eqref{rescal Rj}, \eqref{ok g tilde.1}, 
\begin{align}
\g^\rho \| R_j(x,y;g) \|_{\sigma_\rho} 
& = \g^{\rho-|j|} \| R_j( \g^{-1} x, \g^{-1} y; \tilde g) \|_{\sigma_\rho} 
\notag \\ & 
\leq \g^{\rho-|j|} CM |\g^{-1} x - \g^{-1} y|^{\rho-|j|}
\notag \\ &
= C M |x-y|^{\rho-|j|}
\quad \ \forall |j| \leq k, \ \ x,y \in \R^n.
\label{bic.2}
\end{align}
From \eqref{bic.1}, \eqref{bic.2} and Definition \ref{def Lip sigma}
it follows that $g \in \Lip(\rho,\R^n,\sigma,\g,\d)$
and, taking the inf over all $M$ such that 
\eqref{0903.1-2} holds, 
\begin{equation} \label{bic.4}
\| g \|_{\Lip(\rho,\R^n,\sigma,\g,\d)} \leq 
C \| f \|_{\Lip(\rho,F,\sigma,\g,\d)}.
\end{equation}

By \eqref{bic.0}, $g$ is an extension of $f$ to $\R^n$.  
The regularity properties of $g$ on $\Om$, 
namely $g^{(j)}(x) \in E_\infty$ for all $x \in \Om$,
and $g^{(j)} \in C^\infty(\Om,E_a)$ for all $a \in \mI$, 
follow from the corresponding properties of $\tilde g$
given by Lemma \ref{lemma:summarize.normalized}.
From the construction in subsection \ref{subsec:def g}
and the linearity of the rescaling operator $f \mapsto \tilde f$ in \eqref{def f tilde}
and its inverse $\tilde g \mapsto g$ in \eqref{def gj gj tilde}
we deduce that the extension map $f \mapsto g$ is a linear operator, 
with bounded operator norm $\leq C$ (see \eqref{bic.4}).
The proof of Theorem \ref{thm:WET sigma} is complete.
\end{proof}

\section{Proofs about the dyadic cubes decomposition}
\label{sec:app.cubes}

Following Chapter VI, section 1 of \cite{S}, 
in this appendix we provide detailed references or direct proofs 
for the results stated in subsection \ref{subsec:cubes}.

\begin{definition} \label{def:touch} 
\emph{(Disjoint cubes, cubes that touch)}. 
We say that two cubes (see Definition \ref{def:cubes})
are \emph{disjoint} 
if their interiors are disjoint. 
We say that two disjoint cubes \emph{touch} 
if their boundaries have a common point.
\end{definition}

\begin{theorem} \label{thm:Stein.167}
\emph{(\cite{S}, Theorem 1 on page 167)}.
Let $F \subset \R^n$ be a closed set, and let $\Om = \R^n \setminus F$.
There exists a collection $\mF = \{ Q_1, Q_2, \ldots, Q_i, \ldots \}$ 
of mutually disjoint cubes such that 
\[
\bigcup_{i=1}^\infty Q_i = \Om,
\qquad  
\diam(Q_i) \leq \dist(Q_i,F) \leq 4 \diam (Q_i) \quad \ \forall i=1,2,\ldots
\]
\end{theorem}

From Theorem \ref{thm:Stein.167} one has directly property \eqref{dist.QF.q}
of Proposition \ref{prop:cubes}.

By the construction in \cite{S}, 
every cube in the collection $\mF$ 
is the product of $n$ intervals of the form 
\begin{equation} \label{dyadic.form}
Q = I_1 \times \ldots \times I_n, 
\quad \ 
I_j = [2^{-k} a_j, 2^{-k} (a_j + 1)], 
\quad \ j = 1, \ldots, n; 
\qquad a_j,k \in \Z. 
\end{equation}
The length of every edge of the cube $Q$ in \eqref{dyadic.form} 
is $2^{-k}$, and its diameter is $2^{-k} \sqrt{n}$. 
This implies that $q_i > 0$ for all $i$, as stated in Proposition \ref{prop:cubes}. 

The expanded cubes $Q_i^*$ are defined in Definition \ref{def:Q*} 
with expansion factor $\lm = 1+\e$, 
$\e := \frac{1}{8 \sqrt{n}}$. 
Thus $\diam (Q_i^*) = \lm \diam(Q_i)$, 
and \eqref{diam.Q*.q} in Proposition \ref{prop:cubes} follows immediately.  

The other properties in Proposition \ref{prop:cubes} 
are stated slightly differently than in \cite{S}, 
because, both for simplicity and by the special role  
played by $q_i$ in this paper (see \eqref{def theta}), 
we have chosen $q_i$ as a unique reference 
in the estimates of Proposition \ref{prop:cubes}.
Hence now we prove the rest of Proposition \ref{prop:cubes}, 
and we start with the following observation. 

\begin{lemma} \label{lemma:servizio.6}
For every $y \in Q_i^*$ there exists $x \in Q_i$ such that 
$|x-y| \leq (\e/2) q_i$.
\end{lemma}

\begin{proof}
Let $Q_i$ be the cube 
$\{ p + v : v \in [-r,r]^n \}$. 
Then the expanded cube $Q_i^*$ 
is the set $\{ p + (1+\e)v : v \in [-r,r]^n \}$. 
Let $y \in Q_i^*$. 
Hence $y = p + (1+\e)v$ for some $v \in [-r,r]^n$. 
Let $x := p + v$. Then $x \in Q_i$, and 
$|y-x| = \e|v| \leq \e \sqrt{n}\, r 
= (\e/2) q_i$ 
because $q_i = \diam(Q_i) = 2r \sqrt{n}$.
\end{proof}

\begin{proof}[Proof of \eqref{dist.Q*F.q}]
Since $Q_i \subset Q_i^*$, one has $\dist(Q_i^*,F) \leq \dist(Q_i,F)$ 
by definition of distance. Also, $\dist(Q_i,F) \leq 4 q_i$ by \eqref{dist.QF.q},
therefore $\dist(Q_i^*,F) \leq 4 q_i$, which is the second inequality 
in \eqref{dist.Q*F.q}. 

Now take $y \in Q_i^*$ that realizes $\dist(Q_i^*,F) = \dist(y,F)$, 
and let $x \in Q_i$ be the point such that $|x-y| \leq (\e/2) q_i$ 
given by Lemma \ref{lemma:servizio.6}. 
For all $z \in F$ one has 
$|x-z| \leq |x-y| + |y-z|$, and therefore, 
taking the infimum over all $z \in F$, 
$\dist(x,F) \leq |x-y| + \dist(y,F)$. 
From the properties of $x,y$ we deduce that 
$\dist(x,F) \leq (\e/2) q_i + \dist(Q_i^*,F)$. 
Next, $\dist(Q_i,F) \leq \dist(x,F)$ because $x \in Q_i$ (definition of distance), 
and $q_i \leq \dist(Q_i,F)$ by \eqref{dist.QF.q}. 
Therefore $q_i \leq (\e/2) q_i + \dist(Q_i^*,F)$, 
whence the first inequality in \eqref{dist.Q*F.q} follows 
because $1/2 \leq 1 - (\e/2)$.
\end{proof}

\begin{proof}[Proof of the covering identity $\Om = \cup_{i=1}^\infty Q_i^*$]
One has $Q_i \subset Q_i^*$ by definition of expanded cubes, 
and $\Om = \cup_i Q_i$ by Theorem \ref{thm:Stein.167}. 
Hence $\Om = \cup_i Q_i \subseteq \cup_i Q_i^*$. 

On the other hand, by \eqref{dist.Q*F.q}, one has 
$\dist(Q_i^*,F) \geq \frac12 q_i > 0$, 
whence it follows that $Q_i^*$ and $F$ have no common point, 
namely $Q_i^* \subseteq \Om$ for all $i$; 
therefore $\cup_i Q_i^* \subseteq \Om$. 
\end{proof}

\begin{proof}[Proof of \eqref{dist.xF.q}]
Let $x \in Q_i^*$. One has $\dist(Q_i^*,F) \leq \dist(x,F)$ because $x \in Q_i^*$,
and $\frac12 q_i \leq \dist(Q_i^*,F)$ by \eqref{dist.Q*F.q},
whence we get the first inequality in \eqref{dist.xF.q}.

For all $z \in Q_i^*$ one has $|x-z| \leq \diam(Q_i^*)$ because $x,z \in Q_i^*$, 
and $\diam(Q_i^*) \leq 2 q_i$ by \eqref{diam.Q*.q}. 
Hence for all $y \in F$ one has 
$|x-y| \leq |x-z| + |z-y|
\leq 2 q_i + |z-y|$. 
Taking the infimum over all $z \in Q_i^*$, $y \in F$, 
we obtain $\dist(x,F) \leq 2 q_i + \dist(Q_i^*,F)$. 
By \eqref{dist.Q*F.q}, $\dist(Q_i^*,F) \leq 4 q_i$, 
and we get the second inequality in \eqref{dist.xF.q}.
\end{proof}

\begin{proof}[Proof of \eqref{dist.xpi.q}]
Let $p_i$ be as in \eqref{def.pi}, 
and let $x \in Q_i^*$. 
By Lemma \ref{lemma:servizio.6} there exists $y \in Q_i$ 
such that $|x-y| \leq (\e/2) q_i$. 
Let $z \in Q_i$ be a point that realizes $\dist(Q_i,p_i) = |z-p_i|$.  
By triangular inequality, 
$|x-p_i| \leq |x-y| + |y-z| + |z-p_i|$. 
Now $|x-y| \leq (\e/2) q_i \leq q_i$;
next, $|y-z| \leq \diam(Q_i) = q_i$ because both $y,z \in Q_i$;
next, $|z-p_i| = \dist(Q_i,p_i) = \dist(Q_i,F)$ by \eqref{def.pi}, 
and $\dist(Q_i,F) \leq 4q_i$ by \eqref{dist.QF.q}. 
The sum of the three terms then gives $|x-p_i| \leq q_i + q_i + 4 q_i$, 
which is the second inequality in \eqref{dist.xpi.q}. 

One has $\frac12 q_i \leq \dist(Q_i^*,F)$ by \eqref{dist.Q*F.q},
and $\dist(Q_i^*,F) \leq |x-p_i|$ because $x \in Q_i^*$, $p_i \in F$
(definition of distance). This immediately implies the first inequality 
 in \eqref{dist.xpi.q}. 
\end{proof}

\begin{proof}[Proof of \eqref{dist.xy.q}]
Let $x \in Q_i^*$, $y \in F$. 
Then $\dist(Q_i^*,F) \leq |x-y|$ by definition of distance, 
and $\frac12 q_i \leq \dist(Q_i^*,F)$ by \eqref{dist.Q*F.q}. 
Hence $\frac12 q_i \leq |x-y|$. 
\end{proof}

To complete the proof of Proposition \ref{prop:cubes}, 
it remains to prove that every point $x \in \Om$ admits a neighborhood 
$B_x \subset \Om$ intersecting at most $12^n$ expanded cubes $Q_i^*$. 
This is the content of Proposition 3 on page 169 of \cite{S} and of its short proof; 
we give here a slightly different, more detailed proof. 
Let us begin with another result from \cite{S}. 

\begin{prop} \label{prop:Prop.2.Stein.169}
\emph{(\cite{S}, Proposition 2 on page 169)}.
Suppose $Q \in \mF$. 
Then there are at most $12^n$ cubes in $\mF$ which touch $Q$.
\end{prop}

\begin{lemma} \label{lemma:servizio.do not touch}
If $Q_i, Q_j \in \mF$ do not touch, then 
$\dist(Q_i,Q_j) 
\geq \frac{1}{\sqrt{n}} \, \min \{ q_i, q_j \}$.
\end{lemma}

\begin{proof}
Since $Q_i, Q_j \in \mF$, they are of the form \eqref{dyadic.form}, 
namely $Q_i = I_1 \times \ldots \times I_n$, 
$Q_j = J_1 \times \ldots \times J_n$, 
with
\[
I_m = [2^{-k} a_m, 2^{-k} (a_m + 1)], 
\quad \ 
J_m = [2^{-h} b_m, 2^{-h} (b_m + 1)],
\quad \ 
m=1,\ldots,n,
\]
for some $a_m, b_m, k, h \in \Z$; 
we suppose that $k \geq h$. 
A vector $x = (x_1, \ldots, x_n) \in \R^n$ belongs to $Q_i \cap Q_j$ 
if and only if 
$x_m \in I_m \cap J_m$ for all $m=1,\ldots,n$. 
This means that $Q_i, Q_j$ intersect iff $I_m, J_m$ intersect for all $m$. 
By assumption, $Q_i, Q_j$ do not intersect: 
hence $I_m \cap J_m$ is empty for some $m$. 
Since $I_m, J_m$ are intervals of the real line,
this implies that $\max I_m < \min J_m$ or vice versa. 
Therefore, multiplying by $2^k$,  
\begin{equation} \label{4.ago.1}
2^{k-h} b_m - (a_m + 1) > 0
\quad \text{or} \quad 
a_m - 2^{k-h} (b_m+1)  > 0.
\end{equation}
Now $a_m, b_m, 2^{k-h}$ are all integers (because $k \geq h$). 
Then the quantities in \eqref{4.ago.1} are positive integers, 
and hence 
``$> 0$'' in \eqref{4.ago.1} can be replaced by ``$\geq 1$''. 
Dividing by $2^k$, this gives 
$2^{-h} b_m \geq 2^{-k} (a_m+1) + 2^{-k}$
or 
$2^{-k} a_m \geq 2^{-h} (b_m+1) + 2^{-k}$, 
namely
\[
\min (J_m) \geq \max (I_m) + 2^{-k} 
\quad \text{or} \quad 
\min (I_m) \geq \max (J_m) + 2^{-k}. 
\]
Then $|t-s| \geq 2^{-k}$ for all $t \in I_m$, $s \in J_m$.  
Given any $x \in Q_i$, $y \in Q_j$, 
their $m$-th coordinates satisfy $x_m \in I_m$, $y_m \in J_m$, 
and therefore $|x-y| \geq |x_m-y_m| \geq 2^{-k}$. 
Passing to the infimum over all $x \in Q_i$, $y \in Q_j$, 
we get $\dist(Q_i,Q_j) \geq 2^{-k}$. 
Note that $2^{-k}$ is the length of the edge of $Q_i$,
$2^{-h}$ the one of $Q_j$, 
and $2^{-k}$ is the minimum of the two lenghts.
Finally the edge of a cube is equal to its diameter divided by $\sqrt{n}$.
\end{proof}

\begin{lemma} \label{lemma:servizio.5}
Let $Q_i, Q_j \in \mF$, with $q_j \leq q_i$.
Let $x_i \in Q_i$, $x_j \in Q_j$, 
with $|x_i - x_j| \leq \e q_i$. 
Then the cubes $Q_i , Q_j$ touch.
\end{lemma}

\begin{proof}
Suppose that $Q_i, Q_j$ do not touch. 
Then, by Lemma \ref{lemma:servizio.do not touch}, 
$\dist(Q_i,Q_j) \geq \frac{1}{\sqrt{n}} q_j$. 
Moreover $\dist(Q_i,Q_j) \leq |x_i - x_j|$ because $x_i \in Q_i$, $x_j \in Q_j$. 
Hence 
\begin{equation} \label{notte.3}
\frac{1}{\sqrt{n}}\, q_j 
\leq \dist(Q_i,Q_j) 
\leq |x_i - x_j|
\leq \e q_i.
\end{equation}
For all $z \in F$ one has 
$|x_j-z| \geq |x_i-z| - |x_i - x_j|  
\geq |x_i - z| - \e q_i$, 
and, taking the infimum over all $z \in F$, 
\[
\dist(x_j,F) \geq \dist(x_i,F) - \e q_i.
\]
Now $\dist(x_i,F) \geq \dist(Q_i,F)$ because $x_i \in Q_i$, 
and $\dist(Q_i,F) \geq q_i$ by \eqref{dist.QF.q}. 
Also, $6 q_j \geq \dist(x_j,F)$ by \eqref{dist.xF.q}, 
because $x_j \in Q_j \subset Q_j^*$.
Thus 
\begin{equation} \label{notte.4}
6 q_j \geq (1 - \e) q_i.
\end{equation}
Since $q_i > 0$, \eqref{notte.3}, \eqref{notte.4} imply that 
$1-\e \leq 6 \e \sqrt{n}$,  
namely $1+6\sqrt{n} \geq \frac{1}{\e}$. 
By Definition \ref{def:Q*}, 
$\frac{1}{\e} = 8 \sqrt{n}$, which gives a contradiction.
\end{proof}

\begin{lemma} \label{lemma:servizio.**}
If $Q_i^*$ intersects $Q_j^*$, then $Q_i$ touches $Q_j$. 
\end{lemma}

\begin{proof}
Let $Q_i^* \cap Q_j^*$ be nonempty. 
Then there exists $y \in Q_i^* \cap Q_j^*$. 
Since $y \in Q_i^*$, by Lemma \ref{lemma:servizio.6} 
there exists $x_i \in Q_i$ such that 
$|y-x_i| \leq (\e/2) q_i$. 
Similarly, since $y \in Q_j^*$, there exists 
$x_j \in Q_j$ such that 
$|y-x_j| \leq (\e/2) q_j$. 
Suppose $q_j \leq q_i$. Hence
\[
|x_i - x_j| \leq |x_i-y| + |y-x_j| 
\leq (\e/2) (q_i + q_j)
\leq \e q_i.
\]
Thus $x_i \in Q_i$, $x_j \in Q_j$, and $|x_i - x_j| \leq \e q_i$. 
Then, by Lemma \ref{lemma:servizio.5}, $Q_i$ and $Q_j$ touch.
\end{proof}

\begin{lemma} \label{lemma:good}
For each $x \in \Om$ there exists a neighborhood $B_x$ of $x$, 
contained in $\Om$, 
that intersects at most $12^n$ expanded cubes $Q_i^*$, 
with $Q_i \in \mF$. 
\end{lemma}

\begin{proof}
Let $x \in \Om$. Then $x \in Q_j$ for some $Q_j \in \mF$, 
because, by Theorem \ref{thm:Stein.167},  $\Om = \cup_i Q_i$. 
Consider the open ball $B_x := B(x,\d)$ of center $x$ 
and radius $\d$ sufficiently small to have $B_x \subset Q_j^*$:
since $x \in Q_j$, a radius equal to $\e$ times the half-length of the edge of $Q_j$
is enough, so we choose $\d = \e q_j \frac{1}{2 \sqrt{n}}$. 
If $B_x$ intersects some $Q_i^*$, then $Q_j^*$ intersects $Q_i^*$ 
(because $B_x \subset Q_j^*$). 
Hence, by Lemma \ref{lemma:servizio.**}, $Q_i$ touches $Q_j$. 
The number of cubes of $\mF$ that touch $Q_j$ is at most $12^n$
by Proposition \ref{prop:Prop.2.Stein.169}.
\end{proof}

By Lemma \ref{lemma:good}, the proof of Proposition \ref{prop:cubes} 
is now complete.

\begin{proof}[Proof of Proposition \ref{prop:partition of unity}] 
The partition of unity satisfying \eqref{POU} 
is constructed on page 170 of \cite{S}. 
The identity in \eqref{POU.der} is obtained 
by differentiating the identity $\sum \ph_i^*(x) = 1$,  
taking into account that the series is locally finite: 
by Lemma \ref{lemma:good}, around every point $x$ there is an open set $B_x$ 
such that $x \in B_x \subset \Om$ and only $N$ expanded cubes $Q_i^*$, 
with $N \leq 12^n$, intersect $B_x$; 
therefore there are only $N$ functions $\ph_i^*$ that are not identically zero on $B_x$. 
Hence there is no convergence problem in differentiating the series.  
The inequality in \eqref{POU.der} is equation (13) on page 174 of \cite{S}.
\end{proof}

\begin{proof}[Proof of Lemma \ref{lemma:cubi vicini}]
Suppose that $x \in Q_i^*$ for some $i \notin \mN$. 
By \eqref{dist.xF.q}, $q_i \leq 2 \dist(x,F)$, 
and, by assumption, $2 \dist(x,F) \leq 1$. 
Hence $q_i \leq 1$, namely $i \in \mN$, a contradiction. 
This proves that $x \notin Q_i^*$ for all $i \notin \mN$.
Therefore, by \eqref{POU}, 
$\ph_i^*(x) = 0$ for all $i \notin \mN$. 
Thus only the terms with $i \in \mN$ remain in the sum. 

To make the partial derivative of the sum, 
consider that $\{ x \in \Om : \dist(x,F) < 1/2 \}$ 
is an open set, and that the sum is locally finite 
(Proposition \ref{prop:cubes}).
\end{proof}

The proof of the results stated in subsection \ref{subsec:cubes} 
is complete.

\section{Examples of scales of Banach spaces} 
\label{sec:list}

In this section we give a non exhaustive list 
of well-known examples of scales of Banach spaces. 

In subsection \ref{subsec:Sobolev} we consider $L^2$-based Sobolev spaces $H^s$, 
on $\R^d$ and on $\T^d$, where the scale is parametrized by the continuous (real) parameter $s$; 
properties \eqref{0303.1}, \ldots, \eqref{S2} are immediate.

In subsection \ref{subsec:Ck} we deal with classes $C^k_b$ 
of continuously differentiable functions with bounded derivatives,  
on $\R^d$ and on $\T^d$, where the scale is described by the discrete (integer) parameter $k$; 
we give self-contained proofs of \eqref{0303.1}, \ldots, \eqref{S2} 
(this is classical material). 

H\"older spaces are treated more briefly in subsection \ref{subsec:Holder}, 
following H\"ormander \cite{Geodesy} and Zehnder \cite{Zehnder}. 
Few other interesting examples are mentioned in subsection \ref{subsec:other.scales}. 

In subsection \ref{subsec:Lebesgue} 
we consider Lebesgue spaces of sequences 
$\ell^p$ and of functions $L^p(\Om)$, where $\Om$ has finite measure. 
They are scales of Banach spaces 
satisfying \eqref{0303.1} and also \eqref{interpolazione}, 
for which the families $(S_\theta)$ of linear operators 
that could seem to be natural candidates to be smoothing operators
satisfy \eqref{S.generic}, \eqref{S1}, but not \eqref{S2}. 
These observations, although elementary, seem to be hard to find in literature,
and they are maybe new. 

In fact, for $\ell^p$ spaces we prove more: in Theorem \ref{thm:ell.p}
we prove that there does not exist any family of linear operators 
satisfying all the properties 
\eqref{S.generic}, \eqref{S1}, \eqref{S2}.
To the best of our knowledge, Theorem \ref{thm:ell.p} is new.


\subsection{Sobolev spaces $H^s$}
\label{subsec:Sobolev}

\begin{ex} (\emph{$H^s(\R^d,\C)$ with Fourier truncation $\la \xi \ra \leq \theta$}).
\label{ex:1}
On $\R^d$, with $s \in \R$, consider the Sobolev space
\begin{align}
H^s(\R^d,\C) & := \big\{ u : \R^d \to \C : \| u \|_{H^s(\R^d,\C)} < \infty \big\}, 
\notag \\ 
\| u \|_{H^s(\R^d,\C)} & := 
\Big( \int_{\R^d} |\hat u(\xi)|^2 \la \xi \ra^{2s} \, d\xi \Big)^{\frac12},
\quad \ \la \xi \ra := (1 + |\xi|^2)^{\frac12},
\label{def Jap bra}
\end{align}
where $\hat u$ is the Fourier transform of $u$. 
For any real number $a_0 \in \R$, 
let 
\[
\mI := [a_0, \infty), \quad \ 
E_a := H^a(\R^d,\C), \quad \ 
\| u \|_a := \| u \|_{H^a(\R^d,\C)}.
\] 
Then the family $(E_a, \| \ \|_a)_{a \in \mI}$ satisfies \eqref{0303.1}.
For every real $\theta \geq 1$, let 
\[
(S_\theta u)(x) := \frac{1}{(2\pi)^d} 
\int_{ \la \xi \ra \leq \theta} \hat u(\xi) e^{i\xi \cdot x} \, d\xi.
\]
In other words, $S_\theta$ is the Fourier multiplier 
$\hat u(\xi) \mapsto \chi_{\theta}(\xi) \hat u(\xi)$
where $\chi_\theta(\xi) = 1$ if $\la \xi \ra \leq \theta$ 
and zero otherwise.  
It is immediate to check that \eqref{S1}-\eqref{S2} 
are satisfied with $A_{ab} = B_{ab} = 1$ for all $a,b \in \mI$, $a \leq b$.
\end{ex}

\begin{ex} (\emph{$H^s(\R^d,\C)$ with Fourier truncation $|\xi| \leq \theta$}).
\label{ex:2}
With $E_a, \| \ \|_a$ like in Example \ref{ex:1}, define $S_\theta$ as the Fourier 
multiplier $\hat u(\xi) \mapsto \chi_{\theta}(\xi) \hat u(\xi)$
where $\chi_\theta(\xi) = 1$ if $|\xi| \leq \theta$ 
and zero otherwise. 
Then $\| S_\theta u \|_b^2 \leq (1 + \theta^2)^{b-a} \| u \|_a^2 
\leq (2 \theta^2)^{b-a}$ for $a \leq b$, and \eqref{S1} holds 
with $A_{ab} = 2^{\frac{b-a}{2}}$; \eqref{S2} holds with $B_{ab} = 1$.
\end{ex}

\begin{ex} (\emph{$H^s(\T^d,\C)$ with Fourier truncation $\la k \ra \leq \theta$}).
\label{ex:3}
For $s \in \R$, consider the Sobolev space of periodic functions
\begin{align*}
H^s(\T^d,\C) & := \big\{ u : \T^d \to \C : \| u \|_{H^s(\T^d,\C)} < \infty \big\}, 
\quad \ \T := \R / 2 \pi \Z, 
\\ 
\| u \|_{H^s(\T^d,\C)} & := 
\Big( \sum_{k \in \Z^d} |\hat u_k|^2 \la k \ra^{2s} \Big)^{\frac12}
\end{align*}
where $\hat u_k$ are the Fourier coefficients of $u$, 
and $\la \  \ra$ 
is defined in \eqref{def Jap bra}.
For any $a_0 \in \R$, let 
$\mI := [a_0, \infty)$, 
$E_a := H^a(\T^d,\C)$, 
$\| u \|_a := \| u \|_{H^a(\T^d,\C)}$.
Then the family $(E_a, \| \ \|_a)_{a \in \mI}$ satisfies \eqref{0303.1}.
For every real $\theta \geq 1$, let 
\[
(S_\theta u)(x) := 
\sum_{\begin{subarray}{c} k \in \Z^d \\ \la k \ra \leq \theta \end{subarray}} 
\hat u_k e^{i k \cdot x}.
\]
Then \eqref{S1}-\eqref{S2} hold with $A_{ab} = B_{ab} = 1$ for all $a,b \in \mI$, $a \leq b$.
The fact that $\theta$ is a ``continuous parameter'' 
(namely $\theta$ varies in the interval $\mI$) 
and $\la k \ra$ is a ``discrete'' one (because $k$ varies in $\Z^d$) 
is not a problem in checking the validity of \eqref{S1}-\eqref{S2}.
\end{ex}

\begin{ex} (\emph{$H^s(\T^d,\C)$ with Fourier truncation $|k| \leq \theta$}).
\label{ex:4}
With $E_a, \| \ \|_a$ like in Example \ref{ex:3}, 
define $S_\theta$ as the Fourier truncation $|k| \leq \theta$. 
Then $\| S_\theta u \|_b^2 \leq (1 + \theta^2)^{b-a} \| u \|_a^2 
\leq (2 \theta^2)^{b-a}$ for $a \leq b$, and \eqref{S1} holds 
with $A_{ab} = 2^{\frac{b-a}{2}}$; \eqref{S2} holds with $B_{ab} = 1$.
\end{ex}

\subsection{Spaces $C^k$}
\label{subsec:Ck}

\begin{ex} (\emph{$C^k_b(\R^d,\C)$, $k$ integer}).
\label{ex:5}
For $k \geq 0$ integer, consider the set 
$C^k_b(\R^d,\C)$ of all bounded, $k$ times differentiable functions 
with continuous bounded derivatives, with norm 
\[
\| u \|_{C^k(\R^d,\C)} 
= \max_{\begin{subarray}{c} \a \in \N^d \\ |\a| \leq k \end{subarray}} 
\ \sup_{x \in \R^d} |\pa_x^\a u(x)|.
\]
Let 
\[
a_0 = 0, \quad \ 
\mI = \{ 0,1,2,\ldots \} = \N, \quad \ 
E_a = C^a_b(\R^d,\C), \quad \ 
\| u \|_a = \| u \|_{C^a(\R^d,\C)}.
\]
Then the family $(E_a, \| \ \|_a)_{a \in \mI}$ satisfies \eqref{0303.1}.

We consider smoothing operators $S_\theta$ defined as convolution operators 
(or smooth Fourier cut-off) in the following, classical way. 
Fix a real, even function $\sigma \in C^\infty(\R^d,\R)$, 
vanishing for $|\xi| \geq 1$, 
such that $\s =1$ in the ball $|\xi| \leq 1/2$.  
Define $\psi$ as the Fourier anti-transform of $\sigma$, namely 
\[
\psi(x) := \frac{1}{(2\pi)^d} \int_{\R^d} \sigma(\xi) e^{i\xi \cdot x} \, d\xi, 
\quad \ 
\hat \psi(\xi) := \int_{\R^d} \psi(y) e^{-i \xi \cdot y} \, dy 
= \sigma(\xi), 
\quad \ x,\xi \in \R^d.
\]
Thus $\psi$ is real, even, and belongs to the Schwartz class $\mS(\R^d,\R)$ 
(because $\s \in \mS(\R^d,\R))$. 
For every real $\theta \geq 1$ we define 
\begin{equation} \label{def.JMZ}
\psi_\theta(x) := \theta^d \psi(\theta x), \quad \ 
(S_\theta u)(x) := (u \ast \psi_\theta)(x) 
= \int_{\R^d} u(y) \psi_\theta(x-y) \, dy.
\end{equation}
With standard calculations one has 
\begin{equation} \label{rescaled.F.transform}
\widehat{\psi_\theta} (\xi) = \hat \psi(\theta^{-1} \xi) = \s (\theta^{-1} \xi), 
\quad \ 
\widehat{(S_\theta u)}(\xi) = \hat u(\xi) \widehat{\psi_\theta}(\xi)
= \hat u(\xi) \s (\theta^{-1} \xi),
\end{equation}
so that the smoothing operator $S_\theta$ 
is the Fourier multiplier of symbol $\s(\theta^{-1} \xi)$; 
since $\s(\theta^{-1} \xi) = 1$ for $|\xi| \leq \theta / 2$
and $\s(\theta^{-1} \xi) = 0$ for $|\xi| \geq \theta$, 
$S_\theta$ is in fact 
a smooth version of the crude Fourier truncation of Example \ref{ex:2}. 
We also have
\begin{align} 
\label{int.psi}
\int_{\R^d} \psi_\theta(x) \, dx 
& = 1, 
\qquad  
\int_{\R^d} x^\a \psi_\theta(x) \, dx = 0 
\quad \forall \a \in \N^d, \ \a \neq 0,
\\
\int_{\R^d} |\pa_x^\a \psi_\theta(x)| \, dx 
& = \theta^{|\a|} \| \pa_x^\a \psi \|_{L^1(\R^d)} 
\quad \ \forall \a \in \N^d,
\label{der.psi} \\ 
\int_{\R^d} |x|^p |\psi_\theta(x)| \, dx 
& = \theta^{-p} C_p, \quad \ C_p := \int_{\R^d} |x|^p |\psi(x)| \, dx, 
\quad \ \forall p \in \R, \ p \geq 0.
\label{molt.psi}
\end{align} 
To prove \eqref{int.psi}, first consider $\theta = 1$: 
for every $\a \in \N^d$ 
the function $h(x) := x^\a \psi(x)$ is in $L^1(\R^d)$
(because $\psi \in \mS(\R^d)$), 
and one has $\hat h(\xi) = i^{|\a|} \pa_\xi^\a \hat \psi(\xi)
= i^{|\a|} \pa_\xi^\a \s(\xi)$. 
Therefore $\psi$ satisfies \eqref{int.psi} 
because $\sigma(0) = 1$ and $\pa_\xi^\a \sigma(0) = 0$ for all $\a \neq 0$. 
For $\theta > 1$, \eqref{int.psi} is obtained by the change of variable
$\theta x = y$ in the integral. 
With the same change of variable one also gets 
\eqref{der.psi}, \eqref{molt.psi}.

For all multi-indices $\a,\b \in \N^d$, one has 
$\pa_x^{\a+\b} (u \ast \psi_\theta) = (\pa_x^\a u) \ast (\pa_x^{\b} \psi_\theta)$, 
and therefore, by \eqref{der.psi}, 
\begin{equation} \label{alpha+beta.der}
|\pa_x^{\a+\b} S_\theta u (x)| 
\leq \big( \sup_{y \in \R^d} |\pa_x^\a u(y)| \big) 
\int_{\R^d} |\pa_x^{\b} \psi_\theta(z)| \, dz
\leq \| u \|_{C^{|\a|}(\R^d)} \theta^{|\b|} \| \pa_x^{\b} \psi \|_{L^1(\R^d)}.
\end{equation}
Given $a,b \in \mI$, $a \leq b$, 
for all multi-indices $\ell$ such that $a \leq |\ell| \leq b$ 
we take $\a,\b \in \N^d$ such that $|\a|=a$, $\a+\b = \ell$, 
and use \eqref{alpha+beta.der}; 
for $|\ell| \leq a$ use \eqref{alpha+beta.der} with $\a = \ell$, $\b = 0$. 
Thus we obtain \eqref{S1} with 
$A_{ab} = \max \{ \| \pa_x^\ell \psi \|_{L^1(\R^d)} : |\ell| \leq b-a \}$.

To get \eqref{S2}, we use \eqref{int.psi}, \eqref{molt.psi} and Taylor's expansion. 
Let $a,b \in \mI$, $a \leq b$. 
If $a=b$, then \eqref{S2} follows from \eqref{S1}
by triangular inequality. 
Thus assume that $a < b$.
Let $\a \in \N^d$, $|\a| = m \leq a$. 
We expand $\pa_x^\a u(y)$ around $x$, 
\begin{equation} \label{expand.Ck}
\pa_x^\a u(y) 
= \sum_{|\a + \ell| \leq b-1} \frac{1}{\ell!} \pa_x^{\a+\ell} u(x) (y-x)^\ell + R_\a(y,x),
\end{equation}
where the remainder satisfies 
\begin{equation} \label{est.R.Ck}
|R_\a(y,x)| \leq C_{d,b} \| u \|_{C^b(\R^d)} |x-y|^{b-m}
\end{equation}
for some constant $C_{d,b}$ depending only on $d,b$.
Estimate \eqref{est.R.Ck} can be obtained by the mean value theorem, 
for example imitating the proof of Lemma \ref{lemma:segment.small}.
By \eqref{expand.Ck}, \eqref{int.psi},
\[
\pa_x^\a S_\theta u(x) 
= (\pa_x^\a u) \ast \psi_\theta(x)
= \pa_x^\a u(x) + \int_{\R^d} R_\a(y,x) \psi_\theta(x-y) \, dy.
\]
By \eqref{est.R.Ck}, \eqref{molt.psi}, 
\begin{align*}
\int_{\R^d} |R(y,x)| |\psi_\theta(x-y)| \, dy 
& \leq C_{d,b} \| u \|_{C^b(\R^d)} \int_{\R^d} |x-y|^{b-m} |\psi_\theta(x-y)| \, dy 
\\ & 
\leq C_{d,b} \| u \|_{C^b(\R^d)} \theta^{-(b-m)} C_{b-m}
\quad \ \forall x \in \R^d.
\end{align*}
Moreover $\theta^{-(b-m)} \leq \theta^{-(b-a)}$ 
for all multi-index $\a$ of length $|\a| = m \leq a$,
all $\theta \geq 1$. 
Hence $\| u - S_\theta u \|_{C^a(\R^d)} 
\leq C \| u \|_{C^b(\R^d)} \theta^{-(b-a)}$ for some $C$ depending on $d,b,a$, 
namely \eqref{S2}.
\end{ex}

\begin{ex} (\emph{$C^k(\T^d,\C)$, $k$ integer}).
\label{ex:6}
Let $C^k(\T^d,\C)$ be the set of functions $u \in C^k(\R^d,\C)$ 
that are $2\pi$-periodic in each variable. 
If $u : \R^d \to \C$ is periodic, 
then the function $S_\theta u$ defined in \eqref{def.JMZ} 
is also periodic, because 
\[
(u \ast \psi_\theta)(x + 2 \pi m) 
= \int_{\R^d} u(x+2 \pi m - y) \psi_\theta(y) \, dy 
= \int_{\R^d} u(x - y) \psi_\theta(y) \, dy 
= (u \ast \psi_\theta)(x)
\] 
for all $x \in \R^d$, $m \in \Z^d$.
Hence Example \ref{ex:5} includes the periodic setting. 

Note that $S_\theta$ acts on periodic functions 
as the Fourier coefficients multiplier of symbol $\s(\theta^{-1} k)$, namely 
\begin{equation} \label{1409.1}
u(x) = \sum_{k \in \Z^d} \hat u_k e^{ik \cdot x} 
\ \mapsto \ 
(S_\theta u)(x) = \sum_{k \in \Z^d} \hat u_k \s(\theta^{-1} k) e^{ik \cdot x}, 
\end{equation}
as it can be easily deduced from the definition \eqref{def.JMZ} of $S_\theta u(x)$,
replacing $u(x-y)$ in the integral with its Fourier series, 
and recalling that $\int_{\R^d} \psi_\theta(y) e^{-ik \cdot y} \, dy = \s(\theta^{-1} k)$,
see \eqref{rescaled.F.transform}. 

The same holds for functions having period $L_i$ 
in the variable $x_i$, with possibly different periods $L_i \neq L_j$ 
in different directions.
\end{ex}

\subsection{H\"older spaces $\mH^a$}
\label{subsec:Holder}

\begin{ex} (\emph{H\"older spaces $\mH^a(M)$ on a compact $C^\infty$ manifold
with boundary, from H\oe{}rmander} \cite{Geodesy}).
\label{ex:8}
Let $B \subset \R^d$ be a fixed convex compact set with nonempty interior.  
For $k < a \leq k+1$, where $k$ is an integer $\geq 0$, 
let $u \in C^k(B,\C)$ 
(namely $u = v|_B$ is the restriction to $B$ 
of a function $v \in C^k(\Om,\C)$ where $\Om$ is some open neighborhood of $B$),
let 
\begin{equation} \label{def.Holder}
|u|_{\mH^a(B)} := \sum_{|\a| = k} \ 
\sup_{\begin{subarray}{c} x,y \in B \\ x \neq y \end{subarray}} 
\frac{|\pa_x^\a u(x) - \pa_x^\a u(y)|}{|x-y|^a}, 
\quad \ 
\| u \|_{\mH^a(B)} := |u|_{\mH^a(B)} + \sup_{x \in B} |u(x)|,
\end{equation}
and define (Definition A.3 in Appendix A of \cite{Geodesy}) 
the H\oe{}lder space $\mH^a(B)$ 
as the set of all $u \in C^k(B)$ 
with finite norm $\| u \|_{\mH^a(B)}$. 
For $a=0$ we set $\mH^0(B) := C(B)$ and $\| u \|_{\mH^0(B)} := \sup |u|$.

Let $\mI := [0,\infty)$, $E_a := \mH^a(B)$, $\| \ \|_a := \| \ \|_{\mH^a(B)}$. 
Then (Theorem A.5 in \cite{Geodesy}) 
$(E_a, \| \ \|_a)_{a \in \mI}$ satisfies \eqref{0303.1}. 

Moreover (page 42 of \cite{Geodesy}) $\mH^a(M)$ 
can be defined if $M$ is any compact $C^\infty$ manifold with boundary.
To do so, one covers $M$ by coordinate patches $M_j$ 
and takes a partition of unity $\sum \chi_j = 1$ with $\chi_j \in C^\infty_0(M_j)$. 
A function $u$ on $M$ is then said to be in $\mH^a(M)$ 
if $\chi_j u$ for every $j$ is in $\mH^a$ as a function of the local coordinates, 
and $\| u \|_{\mH^a(M)}$ is defined as $\sum \| \chi_j u \|_{\mH^a}$ 
with the terms defined by means of local coordinates. 
The definition of $\mH^a(M)$ does not depend on the choice of covering, local coordinates
or partition of unity, and the norm is well defined up to equivalences. 

Now let $K$ be a compact set in $\R^d$ 
and choose $\chi \in C^\infty_0(\R^d)$ 
such that $\chi = 1$ in a neighborhood of $K$. 
Define $\s, \psi, \psi_\theta$ like in Example \ref{ex:5}, and set 
$(S_\theta u)(x) := \chi(x) (u \ast \psi_\theta)(x)$
for functions $u$ with support in $K$. 
Then (see Theorem A.10 of \cite{Geodesy}) 
the smoothing inequalities \eqref{S1}-\eqref{S2} hold 
for $u$ supported in $K$.  
Moreover (Remark on page 44 of \cite{Geodesy}) for the spaces $\mH^a(M)$, 
with $M$ compact manifold, one decomposes $u$ by the partition of unity 
and defines $S_\theta u (x) := \chi_j(x) (u \ast \psi_\theta)(x)$ in each coordinate patch. 
\end{ex}

\begin{ex}(\emph{H\"older spaces $\mH^a(\T^d)$, from Zehnder} \cite{Zehnder}).
\label{ex:9}
Let $\mI = [0,\infty)$. 
For $a \geq 0$ integer, define $(E_a, \| \ \|_a)$ as the 
space $C^a(\T^d,\C)$ of $a$ times continuously differentiable functions 
of $\R^d$ that are periodic in each argument, with its usual norm defined 
in Examples \ref{ex:5}-\ref{ex:6}.
For $a \notin \N$, define $(E_a, \| \ \|_a)$ 
as in Example \ref{ex:8}, namely by \eqref{def.Holder} with $\R^d$ in place of $B$. 
Let $S_\theta$ be the convolution operator of Example \ref{ex:6}
(which is the one of Example \ref{ex:5}, applied to periodic functions).
Then (Lemma 6.2.4 of \cite{Zehnder}) \eqref{S1}-\eqref{S2} hold.  
\end{ex}

On $\T^d$, the difference between Examples \ref{ex:8} and \ref{ex:9}
is in the definition of $E_a$ for $a \geq 1$ integer: 
in Example \ref{ex:9} the derivatives of order $a-1$ 
of the functions in $E_a$ are of class $C^1$, 
while in Example \ref{ex:8} they are just Lipschitz
(like in Definition \ref{def Lip Stein} when $\rho=k+1$). 
Except for $a=1,2,\ldots$, the Banach spaces in Examples \ref{ex:8}
and \ref{ex:9} coincide; removing the positive integer values of $a$ 
one obtains another scale.

\begin{ex}(\emph{H\"older spaces $\mH^a(\T^d)$ with noninteger exponent}).
\label{ex:20}
Let $\mI = [0,1) \cup (1,2) \cup \ldots$,
namely $[0,\infty)$ without the positive integers.
Then the definitions of $E_a$ in Examples \ref{ex:8} and \ref{ex:9} agree,  
and this gives another scale with smoothing. 
\end{ex}

\subsection{Other scales}
\label{subsec:other.scales}

In the next three examples 
we briefly sketch few other cases of scales with smoothings. 

\begin{ex} 
\label{ex:13} 
(\emph{Sobolev, Besov, Triebel-Lizorkin spaces}).
With tools from Fourier analysis like 
Littlewood-Paley decomposition, 
convolution estimates, 
and Bernstein inequalities, 
one could also deal with $L^p$-based Sobolev spaces $W^{s,p}(\R^d,\C)$,
Besov spaces $B^s_{p,q}(\R^d,\C)$, 
and Triebel-Lizorkin spaces $F^s_{p,q}(\R^d,\C)$, 
where the ``amount of derivatives'' $s$ 
is the parameter to move to obtain the scale $(E_a)$, 
keeping the summability powers $p,q$ fixed. 
If $u = \sum \Delta_j u$ is a Littlewood-Paley decomposition 
of a function $u$, then one defines
$S_\theta u$ as the partial sum over all $j$ such that $2^j \leq \theta$. 
As observed in Example \ref{ex:3}, the fact that 
$j$ is a discrete parameter ($j$ is an integer) 
and $\theta$ is a continuous one ($\theta \geq 1$ is real)
should have no technical consequences.

We remark that using one of the summability parameters $p,q$ 
of Besov and Triebel-Lizorkin spaces to parametrize the scale 
is very likely to produce a scale of Banach spaces 
that does not admit any family of smoothing operators: 
see subsection \ref{subsec:Lebesgue}, 
and especially Theorem \ref{thm:ell.p}.
\end{ex}

\begin{ex} 
\label{ex:14}
(\emph{Functions with polynomial decay}).
Let $\mI = [0,\infty)$, let $Y$ be a Banach space. 
For $a \in \mI$, for any function $u : \R^d \to Y$, let 
\[
\| u \|_a := \sup_{x \in \R^d} (1+|x|)^a \| u(x) \|_Y,
\]
and let $E_a$ be the space of functions with finite norm. 
For $\theta \geq 1$, let 
\[
(S_\theta u)(x) := u(x) \quad \text{if $1+|x| \leq \theta$}; 
\qquad 
(S_\theta u)(x) := 0 \quad \text{if $1+|x| > \theta$}.
\]
One immediately verifies that \eqref{0303.1}, \ldots, 
\eqref{S2} hold, with $A_{ab}=B_{ab}=1$.  

This example is essentially the same as Example \ref{ex:1}
but with decay in ``space'', or ``time'', or another ``physical variable'' $x$ 
instead of Fourier frequency.

Possible variants of this example can include partial derivatives like 
$\sup (1+|x|)^a \| \pa_x^m u(x) \|_Y$.
\end{ex}

\begin{ex} 
\label{ex:14.1}
(\emph{Functions with values in a scale}).
Let $a_0 \in \R$, $\mI \subseteq [a_0, \infty)$, with $a_0 = \min \mI$. 
Let $(F_a, \| \ \|_{F_a})$, $a \in \mI$, be a scale of Banach spaces
with smoothing operators $(S_\theta)$, $\theta \geq 1$, satisfying 
\eqref{0303.1}, \ldots, \eqref{S2}.  
Let $T>0$, let $E_a := C([0,T], F_a)$ be the set of continuous functions 
$u : [0,T] \to F_a$, with norm
\[
\| u \|_a := \sup_{t \in [0,T]} \| u(t) \|_{F_a}.
\]
For $u \in E_{a_0}$, define $S_\theta u$ by 
$(S_\theta u)(t) := S_\theta [ u(t) ]$ for all $t \in [0,T]$. 
It is immediate to verify that \eqref{0303.1}, \ldots, \eqref{S2} hold 
for the scale $(E_a)$ (without changing the constants $A_{ab}, B_{ab}$).  

Possible variants of this example can be obtained 
by replacing $[0,T]$ with other domains $\Om \subseteq \R^d$, 
or by replacing $C([0,T],F_a)$ with $L^\infty([0,T], F_a)$, 
or by including one derivative, 
like $\sup_t \| u(t) \|_{F_a} + \sup_t \| \pa_t u(t) \|_{F_a}$, 
or more derivatives,
or by including derivatives with decreasing regularity norms, 
like $\sup_t \| u(t) \|_{F_a} + \sup_t \| \pa_t u(t) \|_{F_{a-\d}}$, 
etc.

This example and its many possible variants are based on the same 
basic observations about the ``inherited structure'' 
in Proposition \ref{prop:about.Lip}.
\end{ex}

\subsection{Lebesgue spaces}
\label{subsec:Lebesgue}

In this subsection we consider Lebesgue spaces of sequences and of functions.
They satisfy \eqref{0303.1} and \eqref{interpolazione},
and admit operators satisfying \eqref{S.generic}, \eqref{S1}, but not \eqref{S2}.
For the $\ell^p$ spaces we also prove (Theorem \ref{thm:ell.p}) 
the non-existence of families of smoothing operators 
satisfying all \eqref{S.generic}, \eqref{S1}, \eqref{S2}.
The observations about $(S_\theta)$ 
in Examples \ref{ex:15}, \ref{ex:16} 
and Theorem \ref{thm:ell.p}
seem to be new.

\begin{ex} 
\label{ex:15}
(\emph{Lebesgue space of sequences $\ell^p$}). 
Let $\N_1 := \{ 1,2,\ldots \}$ be the set of positive integers.
For every $p \in [1,\infty]$, let $\ell^p = \ell^p(\N_1,\C)$ 
be the set of sequences $x = (x_1, x_2, \ldots)$
of complex numbers with finite norm $\| x \|_{\ell^p}$, 
where 
\[
\| x \|_{\ell^p} := \Big( \sum_{k=1}^\infty |x_k|^p \Big)^{\frac{1}{p}} 
\quad \ \text{for } p \in [1,\infty); 
\qquad 
\| x \|_{\ell^\infty} := \sup_{k \in \N_1} |x_k|.
\]
Let $\mI := [0,1]$, $a_0 := 0$.
For every $a \in \mI$, let 
\[
E_a := \ell^p, \quad \ 
\| x \|_a := \| x \|_{\ell^p}, \quad \ 
p := \frac{1}{a}\,, 
\]
with the convention $\frac10 := \infty$.
Then $(E_a)_{a \in \mI}$ is a scale of Banach spaces satisfying \eqref{0303.1}.
To prove it, let $a,b \in \mI$, with $a \leq b$, 
and let $p = \frac{1}{a}$, $q = \frac{1}{b}$.
Let $x \in E_b = \ell^q$ with $\| x \|_{\ell^q} = 1$.  
Then $|x_k| \leq 1$ for all $k$, 
therefore $|x_k|^p \leq |x_k|^q$ for all $k$
(because $|x_k| \leq 1$ and $p \geq q$).
Hence $\sum |x_k|^p \leq \sum |x_k|^q = \| x \|_{\ell^q}^q = 1$, 
and $\| x \|_{\ell^p} \leq 1$.   
Now consider any $x \in \ell^q$, $x \neq 0$, 
define $y := x / \| x \|_{\ell^q}$, 
and apply to $y$ the inequality already proved 
for vectors of unitary $\ell^q$ norm. 

From H\oe{}lder's inequality one obtains directly 
the interpolation property \eqref{interpolazione}
(without using \eqref{S1}-\eqref{S2}): 
let $a,b,c \in \mI$, with $a \leq b \leq c$, and let 
$p_a := \frac{1}{a}$, 
$p_b := \frac{1}{b}$, 
$p_c := \frac{1}{c}$. 
Let $b = \th a + (1-\th) c$ for some $\th \in [0,1]$. 
Then 
\begin{equation} \label{interpolaz.ppp}
\frac{1}{p_b} = \frac{\th}{p_a} + \frac{1-\th}{p_c}.
\end{equation}
Multiplying by $p_b$, 
\begin{equation} \label{interpolaz.lm.mu}
1 = \frac{\th p_b}{p_a} + \frac{(1-\th)p_b}{p_c}
= \frac{1}{\lm} + \frac{1}{\mu},
\qquad \ 
\lm := \frac{p_a}{\th p_b}, \quad \ 
\mu := \frac{p_c}{(1-\th)p_b}.
\end{equation}
Write $|x_k|^{p_b}$ as the product 
$|x_k|^{\th p_b} |x_k|^{(1-\th) p_b}$; 
by H\oe{}lder's inequality, 
\begin{align*}
\| x \|_b 
& = \Big( \sum_{k=1}^\infty |x_k|^{\th p_b} |x_k|^{(1-\th) p_b} \Big)^{\frac{1}{p_b}} 
\leq \Big( \sum_{k=1}^\infty |x_k|^{\th p_b \lm} \Big)^{\frac{1}{\lm p_b}} 
\Big( \sum_{k=1}^\infty |x_k|^{(1-\th) p_b \mu} \Big)^{\frac{1}{\mu p_b}} 
= \| x \|_a^\th \| x \|_c^{1-\th},
\end{align*}
which is \eqref{interpolazione}.

We prove below (Theorem \ref{thm:ell.p}) 
that $\ell^p$ spaces do not have smoothing operators;  
before proving the general result, 
it is instructive to make an attempt to construct smoothing operators, 
and to see what goes wrong.  
The attempt, inspired to the previous examples, 
is to define $S_\theta$ as a truncation operator, 
namely as the pointwise product 
$(S_\theta x)_k = w_{\theta,k} x_k$ with 
$w_{\theta,k} := 1$ for $k \leq \theta$ 
and $w_{\theta,k} := 0$ for $k > \theta$. 
These operators map $E_0 = \ell^\infty$ into $E_1 = \ell^1$, 
so that \eqref{S.generic} is satisfied.
Also \eqref{S1} holds, with constants $A_{ab} = 1$, 
because, by H\oe{}lder's inequality, one has
\[
\| S_\theta x \|_{\ell^q} 
= \Big( \sum |w_{\theta,k}|^{q} |x_k|^{q} \Big)^{\frac{1}{q}}
\leq \Big( \sum |w_{\theta,k}|^{q \lm} \Big)^{\frac{1}{q \lm}} 
 \Big( \sum |x_k|^{q \mu} \Big)^{\frac{1}{q \mu}}
\leq \theta^{b-a} \| x \|_{\ell^p}
\]
where $a,b \in [0,1]$, 
$a \leq b$, 
$p = \frac{1}{a}$, 
$q = \frac{1}{b}$, 
$\mu = \frac{b}{a}$ 
(so that $q\mu=p$)
and $\frac{1}{\lm} + \frac{1}{\mu} = 1$
(so that $\frac{1}{q\lm} = b-a$).

However, \eqref{S2} is violated.  
To show it, assume that \eqref{S2} holds. 
Given $\theta \geq 1$, take an integer $k > \theta$, 
and consider the sequence $x = e_k$ with $x_k = 1$ and $x_j = 0$ for all $j \neq k$. 
Then $S_\theta x = 0$, $\| x \|_{\ell^p} = 1$ for all $p \in [1,\infty]$, 
therefore, by \eqref{S2}, $1 \leq B_{ab} \theta^{-(b-a)}$.  
This holds for all $\theta \geq 1$; for $a<b$, this is a contradiction. 
\end{ex}

Example \ref{ex:15} shows that the truncation operators 
satisfy \eqref{S.generic}, \eqref{S1}, but not \eqref{S2}. 
Now we prove that this is not limited to truncations, but it is a general fact. 

\begin{theorem} 
\label{thm:ell.p}
\emph{(The scale of Lebesgue spaces $\ell^p$ does not admit any family of smoothing operators)}.
Consider the scale defined in Example \ref{ex:15}. 
Then there does not exist any family $(S_\theta)$ of linear operators 
satisfying all \eqref{S.generic}, \eqref{S1}, \eqref{S2}. 

More is true: fix any two real numbers $a,b \in [0,1]$ with $a < b$, and define 
$p := \frac{1}{a}$, $q := \frac{1}{b}$. 
Let $(S_\theta)$, $\theta \geq 1$, 
be a family of linear operators $S_\theta : \ell^p \to \ell^q$
such that 
\begin{equation} \label{S1.thm.ell.p}
\| S_\theta x \|_{\ell^q} \leq A_{ab} \theta^{b-a} \| x \|_{\ell^p} 
\quad \ \forall x \in \ell^p, \ \ \forall \theta\geq 1,
\end{equation}
for some constant $A_{ab} > 0$ independent of $x,\theta$. 
Then there does not exists any constant $B_{ab} > 0$ such that 
\begin{equation} \label{S2.thm.ell.p}
\| x - S_\theta x \|_{\ell^p} \leq B_{ab} \theta^{-(b-a)} \| x \|_{\ell^q} 
\quad \ \forall x \in \ell^q, \ \ \forall \theta \geq 1,
\end{equation}
$B_{ab}$ independent of $x,\theta$. 
\end{theorem}

The second part of Theorem \ref{thm:ell.p} implies the first one, 
and it says that $\ell^p$ spaces 
have no smoothing operators even if we consider a scale where $\mI$ 
is any subset of $[0,1]$ containing at least two distinct elements.
To prove Theorem \ref{thm:ell.p}, we begin with two lemmas. 

\begin{lemma} \label{lemma:UN}
Let $U := \{ 1, -1 \}$, and let $N \geq 1$ be an integer. 
Consider the set $U^N$ of vectors $x = (x_1, \ldots, x_N)$ 
with components $x_n \in U$.
For every $k,n \in \{ 1, \ldots, N \}$ one has 
\[
\sum_{x \in U^N} x_k x_n = 2^N \d_{kn},
\]
where $\d_{kn} = 1$ if $k=n$ and $\d_{kn} = 0$ if $k \neq n$. 
\end{lemma}

\begin{proof}
Let $k,n \in \{ 1, \ldots, N \}$, with $k=n$.
Then $\sum_{x \in U^N} x_k^2 = \sum_{x \in U^N} 1$, 
which is the cardinality of $U^N$, namely $2^N$. 
Now let $k,n \in \{ 1, \ldots, N \}$, with $k \neq n$, 
and consider first the case $N=2$: 
then $k=1$ and $n=2$ or vice versa, and 
\[
\sum_{x = (x_1, x_2) \in U^2} x_1 x_2 
= \Big( \sum_{x_1 \in U} x_1 \Big) \Big( \sum_{x_2 \in U} x_2 \Big) 
= 0.
\]
Finally, let $k,n \in \{ 1, \ldots, N \}$, with $k \neq n$, and $N \geq 3$. 
For every $x \in U^N$ let $x' \in U^{N-2}$ be the vector $x$ 
without its components $x_k, x_n$. Then 
\[
\sum_{x \in U^N} x_k x_n 
= \sum_{x' \in U^{N-2}} \Big( \sum_{(x_k, x_n) \in U^2} x_k x_n \Big) = 0.
\qedhere
\]
\end{proof}

\begin{lemma} \label{lemma:wow.matrix}
Let $N \geq 1$ be an integer, 
and let $(c_{kn})$ be an $N \times N$ matrix 
with entries $c_{kn} \in \C$, $k,n = 1, \ldots, N$. 
Let $1 \leq q < \infty$, $R \geq 0$. 
Assume that 
\begin{equation} \label{hyp.magic}
\Big( \sum_{k = 1}^N \Big| \sum_{n=1}^N c_{kn} x_n \Big|^q \Big)^{\frac{1}{q}} \leq R 
\qquad 
\forall x_1, \ldots, x_N \in \{ 1, -1 \}.
\end{equation}
Then 
\[
\Big( \sum_{k=1}^N |c_{kk}|^q \Big)^{\frac{1}{q}}  \leq R.
\]
Also, 
if $\sup_{k=1, \ldots, N} \big| \sum_{n=1}^N c_{kn} x_n \big| \leq R$
for all $x_1, \ldots, x_N \in \{ 1, -1 \}$, 
then  $\sup_{k=1, \ldots, N} |c_{kk}| \leq R$.
\end{lemma}

\begin{proof}
Let $k \in \{ 1, \ldots, N \}$. By Lemma \ref{lemma:UN}, 
\[
c_{kk} = \sum_{n=1}^N c_{kn} \d_{kn} 
= \sum_{n=1}^N c_{kn} 2^{-N} \sum_{x \in U^N} x_k x_n 
= 2^{-N} \sum_{x \in U^N} x_k \sum_{n=1}^N c_{kn} x_n.
\]
By triangular inequality, since $|x_k|=1$,
\begin{equation} \label{ckk.1}
|c_{kk}| 
\leq 2^{-N} \sum_{x \in U^N} \Big| \sum_{n=1}^N c_{kn} x_n \Big|.
\end{equation}
If $1 < q < \infty$, then, by H\"older's inequality, 
\[
\sum_{x \in U^N} \Big| \sum_{n=1}^N c_{kn} x_n \Big|
\leq 
\Big( \sum_{x \in U^N} \Big| \sum_{n=1}^N c_{kn} x_n \Big|^q \Big)^{\frac{1}{q}}
\Big( \sum_{x \in U^N} 1 \Big)^{\frac{1}{q'}} 
= \Big( \sum_{x \in U^N} \Big| \sum_{n=1}^N c_{kn} x_n \Big|^q \Big)^{\frac{1}{q}} 
2^{N(1 - \frac{1}{q})},
\]
where $\frac{1}{q} + \frac{1}{q'} = 1$. Therefore, by \eqref{ckk.1}, 
\begin{equation} \label{ckk.2}
|c_{kk}| 
\leq 2^{-\frac{N}{q}} 
\Big( \sum_{x \in U^N} \Big| \sum_{n=1}^N c_{kn} x_n \Big|^q \Big)^{\frac{1}{q}}.
\end{equation}
Note that \eqref{ckk.2} also holds for $q=1$, in which case it is the same as \eqref{ckk.1}.  
Taking the $q$-th power of \eqref{ckk.2}, 
summing over $k=1, \ldots, N$,
and using assumption \eqref{hyp.magic} gives
\[
\sum_{k=1}^N |c_{kk}|^q 
\leq 2^{-N} \sum_{x \in U^N} \sum_{k=1}^N \Big| \sum_{n=1}^N c_{kn} x_n \Big|^q
\leq 2^{-N} \sum_{x \in U^N} R^q 
= R^q,
\]
which is the thesis. 
The last line of the statement follows from \eqref{ckk.1} 
by taking the sup over $k=1, \ldots, N$.
\end{proof}

Now we prove Theorem \ref{thm:ell.p}, using finite approximations 
of the infinite matrix representing $S_\theta$. 
Recall the notation $\N_1 = \{1, 2, \ldots \}$. 

\begin{proof}[Proof of Theorem \ref{thm:ell.p}]
Let $e_1 = (1, 0, 0, \ldots)$, 
$e_2 = (0,1,0, \ldots)$, etc., 
namely 
$(e_n)_k = \delta_{nk}$ for all $k,n \in \N_1$. 
For every $\theta \geq 1$, the operator $S_\theta$ maps $\ell^p$ into $\ell^q$, 
therefore, for every $n \in \N_1$, 
$S_\theta e_n$ is an $\ell^q$ sequence of complex numbers, say
\begin{equation} \label{def ckn}
S_\theta e_n = ( c_{1 n}(\theta), c_{2 n}(\theta), c_{3 n}(\theta), \ldots),
\end{equation}
where $c_{kn}(\theta)$ is the $k$-th element of the sequence $S_\theta e_n$. 
Let $N \in \N_1$ and 
\begin{equation} \label{x finito}
x := (x_1, \ldots, x_N, 0, 0, \ldots) = \sum_{n=1}^N x_n e_n, 
\quad \ x_1, \ldots, x_N \in U := \{ 1, -1 \}. 
\end{equation}
Since $S_\theta$ is linear, 
by \eqref{def ckn}, \eqref{x finito} one has 
\[
S_\theta x 
= S_\theta \Big( \sum_{n=1}^N x_n e_n \Big)
= \sum_{n=1}^N x_n S_\theta e_n 
= y = (y_1, y_2, \ldots), 
\qquad 
y_k := \sum_{n=1}^N c_{kn}(\theta) x_n 
\quad \ \forall k \in \N_1.
\]
Then
\[
\Big( \sum_{k=1}^N \Big|\sum_{n=1}^N c_{kn}(\theta) x_n \Big|^q \Big)^{\frac{1}{q}}
= \Big( \sum_{k=1}^N |y_k|^q \Big)^{\frac{1}{q}}
\leq 
\Big( \sum_{k=1}^\infty |y_k|^q \Big)^{\frac{1}{q}}
= \| S_\theta x \|_{\ell^q}
\]
and, by \eqref{S1.thm.ell.p}, 
\[
\| S_\theta x \|_{\ell^q} 
\leq A_{ab} \theta^{b-a} \| x \|_{\ell^p} 
= A_{ab} \theta^{b-a} N^{\frac{1}{p}} 
\]
for all $x_1, \ldots, x_N \in U$, all $\theta \geq 1$. 
By Lemma \ref{lemma:wow.matrix} applied with $R = A_{ab} \theta^{b-a} N^{\frac{1}{p}}$
we get
\begin{equation} \label{contra.S1}
\Big(\sum_{k=1}^N |c_{kk}(\theta)|^q \Big)^{\frac{1}{q}} 
\leq A_{ab} \theta^{b-a} N^{\frac{1}{p}}.
\end{equation}

Now assume, by contradiction, that there exists $B_{ab} > 0$ such that 
\eqref{S2.thm.ell.p} holds. 
For every $k \in \N_1$, $\theta \geq 1$, one has 
\[
|1 - c_{kk}(\theta)| 
= | (e_k - S_\theta e_k)_k |
\leq \Big( \sum_{n=1}^\infty |(e_k - S_\theta e_k)_n|^p \Big)^{\frac{1}{p}} 
= \| e_k - S_\theta e_k \|_{\ell^p} 
\]
and, by \eqref{S2.thm.ell.p},
\[
\| e_k - S_\theta e_k \|_{\ell^p} 
\leq B_{ab} \theta^{-(b-a)} \| e_k \|_{\ell^q}
= B_{ab} \theta^{-(b-a)}.
\]
Since $b-a>0$, there exists $\theta_* \geq 1$ such that 
$B_{ab} \theta_*^{-(b-a)} \leq \frac12$. 
Hence $|1 - c_{kk}(\theta_*)| \leq \frac12$ for all $k \in \N_1$.  
As a consequence, $|c_{kk}(\theta_*)| \geq \frac12$ for all $k \in \N_1$, 
and 
\begin{equation} \label{contra.S2}
\Big(\sum_{k=1}^N |c_{kk}(\theta_*)|^q \Big)^{\frac{1}{q}} 
\geq \Big(\sum_{k=1}^N \frac{1}{2^q} \Big)^{\frac{1}{q}} 
= \frac12 N^{\frac{1}{q}}.
\end{equation}
From \eqref{contra.S1}, \eqref{contra.S2} it follows that
\[
\frac12 N^{\frac{1}{q}} \leq A_{ab} \theta_*^{b-a} N^{\frac{1}{p}},
\]
namely $N^{\frac{1}{q} - \frac{1}{p}} = N^{b-a} \leq C$, where $C := 2 A_{ab} \theta_*^{b-a}$ does not depend on $N$. 
For $N \to \infty$ this gives a contradiction, 
and the theorem is proved.
\end{proof}

\begin{ex} 
\label{ex:16}
(\emph{Lebesgue space $L^p(\Om,\C)$ with $|\Om| < \infty$}).
Let $\Om \subset \R^d$ be a measurable set of positive, finite Lebesgue measure $|\Om|$. 
For $p \in [1,\infty]$, let $L^p(\Om,\C)$ be the Lebesgue space 
of complex-valued, measurable functions with finite norm 
\[
\| u \|_{L^p(\Om)} = \Big( \int_\Om |u(x)|^p \, dx \Big)^{\frac{1}{p}}
\quad \ \text{for } p \in [1,\infty); 
\quad \ 
\| u \|_{L^\infty(\Om)} = \mathrm{ess} \sup_{\Om} |u|.
\]
Let $a_0 := 0$, $\mI := [0,d]$. 
For $a \in \mI$, let 
\[
E_a := L^p(\Om,\C), \quad \ 
\| u \|_a := |\Om|^{-\frac{1}{p}} \| u \|_{L^p(\Om)}, \quad \ 
p := \frac{d}{d-a}\,, 
\]
with $p = \infty$ for $a=d$. 
By H\oe{}lder's inequality, one has 
\[
\| u \|_a \leq |\Om|^{-\frac{1}{p}} \Big( \int_\Om |u|^{p \lm} \Big)^{\frac{1}{p\lm}} 
\Big( \int_\Om 1^\mu \Big)^{\frac{1}{p \mu}} = \| u \|_b
\]
for all $a,b \in \mI$, $a \leq b$, where 
$p = \frac{d}{d-a}$, 
$q = \frac{d}{d-b}$, 
$\lm = \frac{q}{p}$, 
and $\frac{1}{\lm} + \frac{1}{\mu} = 1$. 
Thus $(E_a)_{a \in \mI}$ is a scale of Banach spaces 
satisfying \eqref{0303.1}. 

The interpolation property \eqref{interpolazione} follows directly 
from H\oe{}lder's inequality (without using \eqref{S1}-\eqref{S2}): 
let $a,b,c \in \mI$, with $a \leq b \leq c$, and let 
$p_a := \frac{d}{d-a}$, 
$p_b := \frac{d}{d-b}$, 
$p_c := \frac{d}{d-c}$. 
Let $b = \th a + (1-\th) c$ for some $\th \in [0,1]$. 
Then (despite the different definition of $p_a, p_b, p_c$ with respect to Example \ref{ex:15})
$p_a, p_b, p_c$ satisfy 
\eqref{interpolaz.ppp}, \eqref{interpolaz.lm.mu}.
Write $|u|^{p_b}$ as the product 
$|u|^{\th p_b} |u|^{(1-\th) p_b}$; 
by H\oe{}lder's inequality, 
\begin{align*}
\| u \|_b 
& = |\Om|^{- \frac{1}{p_b}} 
\Big( \int_\Om |u|^{\th p_b} |u|^{(1-\th) p_b} \, dx \Big)^{\frac{1}{p_b}} 
\\
& \leq |\Om|^{- \frac{\th}{p_a} - \frac{1-\th}{p_c}} 
\Big( \int_\Om |u|^{\th p_b \lm} \, dx \Big)^{\frac{1}{\lm p_b}} 
\Big( \int_\Om |u|^{(1-\th) p_b \mu} \, dx \Big)^{\frac{1}{\mu p_b}} 
= \| u \|_a^\th \| u \|_c^{1-\th},
\end{align*}
which is \eqref{interpolazione}.

Like in Example \ref{ex:15}, the difficulty is in the construction 
of the smoothing operators. 
An attempt is the following. Given $u \in E_0 = L^1(\Om)$, 
extend it trivially to $\R^d$ by setting $u=0$ outside $\Om$, 
then define $S_\theta u := (u \ast \psi_\theta)_{|\Om}$, 
namely the restriction to $\Om$ of the convolution $u \ast \psi_\theta$,
with $\psi_\theta$ defined in Example \ref{ex:5}. 

The set $E_\infty := \cap_{a \in [0,d]} E_a$ is the space $E_d = L^\infty(\Om)$. 
Since $S_\theta$ maps $E_0 = L^1(\Om)$ into $E_\infty = E_d = L^\infty(\Om)$, 
property \eqref{S.generic} is satisfied.  

The smoothing property \eqref{S1} is also satisfied. 
To prove it, let $a,b \in \mI$, $a \leq b$, 
and let $p = \frac{d}{d-a}$, $q = \frac{d}{d-b}$. 
By Young's convolution inequality, one has 
$\| u \ast \psi_\theta \|_{L^q} 
\leq \| u \|_{L^p} \| \psi_\theta \|_{L^r}$
where $r$ satisfies 
$\frac{1}{p} + \frac{1}{r} = 1 + \frac{1}{q}$.  
By rescaling, one has $\| \psi_\theta \|_{L^r} 
= \theta^{d - \frac{d}{r}} \| \psi \|_{L^r}$,
and $d(1-\frac{1}{r}) = b-a$, whence \eqref{S1} follows. 

However, in general, \eqref{S2} does not hold. 
To show it, we consider the case in which $\Om$ is the cube $[0,2\pi]^d$,
so that Fourier series can be used, 
and we mimic the argument of Example \ref{ex:15}.
Assume that \eqref{S2} holds. 
Given $\theta \geq 1$, take an integer vector $k \in \Z^d$ with $|k| > \theta$, 
and consider the function $u(x) = e^{i k \cdot x}$. 
Then, recalling \eqref{1409.1}, one has $S_\theta u = 0$, 
and, by direct calculation, $\| u \|_a = 1$ for all $a \in [0,d]$ 
(just because $|u(x)| = 1$ for all $x$). 
Therefore, by \eqref{S2}, $1 \leq B_{ab} \theta^{-(b-a)}$;  
this holds for all $\theta \geq 1$, 
and hence, for $a<b$, we have a contradiction. 
\end{ex}

More on scales of Banach spaces can be found
in the article \cite{KP} of Krein and Petunin, 
and in the book \cite{Caps} of Caps.

\begin{footnotesize}

\end{footnotesize}

\bigskip

Pietro Baldi

\smallskip

Dipartimento di Matematica e Applicazioni ``R. Caccioppoli''

Universit\`a di Napoli Federico II  

Via Cintia, 80126 Napoli, Italy

\smallskip

\texttt{pietro.baldi@unina.it}

\end{document}